\documentclass[reqno,10pt,centertags]{amsart}
\usepackage{amsmath,amsthm,amscd,amssymb,latexsym,esint,upref,stmaryrd,
	enumerate,color,verbatim,yfonts}
\usepackage{hyperref}
\newcommand*{\mailto}[1]{\href{mailto:#1}{\nolinkurl{#1}}}

\newcommand{\R}{{\mathbb R}}
\newcommand{\N}{{\mathbb N}}
\newcommand{\Z}{{\mathbb Z}}
\newcommand{\C}{{\mathbb C}}

\newcommand{\bbC}{{\mathbb{C}}}

\newcommand{\bbN}{{\mathbb{N}}}

\newcommand{\bbR}{{\mathbb{R}}}

\newcommand{\bbZ}{{\mathbb{Z}}}

\newcommand{\cB}{{\mathcal B}}

\newcommand{\cD}{{\mathcal D}}

\newcommand{\cF}{{\mathcal F}}

\newcommand{\cH}{{\mathcal H}}

\newcommand{\cK}{{\mathcal K}}

\newcommand{\cS}{{\mathcal S}}

\newcommand{\cV}{{\mathcal V}}
\newcommand{\cW}{{\mathcal W}}

\newcommand{\beq}{\begin{equation}}
	\newcommand{\enq}{\end{equation}}

\DeclareMathOperator{\supp}{supp}

\DeclareMathOperator{\dom}{dom}

\DeclareMathOperator*{\slim}{s-lim}

\renewcommand{\Re}{\text{\rm Re}}

\newcommand{\no}{\notag}
\newcommand{\lb}{\label}
\newcommand{\f}{\frac}

\newcommand{\ol}{\overline}

\newcommand{\wti}{\widetilde}

\newcommand{\hatt}{\widehat}
\newcommand{\dott}{\,\cdot\,}

\renewcommand{\dot}{\overset{\textbf{\Large.}}}

\newcommand{\bi}{\bibitem}

\let\geq\geqslant
\let\leq\leqslant

\newcommand{\la}{\lambda}

\newcommand{\al}{\alpha}

\newcommand{\ga}{\gamma}

\newcommand{\vt}{\theta}

\makeatletter
\def\theequation{\@arabic\c@equation}

\allowdisplaybreaks
\numberwithin{equation}{section}

\newtheorem{theorem}{Theorem}[section]

\newtheorem{lemma}[theorem]{Lemma}
\newtheorem{corollary}[theorem]{Corollary}
\newtheorem{definition}[theorem]{Definition}
\newtheorem{hypothesis}[theorem]{Hypothesis}
\newtheorem{example}[theorem]{Example}

\theoremstyle{remark}
\newtheorem{remark}[theorem]{Remark}

\begin{document}

	\title[Abstract Left-Definite Theory]{Abstract Left-Definite Theory: A Model Operator Approach, Examples, Fractional Sobolev Spaces, and Interpolation Theory}

	\author[C.\ Fischbacher]{Christoph Fischbacher}
	\address{Department of Mathematics,
		Baylor University, Sid Richardson Bldg., 1410 S.\,4th Street, Waco, TX 76706, USA}
	\email{\mailto{Christoph\_Fischbacher@baylor.edu}}
	\urladdr{\url{http://www.baylor.edu/math/index.php?id=985897}}

	\author[F.\ Gesztesy]{Fritz Gesztesy}
	\address{Department of Mathematics,
		Baylor University, Sid Richardson Bldg., 1410 S.\,4th Street, Waco, TX 76706, USA}
	\email{\mailto{Fritz\_Gesztesy@baylor.edu}}
	\urladdr{\url{http://www.baylor.edu/math/index.php?id=935340}}

	\author[P.\ Hagelstein]{Paul Hagelstein}
	\address{Department of Mathematics,
		Baylor University, Sid Richardson Bldg., 1410 S.\,4th Street, Waco, TX 76706, USA}
	\email{\mailto{Paul\_Hagelstein@baylor.edu}}
	\urladdr{\url{https://www.baylor.edu/math/index.php?id=54007}}

	\author[L.\ L.\ Littlejohn]{Lance L. Littlejohn}
	\address{Department of Mathematics,
		Baylor University, Sid Richardson Bldg., 1410 S.\,4th Street, Waco, TX 76706, USA}
	\email{\mailto{Lance\_Littlejohn@baylor.edu}}
	\urladdr{\url{http://www.baylor.edu/math/index.php?id=53980}}


	\date{\today}
	\thanks{P. H. is partially supported by grant \#521719 from the Simons Foundation.}
		\@namedef{subjclassname@2020}{\textup{2020} Mathematics Subject Classification}

\begin{abstract}
We use a model operator approach and the spectral theorem for self-adjoint operators in a Hilbert space to derive the basic results of abstract left-definite theory in a straightforward manner. The theory is amply illustrated with a variety of concrete examples employing scales of Hilbert spaces, fractional Sobolev spaces, and domains of (strictly) positive fractional powers of operators, employing interpolation theory.

In particular, we explicitly describe the domains of positive powers of the harmonic oscillator operator in $L^2(\bbR)$ \big(and hence that of the Hermite operator in $L^2\big(\bbR; e^{-x^2}dx)\big)$\big) in terms of fractional Sobolev spaces, certain commutation techniques, and positive powers of (the absolute value of) the operator of multiplication by the independent variable in $L^2(\bbR)$.            
\end{abstract}

\maketitle

{\scriptsize{\tableofcontents}}

\section{Introduction and a Brief Historical Overview of General Left-Definite Theory} \lb{s1}

Left-definite theory has its roots in the theory of second-order
Sturm--Liouville differential equations. Indeed, consider the Sturm--Liouville
(SL) generalized eigenvalue equation
\begin{equation}
(\tau f)(x):= r(x)^{-1} \big[(-p(x) f^{\prime}(x))^{\prime}+q(x) f(x)\big]
=\lambda f(x) \, \text{ for a.e.~$x \in (a,b)$,} \lb{1.1}
\end{equation}
with $(a,b) \subseteq \bbR$. Here, we make the
usual assumptions that $1/p,$ $q,$ $r$ are real-valued (Lebesgue) a.e.~on $(a,b)$, locally integrable
on $(a,b)$, with $p>0$ a.e.~on $(a,b)$ and $f \in AC_{loc}((a,b))$ with $f^{[1]} := p f' \in AC_{loc}((a,b))$. The two obvious sesquilinear forms in the study of \eqref{1.1} are
\begin{equation}
\int_{a}^{b} r(x) dx \, \ol{f(x)} g(x) \, \text{ and } \, \int_{a}^{b} dx \, \big[p(x) \ol{f^{\prime}(x)}
g^{\prime}(x) +q(x) \ol{f(x)} g(x)\big].
\end{equation}
If $r>0$ a.e.~on $(a,b)$, this SL problem is called {\it right-definite} since in this case the sesquilinear form
\begin{equation}
\int_{a}^{b} r(x) dx \, \ol{f(x)} g(x), \quad f, g \in L^2((a,b); rdx)
\end{equation}
is positive in the sense that it defines the scalar product in the Hilbert space $L^2((a,b); rdx)$. The problem is called {\it left-definite} if the sequilinear form
\begin{equation}
\int_{a}^{b} dx \, \big[p(x) \ol{f^{\prime}(x)} g^{\prime}(x) + q(x) \ol{f(x)} g(x)\big]
\end{equation}
is positive, that is, it defines a positive scalar product. When a SL problem is both right-definite and left-definite, the sesquilinear forms are connected through {\it Dirichlet's formula }
\begin{align}
\begin{split}
\left(\tau f,g\right)  _{L^{2}((a,b);rdx)} &=\int_{a} r(x) dx \, \ol{(\tau f)(x)} g(x)    \\
&= \int_a^b dx \, \big[- (p(x) \ol{f'(x)})' + q(x) \ol{f(x)}\,\big] g(x)    \\
&= \int_{a}^{b} dx \, \big[p(x)^{-1} \ol{f^{[1]}(x)} g^{[1]}(x) + q(x) \ol{f(x)} g(x)\big]   \\
&= \int_a^b dx \, \big[p(x) \ol{f'(x)} g'(x) + q(x) \ol{f(x)} g(x)\big] :=(f,g)_{1}.
\end{split}
\end{align}
Here $f,g \in AC_{loc}((a,b))$, $p f', p g' \in AC_{loc}((a,b))$ with $f, g$ of compact support in $(a,b)$.  In the literature, the inner product $(\dott,\dott)_{1}$, extended to elements in appropriate function spaces (e.g., Sobolev-type spaces), is called the {\it left-definite inner} product. Littlejohn and
Wellman \cite{LW02}, \cite{LW13} refer to $(\dott.\dott)_{1}$ as the {\it first} left-definite inner
product. Curiously, before the work of Littlejohn and Wellman, there is no
mention of other left-definite inner products in the literature associated
with the Sturm--Liouville expression $\tau$. R. Vonhoff \cite{Vo00} was the first to consider the (first) left-definite analysis of a fourth-order (Legendre type) differential expression. We recommend the texts
of Bennewitz, Brown, and Weikard \cite[Ch.~5]{BBW20} and Zettl \cite[Chs.~5, 12]{Ze05} for in-depth
discussions and an extensive literature of left-definite SL theory.

In \cite{LW02}, Littlejohn and Wellman computed the continuum of left-definite
spaces and operators for the self-adjoint operator $A:\ell^{2}(\bbN) \rightarrow
\ell^{2}(\bbN)$ with compact resolvent, discussed in Example \ref{e3.2a}. In \cite{EKLWY07}, \cite{ELW00}, \cite{ELW02}, \cite{LW02}, the authors apply this general left-definite
theory to several examples, including the self-adjoint,
positively bounded below operators $A$ generated by, respectively, the
classical second-order expressions of Laguerre, Hermite, Legendre (and more
generally, Jacobi), having the same-named orthogonal polynomials as
eigenfunctions. Even though the general theory guarantees a continuum of
left-definite spaces and left-definite operators in each of these examples,
these authors only computed a countably infinite sequence of the spaces
$\{\cH_{n}\}_{n \in \bbN}$ and associated operators $\{A_{n}\}_{n \in \bbN}$.
We further address on this issue below.

One aim of this paper is to present a new model operator approach to left-definite
theory. In \cite{LW02}, \cite{LW13}, the authors started with basic hypotheses (see
(\ref{1.6})), and then employed the Hilbert space spectral theorem to obtain
the main theoretical results of left-definite theory. With our model operator approach,
the spectral theorem is also involved front and center, but now in such a way that the abstract results obtained in \cite{LW02}, \cite{LW13} follow in a straightforward manner.

A second aim of this paper is to employ interpolation theory to describe domains of (strictly) positive fractional powers of the one-dimensional Laplacian on a variety of bounded and unbounded intervals and, in particular, describe the domains of positive fractional powers of the harmonic oscillator and the unitarily equivalent Hermite operator on $\bbR$. 

Let $\cH=(\cV,(\dott,\dott)_{\cH})$ be a Hilbert space
(with underlying vector space $\cV$). Suppose $A:\dom(A)\subset\cH \rightarrow \cH$ is self-adjoint and bounded below in $\cH$ by $c I_{\cH}$, for some $c \in (0,\infty)$ and $I_{\cH}$ the
identity operator in $\cH$; that is to say,
\[
(Af,f)_{\cH} \geq c \|f\|^2_{\cH}, \quad f \in \dom(A).
\]
For $s \in (0,\infty)$, it follows from the spectral theorem for self-adjoint operators in a Hilbert space that $A^{s}:\dom(A^{s})\subset\cH \rightarrow \cH$ is self-adjoint and
bounded below in $\cH$ by $c^{s} I_{\cH}$. Suppose
$\cV_{s}$ is a linear subspace of $\cV$ and $(\dott,\dott)_{s}$ is an inner product on $\cV_{s}\times \cV_{s}.$ In \cite{LW02}, $\cH_{s}=(\cV_{s}$, $(\dott,\dott)_{s})$ is called an $s$-th left-definite space associated with the pair $(\cH,A)$ if each of the following five conditions is satisfied:
\begin{equation}
\begin{array}
[c]{l}
(i) \; \; \text{ $\cH_{s}$ is a Hilbert space.}\\
(ii) \,\, \text{ $\dom(A^{s})$ is a subspace of $\cV_{s}$.} \\
(iii) \text{ $\dom(A^{s})$ is dense in $\cH_{s}$.} \\
(iv) \, \text{ $(f,f)_{s}\geq c^{s}(f,f)_{\cH}$ for all $f \in\cV_{s}$.} \\
(v) \,\,\, \text{ $(A^{s}f,g)_{\cH}=(f,g)_{s}$ for all $f \in\dom(A^{s})$, $g \in\cV_{s}$.}
\end{array}
\lb{1.6}
\end{equation}

\noindent Initially, Littlejohn and Wellman primarily considered the case $s=1,$ the first left-definite setting. However, they realized that the above definition naturally extends to the scale of spaces indexed by $s \in (0,\infty)$.

In \cite{LW02}, \cite{LW13}, Littlejohn and Wellman establish the following theorem by
applying the Hilbert space spectral theorem.

\begin{theorem}
Under the above conditions, with $\cH_{0}=\cH$ and
$(\dott,\dott)_{0}=(\dott,\dott)_{\cH},$ and $s \in (0,\infty)$, the following items $(i)$--$(v)$ hold\,$:$ \\[1mm]
$(i)$ $\cH_{s}=(\cV_{s},(\dott,\dott)_{s})$, where
\begin{equation}
\cV_{s}=\dom\big(A^{s/2}\big) \, \text{ and } \, (f,g)_{s}= \big(A^{s/2}f,A^{s/2}g\big)_{\cH}, \quad
f, g \in \dom\big(A^{s/2}\big),
\end{equation}
satisfy conditions $(i)$--$(v)$ in (\ref{1.6}). Moreover, $\cV_{s}$ and
$(\dott,\dott)_{s}$ are uniquely determined by these conditions. \\[1mm]
$(ii)$ If $A$ is bounded, then $\mathcal{V}_{s}=\mathcal{V}$ for all $s \in (0,\infty)$
and the inner products $(\dott,\dott)_{\cH}$ and $(\dott,\dott)_{s}$
are equivalent. \\[1mm]
$(iii)$ If $A$ is unbounded, then $\cV_{s}$ is a proper subspace of
$\cV;$ moreover, if $0<s<t,$ then $\cV_{t}$ is a proper subspace of $\cV_{s}.$ Moreover, none of the inner products $(\dott,\dott)_{\cH},$ $(\dott,\dott)_{s},$ $(\dott,\dott)_{t}$ are equivalent. \\[1mm]
$(iv)$ For $s \in (0,\infty)$, $\cH_{s}=A^{-s/2}\cH$. \\[1mm]
$(v)$ For $s, t\in [0,\infty)$, the operator $A^{(s-t)/2}:\cH_{s}\rightarrow\cH_{t}$ is an isometric isomorphism of $\cH_{s}$ onto $\cH_{t}$.
\end{theorem}

Of course, these left-definite spaces $\{\cH_{s}\}_{s \in (0,\infty)}$ are an
example of a scale of Hilbert spaces; see, for instance, \cite[Ch.~I]{Be68}, and \cite[Ch.~14]{BSU96}
for a detailed treatment of scales of Hilbert spaces.

It was mentioned above that the authors in \cite{EKLWY07},
\cite{ELW00}, \cite{ELW02}, \cite{LW02} were only able to find a countably infinite
sequence of left-definite spaces $\{\cH_{n}\}_{n=1}^{\infty}$
associated with each of the self-adjoint second-order differential operators
$A$ having the Laguerre, Hermite, or Jacobi polynomials as eigenfunctions. The
reason for this can be seen to be a consequence of part $(v)$ of 
\eqref{1.6}. Indeed, part $(v)$ implies that the $s$-th left-definite space
is generated by $A^{s}.$ Since the form of $A$ in each of these three examples
is the second-order classical Laguerre, Hermite, or Jacobi differential
expression, the $s$-th power of these differential expressions naturally has a sensible meaning when $s$ is a nonnegative integer.

\begin{remark}
We note that Zettl \cite[Chs.~5, 12]{Ze05} has a related, but
different, definition of left-definite theory for SL problems. Indeed, Zettl
allows for the weight function $r$ to change signs. In this respect, in the
study of left-definite SL problems, his work is more general.
\hfill $\diamond$
\end{remark}

With the same assumptions on $A$ in $\cH$ as above (i.e., $A$ self-adjoint in $\cH$, bounded from below by $c I_{\cH}$, $c \in (0,\infty)$), and assuming $A$ to be unbounded for simplicity, suppose
$\{\cH_{s}=(\mathcal{V}_{s},(\dott,\dott)_{s})\}_{s \in (0,\infty)}$ is the
continuum of left-definite spaces associated with the pair $(\cH,A).$
If there exists a self-adjoint operator $A_{s}$ in $\cH_{s}$ satisfying
\begin{equation}
A_{s} f  =A f, \quad f \in \dom(A_{s}) \subset \dom(A),
\end{equation}
we call $A_{s}$ the $s$-th left-definite operator associated with
$(\cH,A).$

Littlejohn and Wellman \cite{LW02}, \cite{LW13} proved the following result:

\begin{theorem} \lb{t1.3}
Let $s \in (0,\infty)$. \\[1mm]
$(i)$ There exists a unique left-definite operator $A_{s}$ in
$\cH_{s}$. Indeed,
\begin{equation}
A_{s}f = A f, \quad f \in\dom(A_{s}) =\cV_{s+2}.
\end{equation}
$(ii)$ The operators $A$ in $\cH$ and $A_{s}$ in $\cH_s$ are unitarily equivalent.
\\[1mm]
$(iii)$ If $\{\phi_{\alpha}\}_{\alpha\in J},$ where $J$ is some indexing
set, is a $($complete\,$)$ set of eigenfunctions of $A$ in $\cH,$ then
$\{\phi_{\alpha}\}_{\alpha\in J}$ is a $($complete\,$)$ set of eigenfunctions of
$A_{s}$ in $\cH_{s}$. \\[1mm]
$(iv)$ For $s, t \in [0,\infty)$, we have the following operator identity
\begin{equation}
A^{(s-t)/2}A_{s}A^{(t-s)/2}=A_{t} \, \text{ valid on } \, \dom(A_{t}).
\end{equation}
$(v)$ $\dom(A_{s}) = A^{-s/2} \dom(A)$.
\end{theorem}

We note that, before the publication of \cite{LW02}, the papers on left-definite
theory made no\ serious mention of a self-adjoint operator $A_{1}$ in the
first left-definite setting $\cH_{1}$. In fact, in the case of the
first left-definite setting for the Legendre self-adjoint
operator $A_{Leg}$ (see \eqref{1.12}), it was believed (see
\cite{Ev80}) that the self-adjoint operator in $\cH_{1}$,
having the Legendre polynomials as eigenfunctions, was not even a differential
operator.

This left-definite theory has proven useful in solving several problems and in obtaining new characterizations of some specific self-adjoint operator domains. We mention a few examples next:

$\bullet$ The self-adjoint operator $A_{Leg}$ in $L^{2}((-1,1))$, generated by
the second-order Legendre differential expression
\begin{equation}
(\tau_{Leg}f)(x)=-\left(  (1-x^{2}) f^{\prime}(x)\right)^{\prime}  + f(x), \quad x\in(-1,1),
\end{equation}
which has the Legendre polynomals $\{P_{n}\}_{n=0}^{\infty}$ as
eigenfunctions, is classically given by
\begin{align}
\begin{split}
& \dom(A_{Leg})= \big\{g \in L^{2}((-1,1)) \, \big| \,  g, g^{\prime} \in AC_{loc}((-1,1)),\\
& \hspace*{2.4cm}  \tau_{Leg} g \in L^{2}((-1,1)); \, \lim_{x\rightarrow\pm1}(1-x^{2}) g^{\prime}(x)=0\big\}.    \lb{1.12}
\end{split}
\end{align}
A complementary representation of $\dom(A_{Leg})$ (see \cite{ELW02}) follows from a left-definite analysis of $A_{Leg}$:
\begin{align}
\begin{split}
& \dom(A_{Leg})=\{g \in L^{2}((-1,1)) \, \big| \,  g, g^{\prime}\in AC_{loc}((-1,1)); \\
& \hspace*{4.5cm} (1-x^{2}) g^{\prime\prime} \in L^{2}((-1,1))\big\}.    \lb{1.13}
\end{split}
\end{align}
One notices that this last representation of $\dom(A_{Leg})$ does
not involve boundary conditions. Similar new representations have been
determined for each of the other classical orthogonal polynomials.

$\bullet$ In 1980, Everitt \cite{Ev80} asked what the domain of the first
left-definite Legendre operator $A_{1, Leg}$ is. From Theorem \ref{t1.3}, this domain equals $\cV_{3}
(A_{Leg})$, the third left-definite space associated with $A_{Leg}$ in $L^{2}((-1,1))$.  It is explicitly given by
\begin{align}
\begin{split}
& \cV_{3}(A_{Leg}) = \big\{g \in L^{2}((-1,1)) \, \big| \,
g, g^{\prime}, g^{\prime\prime}\in AC_{loc}((-1,1));      \lb{1.14} \\
& \hspace*{4.2cm} (1-x^{2})^{3/2} g^{\prime\prime\prime}\in L^{2}((-1,1))\big\}.
\end{split}
\end{align}
In a follow-up paper, Everitt, Mari\'{c}, and Littlejohn \cite{ELM02} established this representation using standard real-analytic techniques.

$\bullet$ In the Legendre context we emphasize that in 1970, Triebel \cite{Tr70}, \cite{Tr70a} (see also \cite[Ch.~7]{Tr78}) studied in detail domains of fractional powers of the Legendre operator. More precisely, Triebel considered generalized Legendre differential operators in $L^2((a,b); dx)$, $a,b \in \bbR$, 
$a < b$, of the form
\begin{align}
\begin{split} 
A_{Leg,m,k} f = (-1)^m \big(q^k f^{(m)}\big)^{(m)}, \quad f \in \dom(A_{Leg,m,k}) = C^{\infty}([a,b]),&   \\
m, k \in \bbN, \; 1 \leq k \leq m,& 
\end{split} 
\end{align}
where  
\begin{equation}
C^{\infty}([a,b]) = \big\{f \colon [a,b] \to \bbC \, \big| \, f \text{ is $C^{\infty}$ on $[a,b]$}\big\}    
\end{equation}
(actually, Triebel \cite{Tr70}, \cite{Tr70a} also considers the more general case where $1 \leq k \leq 2m-1$, but we will restrict our considerations to $1 \leq k \leq m$), and 
\begin{equation} 
0 < q \in C^{\infty}([a,b]) \, \text{ is real-valued} 
\end{equation} 
and satisfying
\begin{equation}
\lim_{x \downarrow a} \f{q(x)}{x-a} = C_a \in (0,\infty), \quad \lim_{x \uparrow b} \f{q(x)}{b-x} = C_b \in (0,\infty).
\end{equation} 
Of course, the special case 
\begin{equation}
m=k=1 \, \text{ and } \, q_0(x) = (b-x)(x-a), \; x \in (a,b),  
\end{equation}
corresponds to the classical Legendre example on the interval $(a,b)$ and from this point on we will restrict ourselves to this special case. In addition to $A_{Leg,1,1}$ one also introduces the minimal operator in $L^2((a,b))$ 
\begin{equation}
A_{Leg,1,1,min} f = - \big(q_0 f^{\prime}\big)^{\prime}, \quad f \in \dom(A_{Leg,1,1,min}) = C_0^{\infty}((a,b)).
\end{equation}
Then one can show that 
\begin{equation}
0 \leq A_{Leg,1,1} \, \text{ is essentially self-adjoint in $L^2((a,b))$},   
\end{equation}
the domain and form domain of the (self-adjoint) closure of $A_{Leg,1,1}$ are given by 
\begin{align}
& \dom\big(\, \ol{A_{Leg,1,1}} \, \big) = H^2_{2,0} \big((a,b); q_0^2;1\big),    \\
& \dom\Big(\big(\, \ol{A_{Leg,1,1}} \, \big)^{1/2}\Big) = H^1_{1,0} \big((a,b); q_0;1\big),
\end{align} 
see \eqref{1.26} below, and the closure of $A_{Leg,1,1}$ coincides with the Friedrichs extension of $A_{Leg,1,1,min}$ in $L^2((a,b))$, that is, 
\begin{equation}
(A_{Leg,1,1,min})_F = \ol{A_{Leg,1,1}}. 
\end{equation}
Regarding domains of (integer or fractional) powers of $\ol{A_{Leg,1,1}}$, Triebel \cite{Tr70}, \cite{Tr70a} obtains
\begin{equation}
\dom\Big(\big( \, \ol{A_{Leg,1,1}} \, \big)^s\Big) = \begin{cases} 
H^{2s}_{2s,0;0} \big((a,b); q_0^{2s};1\big), & s \in (0,1/2],   \\[1mm] 
H^{2s}_{2s,0} \big((a,b); q_0^{2s};1\big), & s \in [1/2,\infty).
\end{cases}
\end{equation}
Here we employ the following notation for the underlying spaces: In the integer case 
$\sigma \in \bbN$, one has the weighted Sobolev spaces 
\begin{align}
& H^{\sigma}_{\kappa_1,\kappa_2} \big((a,b);q^{\kappa_1};q^{\kappa_2}\big) = \bigg\{ g \in \cD((a,b))^{\prime} 
\, \bigg| \, g \in L^2((a,b)), \, g^{(\sigma)} \in L_{loc}^2((a,b));    \no \\
& \quad \|g\|_{H^{\sigma}_{\kappa_1,\kappa_2} ((a,b);q^{\kappa_1};q^{\kappa_2})}^2 :=\int_a^b dx \, \Big[q(x)^{\kappa_1} \big|g^{(\sigma)}(x)\big|^2 + q(x)^{\kappa_2} |g(x)|^2\Big] < \infty\bigg\},    \no \\
& \hspace*{7.2cm} \sigma \in \bbN, \; \kappa_1, \kappa_2 \in [0,\infty),      \lb{1.26} 
\end{align}
with $\cD((a,b))^{\prime}$ the space of distributions on $(a,b)$. 
In the general case, where $\sigma \in (0,\infty) \backslash \bbN$, upon writing 
\begin{equation}
\sigma = \lfloor \sigma \rfloor + \{\sigma\}, \quad \lfloor \sigma \rfloor \in \bbN, \quad 
\{\sigma\} \in (0,1),    \lb{1.27}
\end{equation}
one obtains the fractional, weighted Sobolev spaces
\begin{align}
& H^{\sigma}_{\kappa_1,\kappa_2} \big((a,b);q^{\kappa_1};q^{\kappa_2}\big) = \bigg\{ g \in \cD((a,b))^{\prime} 
\, \bigg| \, g \in L^2((a,b)), \, g^{(\sigma)} \in L_{loc}^2((a,b));    \no \\
& \quad \|g\|_{H^{\sigma}_{\kappa_1,\kappa_2} ((a,b);q^{\kappa_1};q^{\kappa_2})}^2 :=
\int_a^b dx \int_a^b dx' \f{\big|q(x)^{\kappa_1/2} g^{(\lfloor \sigma \rfloor)}(x) - 
q(x')^{\kappa_1/2} g^{(\lfloor \sigma \rfloor)}(x')\big|^2}{|x-x'|^{1+2 \{\sigma\}}}     \no \\
& \hspace*{4.1cm} + \int_a^b dx \, q(x)^{\kappa_2} |g(x)|^2 < \infty\bigg\},    \no \\
& \hspace*{4.1cm} \sigma \in (0,\infty) \backslash \bbN, \; \kappa_1, \kappa_2 \in [0,\infty).  
\end{align}
Finally,
\begin{equation}
H^{\sigma}_{\kappa_1,\kappa_2;0} \big((a,b);q^{\kappa_1};q^{\kappa_2}\big) = 
\ol{C_0^{\infty}((a,b))}^{\|\dott\|_{H^{\sigma}_{\kappa_1,\kappa_2} ((a,b);q^{\kappa_1};q^{\kappa_2})}}. 
\end{equation}

$\bullet$ Connections have been made between the left-definite operator theory for
the classical orthogonal polynomials and the subject of combinatorics. For
example, integral powers of the second-order Laguerre expression (and the
second-order Hermite expression) involve Stirling numbers of the second kind;
this is a new application of these combinatorial numbers. In the case of the
Jacobi expression (which includes Legendre polynomials), a new set of
combinatorial numbers called the Jacobi--Stirling numbers (which include the
Legendre--Stirling numbers) were discovered. Their properties mirror those of
the Stirling numbers of the second kind, see \cite{ELW00}, \cite{EKLWY07}, \cite{ELW02}, and \cite{LW02}.

$\bullet$ One of the more surprising applications of the left-definite theory
involves the analytic study of the classical Lagrangian symmetric Laguerre differential
expression
\begin{equation}
(\tau_{Lag,\alpha} f)(x):=\frac{1}{x^{\alpha}e^{-x}}\left(  -\left(  x^{\alpha
+1}e^{-x} f^{\prime}(x)\right)  ^{\prime}+x^{\alpha}e^{-x} f(x)\right), \quad \alpha\in \bbR, \; x \in (0,\infty),     \lb{1.15}
\end{equation}
with associated generalized eigenvalue equation
\begin{equation}
(\tau_{Lag,\alpha} f)(x) =(n+1) f(x), \quad \alpha\in\bbR, \; x \in (0,\infty),        \lb{1.16}
\end{equation}
in the special case where $\alpha$ is a negative integer, say, $-\alpha=k\in\mathbb{N}$.
Of course, for any real $\alpha$ and $n\in\mathbb{N}_{0},$ the $n$th degree Laguerre
polynomial $L_{n}^{\alpha}$ is a solution of (\ref{1.16}). Furthermore,
when $\alpha>-1,$ it is well-known that $\{L_{n}^{\alpha}\}_{n=0}^{\infty}$
forms a complete orthogonal set in $L^{2}\big((0,\infty);x^{\alpha}e^{-x}dx\big)$; see
\cite[Ch.~V]{Sz75}. The non-classical case $\alpha<-1$ and $-\alpha
\notin\mathbb{N}$ has also been well-studied \cite{De98}; the appropriate
space in this case is a certain Pontryagin space. In the remaining case,
$-\alpha=k\in\mathbb{N},$ the tail-end Laguerre polynomials $\{L_{n}
^{-k}\}_{n=k}^{\infty}$ form a complete set of eigenfunctions of the
self-adjoint operator in $L^{2}\big((0,\infty);x^{-k}e^{-x}dx\big)$,
\begin{align}
& A_{k}f  =\tau_{Lag,-k} f,    \no \\
& f \in \dom(A_{k}) = \big\{g \in L^{2}\big((0,\infty);x^{-k}e^{-x}dx\big) \, \big| \,
g, g^{\prime} \in AC_{loc}((0,\infty)); \\
& \hspace*{5cm} \tau_{Lag,-k} g \in L^{2} \big((0,\infty);x^{-k}e^{-x}dx\big)\big\}.    \no
\end{align}
Moreover, $A_{k}$ is bounded from below in $L^{2}\big((0,\infty);x^{-k}e^{-x}dx\big)$ by $I_{L^{2}((0,\infty);x^{-k}e^{-x}dx)}$ so, from the point of view of left-definite theory, there is a continuum of
left-definite spaces $\{\cH_{s,k}\}_{s \in (0,\infty)}$ and left-definite operators
$\{A_{s,k}\}_{s \in (0,\infty)}$ associated with the pairs $\big(L^{2}\big((0,\infty);x^{-k}e^{-x}dx\big),A_{k}\big)$.
When $- \alpha = k\in\mathbb{N},$ Kwon and Littlejohn in
\cite{KL95} showed that these Laguerre polynomials
$\{L_{n}^{-k}\}_{n=0}^{\infty}$ form a complete orthogonal set in the Sobolev
space
\begin{align}
\begin{split}
& S_{k}((0,\infty))= \big\{g \in L^{2} \big((0,\infty);e^{-x}dx\big) \, \big| \, g, g^{\prime}
,\ldots,g^{(k-1)} \in AC_{loc}([0,\infty));    \lb{1.18} \\
& \hspace*{7cm} g^{(k)}\in L^{2} \big((0,\infty);e^{-x}dx\big)\big\},
\end{split}
\end{align}
endowed with the inner product
\begin{equation}
(f,g)_{k}:=\sum_{\ell=0}^{k-1} \big\langle w_{k,\ell}, \ol{f^{(\ell)}} g^{(\ell)} \big\rangle
+ \int_{0}^{\infty} e^{-x}dx \, \ol{f^{(k)}(x)} g^{(k)}(x), \quad f, g \in S_{k}((0,\infty)), \lb{1.19}
\end{equation}
where
\begin{equation}
w_{k,\ell}=\dbinom{k}{\ell}\sum_{j=0}^{k-\ell-1}\dbinom{k-\ell-1}{j}\delta^{(j)}_0,
\end{equation}
and $\delta^{(j)}_0$ is the $j$-th derivative of the classic Dirac point measure supported at $0$ defined for appropriate differentiable functions, with the pairing $\langle \dott, \dott \rangle$ given by
\begin{equation}
\big\langle\delta^{(j)}_0,f \big\rangle=(-1)^{j}f^{(j)}(0), \quad j=0,1,\dots \, .
\end{equation}
It is natural to ask if there exists a self-adjoint operator $T_{k}$,
generated by the Laguerre expression $\tau_{Lag,-k}$ in $S_{k}((0,\infty))$, 
having the Laguerre polynomials $\{L_{n}^{-k}\}_{n=0}^{\infty}$ as
eigenfunctions. The answer is yes; Everitt, Littlejohn, and Wellman
\cite{ELW04} constructed this operator using left-definite
theory. They decomposed $S_{k}((0,\infty))$ as
\begin{equation}
S_{k}((0,\infty)) = S_{k,1}((0,\infty)) \oplus S_{k,2}((0,\infty)),
\end{equation}
where
\begin{align}
\begin{split}
& S_{k,1}((0,\infty)) = \big\{g \in S_{k}((0,\infty)) \, \big| \, g^{(j)}(0)=0, \quad j=0,1,\ldots,k-1)\big\} \\
& \quad = \cV_{k,-k} = \big\{g \in L^2\big((0,\infty);x^{-k} e^{-x}dx\big) \, \big| \, g^{(j)} \in AC_{loc}([0,\infty)), \, j =0,\dots,k-1;    \\
& \hspace*{5.8cm} g^{(j)} \in L^2\big((0,\infty);x^{j-k} e^{-x}dx\big), \, j=0,\dots,k\big\}     \\
& \quad \subseteq H^k_0 \big ((0,\infty); e^{-x} dx\big),
\end{split}
\end{align}
and
\begin{equation}
S_{k,2}((0,\infty)) = \big\{g \in S_{k}((0,\infty)) \, \big| \,  g^{(k)}(x)=0\text{ for a.e.~$x\in (0,\infty)$}\big\}.
\end{equation}
For $n\geq k,$ the Laguerre polynomial $L_{n}^{-k}(\dott)$ has the property that
$(L_{n}^{-k})^{(j)}(0)=0$ for $j=0,1,\ldots,k-1$ so, when $n\geq k,$
$L_{n}^{-k} \in S_{k,1}((0,\infty))$. Moreover, it is clear that $\{L_{n}
^{-k}\}_{n=0}^{k-1}\in S_{k,2}((0,\infty))$. It is shown in \cite{ELW04} that the $k$th
\textit{left-definite} operator
\begin{equation}
T_{k,k} f =\tau_{Lag,-k} f, \quad f \in \dom(T_{k,k}) =V_{k+2,-k},
\end{equation}
is self-adjoint in $S_{k,1}((0,\infty))$ and has the Laguerre polynomials
$\{L_{n}^{-k}\}_{n=k}^{\infty}$ as a complete set of eigenfunctions. Here
\begin{align}
& \cV_{k+2,-k} = \big\{g \in L^2\big((0,\infty); x^{-k}e^{-x}dx\big) \, \big| \, g^{(j)} \in AC_{loc}((a,b)), \, j=0,\dots,k+1;      \no \\
& \hspace*{3.8cm} g^{(\ell)} \in L^2\big((0,\infty); x^{\ell-k}e^{-x}dx\big), \, \ell=0,\dots,k+2\big\}.
\end{align}
Moreover, on the finite-dimensional space $S_{k,2}((0,\infty))$, it is clear
that the operator
\begin{equation}
B_{k} f  =\tau_{Lag,-k}f, \quad f \in \dom(B_{k}) =S_{k,2}((0,\infty)),
\end{equation}
is self-adjoint in $S_{k,2}((0,\infty))$ and has the Laguerre polynomials
$\{L_{n}^{-k}\}_{n=0}^{k-1}$ as a complete set of eigenfunctions. Finally, it follows that
\begin{equation}
T_{k}=T_{k,k}\oplus B_{k}
\end{equation}
is self-adjoint in $S_{k}((0,\infty))$ with $\{L_{n}^{-k}\}_{n=0}^{\infty}$ as a
complete set of eigenfunctions.

$\bullet$ For additional details in connection to left-definite theory and orthogonal polynomials, see, for instance, \cite{DHS11}, \cite{Ev80}, \cite{ELM02}, \cite{ELW00}, \cite{ELW04}, \cite{FL21}, \cite{KL95}, \cite{LW13}, \cite{LW16}, \cite{LZ07}, \cite[Ch.~7]{Tr78}, \cite{Vo00}, and \cite[Chs.~5, 12]{Ze05}.

$\bullet$ A mathematical framework for machine learning was developed by Cucker and
Smale in \cite[Remark~4, p.~29]{CS02} and Smale and Zhou in
\cite[Sect.~3]{SZ03} and \cite[Theorem~2]{SZ07}. In
their general setting for learning in a Hilbert space, they construct a
continuum of left-definite Hilbert spaces $\{\mathcal{H}_{r}(A)\}_{r=1}^{2}$ -
which they call interpolation spaces - generated from the unbounded
self-adjoint inverse $A$ of a certain Hilbert--Schmidt operator $A^{-1}$ of integral operator type, 
\begin{equation}
(A^{-1}f)(x)=\int_{X} d\mu(y) \, K(x,y) f(y), \quad f \in L^{2}(X;d\mu).
\end{equation} 
In this setting, $X\subset\mathbb{R}^{n}$ is a compact data input space$,$
$\mu$ is a positive Borel measure on $X,$ and $K(\dott,\dott)$ is a Mercer (i.e., positive definite and square integrable w.r.t. $d\mu$ in both variables) integral kernel. The
first left-definite space is the reproducing kernel Hilbert space for $K.$ In this
framework, solutions of certain Tychonov regularization schemes are shown to
converge to idealized target functions associated with a wide range of
learning problems. If the measure $\mu$ is generated by the marginal
probability distribution for the learning problem, then convergence is given
in terms of probability. Furthermore, in their analysis, the location of the
target function between the first and second left definite spaces determines
convergence rates. These convergence rates are characterized by certain
operator inequalities involving positive powers of $A$. 

In Section \ref{s2} we present the model operator approach to abstract left-definite theory and derive its principal results in an efficient manner. In Section \ref{s3} we present several concrete applications including the scale of Sobolev spaces on $\bbR$ associated with the one-dimensional Laplace operator (the Fourier transform arguments extend to the multi-dimensional case), the self-adjoint operator $A$ in $\ell^2(\bbN)$ given by
\begin{align}
\begin{split}
&(A f)_n = n f_n,  \quad n \in \bbN,   \\
& f \in \dom(A) = \big\{g = \{g_n\}_{n\in \bbN} \in \ell^2(\bbN) \, \big| \, \{n g_n\}_{n\in \bbN} \in \ell^2(\bbN)\big\},     \lb{1.28}
\end{split}
\end{align}
and its associated $s$-th left-definite operators, the periodic Laplacian on $(0,2\pi)$, the one-parameter family of self-adjoint Laplacians on the half-line $(0,\infty)$, and the self-adjoint extensions of the Bessel operator on $(0,\infty)$. Section \ref{s4}, in many ways the {\it pi\`ece de r\'esistance} of this paper, takes a rather detailed look at  Hermite polynomials, the underlying Hermite operator $A_H$ in $L^2\big(\bbR; e^{-x^2} dx\big)$, and the unitarily equivalent harmonic oscillator operator $T_{HO}$ in $L^2(\bbR)$. In particular, employing interpolation techniques for operator domains and certain commutation techniques, we derive explicit expressions for domains of (strictly) positive powers $T_{HO}^s$ (and hence those of $A_H^s$) in terms of fractional Sobolev spaces and domains of $|X|^s$, $s \in (0,\infty)$, with $X$ denoting the operator of multiplication by the independent variable $x \in \bbR$. To the best of our knowledge, these results are new. 

A variety of domains of fractional powers of periodic Laplacians and Laplacians on the half-line are studied in Appendices \ref{sA}--\ref{sC}.

Finally, we briefly summarize some of the notation used in this paper: Let $\cH$ be a separable
complex Hilbert space, $(\cdot,\cdot)_{\cH}$ the scalar product in
$\cH$ (linear in the second factor), and $I_{\cH}$ the identity operator in $\cH$.
Next, let $T$ be a linear operator mapping (a subspace of) a
Banach space into another, with $\dom(T)$
denoting the domain
of $T$.

The spectrum, essential spectrum, point spectrum (i.e., the set of eigenvalues), discrete spectrum, absolutely continuous spectrum, singularly continuous spectrum, and resolvent set of a closed linear operator in $\cH$ will be denoted by
$\sigma(\cdot)$, $\sigma_{ess}(\cdot)$, $\sigma_{p}(\cdot)$, $\sigma_{d}(\cdot)$,
$\sigma_{ac}(\cdot)$, $\sigma_{sc}(\cdot)$, and $\rho(\cdot)$, respectively.

The Banach spaces of bounded and compact linear operators in $\cH$ are denoted by $\cB(\cH)$ and $\cB_\infty(\cH)$, respectively. Similarly, the Schatten--von Neumann (the $\ell^2$-based) trace ideals will subsequently be denoted by $\cB_p(\cH)$, $p \in [1,\infty)$.

We denote by $E_A(\cdot)$ the family of strongly right-continuous spectral projections of a self-adjoint operator $A$ in $\cH$ (in particular, $E_A(\lambda) = E_A((-\infty,\lambda])$, $\lambda \in \bbR$).

To simplify notation, we will write $L^2(\Omega)$ instead of $L^2(\Omega; d^nx)$, where $\Omega \subseteq \bbR^n$, $n \in \bbN$, whenever the underlying Lebesgue measure is understood. \\


\section{A Model Operator Approach to Abstract Left-Definite Problems} \lb{s2}

We briefly summarize the principal results of abstract left-definite theory considered in \cite{LW02}, \cite{LW13} and proceed to demonstrate that the theory permits an elementary model operator approach.

\begin{hypothesis} \lb{h2.1}
Let $A$ be a self-adjoint operator in the complex, separable Hilbert space $\cH = (\cV, (\dott, \dott)_{\cH})$ bounded from below, satisfying $($without loss of generality$)$
\begin{equation}
A \geq I_{\cH}.     \lb{2.1}
\end{equation}
\end{hypothesis}

Given $A$, we introduce the scale of Hilbert spaces $\cH_r(A)$, $r \in [0,\infty)$, introduced via\footnote{For notational simplicity, we identify $\cH$ and $\cH_0(A)$ as well as $\cV_0$ and $\cV_0(A)$, even though $\cH_0(A)$ and $\cV_0(A)$ are independent of $A$.}
\begin{align}
& \cH_0(A) = (\cV_0(A), (\dott,\dott)_{\cH_0})= (\cV, (\dott, \dott)_{\cH}) = \cH,    \no \\
& \cH_r(A) = (\cV_r(A), (\dott,\dott)_{\cH_r(A)})       \lb{2.3} \\
& \hspace*{1cm} = \big(\dom\big(A^{r/2}\big), (\dott,\dott)_{\cH_r(A)} = \big(A^{r/2} \dott, A^{r/2} \dott\big)_{\cH}\big), \quad r \in (0,\infty),    \no
\end{align}
where
\begin{equation}
\cV_r(A) = \dom\big(A^{r/2}\big), \quad r \in [0,\infty).    \lb{2.4}
\end{equation}
In particular,
\begin{equation}
\dom(A) = \cV_2(A).     \lb{2.5}
\end{equation}

Traditionally, $\cH_r(A)$ is also called the {\it $r$th left-definite space}, $r \in (0,\infty)$, associated with the pair $(A, \cH)$, see, for instance, \cite{LW02}, \cite{LW13}, and \cite{LZ07}.

Before we start our analysis, it is convenient to introduce an elementary model operator: Let $\mu$ be a Borel measure on $[1,\infty)$ and consider the maximally defined operator of multiplication by the independent variable in the Hilbert space $\cK = \cK_0 = L^2([1, \infty); d \mu(\lambda))$ denoted by $T$,
\begin{align}
\begin{split}
& (T f)(\lambda) = \lambda f(\lambda) \, \text{ for $\mu$-a.e. } \,  \lambda \in [1,\infty),   \\
& f \in \dom(T) = \bigg\{g \in L^2([1, \infty); d \mu(\lambda)) \, \bigg | \,
\int_{[1,\infty)} d \mu(\lambda) \, \lambda^2 |g(\lambda)|^2 < \infty\bigg\}   \lb{2.6a} \\
& \hspace*{1.85cm} = L^2\big([1, \infty); \big(1+\lambda^2\big) d \mu(\lambda)\big).
\end{split}
\end{align}
The model operator $T$ will play a crucial role in the proof of our principal result, Theorem \ref{t2.3}. In particular, if $A$ has simple spectrum (i.e., $A$ possesses a cyclic vector), and $\inf(\sigma(A))=1$, then $A$ is unitarily equivalent to some positive definite operator $T$ of the type \eqref{2.6a} with $\mu([1,\infty)) < \infty$ (cf.\ \eqref{2.7a} below). For general operators $A$ it suffices to recall the spectral theorem in multiplication operator form and hence refer to Theorem \ref{t2.4}.

\begin{remark} \lb{r2.1a}
Due to its importance in several concrete examples in Sections \ref{s3} and \ref{s4}, we briefly review the situation of self-adjoint operators that have simple spectrum (see, e.g., \cite[Sect.~VII.2]{RS80}, or \cite[Sect.~5.4]{Sc12}): Let $A$ be a self-adjoint operator $A$ in the separable, complex Hilbert space $\cH$ with associated strongly right-continuous family of spectral projections denoted by $E_A(\dott)$, that is, 
\begin{equation} 
A = \int_{\sigma(A)} dE_A(\lambda) \, \lambda. 
\end{equation} 
Then $g_0 \in \bigcap_{n\in\bbN_0} \dom(A^n)$ is called a cyclic (or generating) vector for $A$ if $\cH$ is the smallest reducing subspace for $A$ that contains $g_0$ (equivalently, $\mbox{lin.span} \{A^n g_0 \, | \, n \in \bbN_0\}$ is dense in $\cH$). The operator $A$ is then said to have simple spectrum if it  possesses a cyclic vector $f_0 \in \cH$. Given $f_0 \in \cH$, let
\begin{equation} 
\mu_{0,A}(\dott) = (f_0, A f_0)_{\cH}.     \lb{2.7a}
\end{equation} 
Then $\mu_{0,A}(\dott)$ generates a finite nonnegative Borel measure on $\bbR$. In this context, $A$ is unitarily equivalent to the maximally defined operator of multiplication $M_A$ by the independent variable in $L^2(\bbR; d\mu_{0,A})$ given by
\begin{equation}
(M_A g)(\lambda) = \lambda g(\lambda), \quad g \in \dom(M_A) = L^2\big(\bbR; \big(1+\lambda^2\big) d\mu_{0,A}(\lambda)\big).  
\end{equation}
More precisely, there exists a unitary operator $U_0 : \cH \to L^2(\bbR; d\mu_{0,A})$ (i.e., an isometric isomorphism), satisfying 
\begin{equation} 
(U_0 f_0)(\lambda) =1 \, \text{ for $\mu_{0,A}$-a.e. $\lambda \in \bbR$}, 
\end{equation} 
such that
\begin{equation}
A = U_0 M_A U_0^{-1} \, \text{ and } \, \sigma(A) = \supp (E_A) = \supp (\mu_{0,A}). 
\end{equation}
\end{remark}

\begin{lemma} \lb{l2.2}
Assume Hypothesis \ref{h2.1}. Then
\begin{equation}
\cH_s(A) \subseteq \cH_r(A), \quad 0 \leq r < s,    \lb{2.6}
\end{equation}
and, if $A$ is unbounded, strict containment holds in \eqref{2.6}, that is,
\begin{equation}
\cH_s(A) \subsetneqq \cH_r(A), \quad 0 \leq r < s, \, \text{ $A$ unbounded}.    \lb{2.7}
\end{equation}
\end{lemma}
\begin{proof}
Employing the spectral theorem for $A$ and $F(A)$ for measurable functions $F$ on $[\varepsilon,\infty)$, particularly, the powers $F_s(\lambda) = \lambda^s$, $\lambda \in [\varepsilon, \infty)$, $s \in (0,\infty)$,
\begin{equation}
A = \int_{\sigma(A)} \lambda \, dE_A(\lambda), \quad F(A) = \int_{\sigma(A)} F(\lambda) \, dE_A(\lambda),
\end{equation}
with $E_A((-\infty,\lambda]) = E_A(\lambda)$, $\lambda \in [\varepsilon, \infty)$, the strongly right-continuous family of spectral projections associated with $A$, one immediately concludes that
\eqref{2.6} holds. To show that strict containment in \eqref{2.7} holds, we argue by contradiction and assume that $\cH_s(A) = \cH_r(A)$ for some $0 \leq r < s$. Thus, $\dom\big(A^{s/2}\big) = \dom\big(A^{r/2}\big)$ for some $0 \leq r < s$ and reverting to the operator $T$ in \eqref{2.6a}, this amounts to assuming $\dom\big(T^{s/2}\big) = \dom\big(T^{r/2}\big)$ for some $0 \leq r < s$. Then
\begin{equation}
\text{$f \in \dom\big(T^{r/2}\big)$ is equivalent to $f \in L^2([1,\infty); d\nu_r)$},    \lb{2.10a}
\end{equation}
where $d\nu_r = \big(1 + |\dott|^r\big) d\mu$, and
\begin{equation}
\text{$f \in \dom\big(T^{s/2}\big)$ is equivalent to
$f \in L^2\big(\big[\big(1+ \lambda^s\big)\big/\big(1+ \lambda^r\big)\big] d\nu_r\big)$}.    \lb{2.11a}
\end{equation}
The facts \eqref{2.10a} and \eqref{2.11a} imply that the operator of multiplication by the function
$\big(1+ \lambda^s\big)\big/\big(1+ \lambda^r\big)$ defines a symmetric and everywhere defined operator in $L^2([1,\infty); d\nu_r)$. By the Hellinger--Toeplitz theorem (see, e.g., \cite[Theorem~5.7]{We80}), this operator would have to be bounded, contradicting $s > r$ and the fact that $d\nu_r$ has unbounded support.
\end{proof}

Before describing the principal result on left-definite spaces and operators, we pause for a moment and recall the introduction of the general scale of Hilbert spaces associated with a self-adjoint operator $A$ satisfying Hypothesis \ref{h2.1}. $\cH_r(A)$, for $r \in \bbR$, can be defined via
\begin{equation}
\cH_r(A) = \ol{\bigg(\cD = \bigcap_{n \in \bbN} \dom\big(A^n\big); \, \|u\|_{\cH_r(A)}
= \big\|A^{r/2} u\big\|_{\cH}, \, u \in \cD \bigg)},
\quad r \in \bbR.     \lb{2.9}
\end{equation}
(Of course, $\cH_0(A) = \cH$ and \eqref{2.3} and \eqref{2.9} coincide for $r \in (0,\infty)$.)

One then obtains the dense and continuous embeddings,
\begin{align}
& \cH_{r_2}(A) \hookrightarrow \cH_{r_1}(A) \hookrightarrow \cH_0(A) = \cH = \cH^* = \cH_0(A)^*      \\
& \quad \hookrightarrow \cH_{r_1}(A)^* = \cH_{-r_1}(A) \hookrightarrow \cH_{r_2}(A)^* = \cH_{-r_2}(A), \quad r_1, r_2 \in (0,\infty), \; r_1 < r_2.    \no
\end{align}

One also observes that $A^{r/2}$, $r \in \bbR$, extends by continuity to an isometric map, denoted by $\big(\wti A\big)^{r/2}$, from $\cH_s(A)$ onto $\cH_{s-r}(A)$, $r,s \in \bbR$, that is,
\begin{align}
& \big(\wti A\big)^{r/2} \colon \begin{cases} \cH_s(A) \to \cH_{s-r}(A), \\
u \mapsto \big(\wti A\big)^{r/2} u,
\end{cases}  \text{ is an isometric isomorphism}      \lb{2.11}  \\
& \quad \text{(i.e., a unitary map), } \, r, s \in \bbR.    \no
\end{align}

The duality pairing between $\cH_r(A)$ and $\cH_{-r}(A)$, $r \in (0,\infty)$, is given by
\begin{align}
\begin{split}
{}_{\cH_r(A)}\langle u, \ell \rangle_{\cH_{-r}(A)} = \Big(\big(\wti A\big)^{r/2} u,\big(\wti A\big)^{-r/2} \ell \Big)_{\cH},& \\
u \in \cH_r(A), \; \ell \in \cH_{-r}(A), \; r \in (0,\infty),&
\end{split}
\end{align}
and
\begin{align}
& |{}_{\cH_r(A)}\langle u, \ell \rangle_{\cH_{-r}(A)}| \leq \|u\|_{\cH_r(A)} \|\ell\|_{\cH_{-r}(A)}, \quad u \in \cH_r(A), \; \ell \in \cH_{-r}(A),  \; r \in (0,\infty),    \no \\
& {}_{\cH_r(A)}\langle u, f \rangle_{\cH_{-r}(A)} = (u,f)_{\cH}, \quad u \in \cH_r(A), \; f \in \cH, \; r \in (0,\infty),   \\
& \| \dott \|_{\cH_{-r}(A)} \leq \| \dott \|_{\cH} \leq \| \dott \|_{\cH_r(A)}, \quad r \in (0,\infty).   \no
\end{align}

Returning to the situation $r \in [0,\infty)$, the principal abstract results in \cite{LW02}, \cite{LW13} then can be summarized as follows:

\begin{theorem} \lb{t2.3}
Assume Hypothesis \ref{h2.1}. Then the following items $(i)$ and $(ii)$ hold: \\[1mm]
$(i)$ There exists a unique operator $A_r$ in $\cH_r(A)$ satisfying
\begin{align}
& \text{$A_r$ on $\dom(A_r) = \dom\big(A^{(r+2)/2}\big) = \cV_{r+2}(A)$ is self-adjoint in $\cH_r(A)$, $r \in [0,\infty)$,}    \no \\
& A_0=A,   \\
& A_r = A\big|_{\cV_{r+2}(A)}, \; r \in [0,\infty).   \no
\end{align}
In particular,
\begin{align}
& \cV_{r+2}(A) = \dom(A_r) \subseteq \dom(A) = \cV_2(A), \quad r \in [0,\infty),    \lb{2.19} \\
& \cV_{s+2}(A) = \dom(A_s) \subseteq \dom(A_r) = \cV_{r+2}(A), \quad 0 \leq r < s,    \lb{2.20}
\end{align}
and the inclusions in \eqref{2.19} and \eqref{2.20} are strict if $A$ is unbounded, that is,
\begin{align}
\begin{split} 
& \cV_{r+2}(A) = \dom(A_r) \subsetneqq \dom(A) = \cV_2(A), \quad r \in (0,\infty),  \, \text{ $A$ unbounded},  \\
& \cV_{s+2}(A) = \dom(A_s) \subsetneqq \dom(A_r) = \cV_{r+2}(A), \quad 0 \leq r < s,
\, \text{ $A$ unbounded}.
\end{split} 
\end{align}
$(ii)$ If $A \in \cB(\cH)$, then the norms induced by $(\dott,\dott)_{\cH_r(A)}$ and $(\dott,\dott)_{\cH_s(A)}$, $r,s \in [0,\infty)$ $($in particular, those of $(\dott,\dott)_{\cH_r(A)}$, $r \in [0,\infty)$, and $(\dott,\dott)_{\cH}$$)$ are equivalent, and hence,
\begin{equation}
\cV_r = \cV, \quad \cH_r(A)=\cH, \quad A_r =A, \quad r \in [0,\infty).
\end{equation}
$(iii)$ For $r,s,t \in [0,\infty)$,
\begin{equation}
A^{(r-s)/2} \colon \begin{cases} \cH_{r+t}(A) \to \cH_{s+t}(A), \\
u \mapsto A^{(r-s)/2} u,
\end{cases}
\end{equation}
is an isometric isomorphism $($i.e., a unitary operator\,$)$  from $\cH_{r+t}(A)$ onto $\cH_{s+t}(A)$ and
\begin{align}
& A^{(r-s)/2} A_r A^{-(r-s)/2} = A_s, \quad r, s \in [0,\infty),   \\
& \cV_{s+2}(A) = \dom(A_s) = A^{(r-s)/2} \dom(A_r) = A^{(r-s)/2} \cV_{r+2}(A), \quad r,s \in [0,\infty).  \no 
\end{align}
Thus, $A_r$ in $\cH_r$ and $A_s$ in $\cH_s$ are unitarily equivalent for all $r, s \in [0,\infty)$, in particular, $A_r$ in $\cH_r$ and $A$ in $\cH$ are unitarily equivalent for all $r \in [0,\infty)$. As a consequence, all unitary invariants $($cf.\ Remark \ref{r2.5}\,$(ii)$$)$ of $A_r$ and $A_s$,  $r, s \in [0,\infty)$, and hence all unitary invariants of $A_r$, $r \in [0,\infty)$, and $A$, coincide.
\end{theorem}
\begin{proof}
To prove Theorem \ref{t2.3}, it suffices to recall the simple model operator $T$ in \eqref{2.6a}, for which all properties asserted in Theorem \ref{t2.3} become obvious. Indeed, identifying
\begin{align}
& A \, \text{ and } \, T,   \no \\
& \cH_r(A), \; \cH_0(A) = \cH, \,  \text{ and } \, \cK_r(T)
= L^2\big([1,\infty); \big(1+\lambda^r\big) d\mu(\lambda)\big), \; \cK_0 (T) = \cK,     \no \\   
& \hspace*{9cm} r \in [0,\infty),     \\
& A_r \, \text{ and } \,
T_r = T\big|_{L^2([1,\infty); (1+\lambda^{r+2}) d\mu(\lambda))}, \quad r \in [0,\infty),     \no 
\end{align}
all assertions in Theorem \ref{t2.3} are evident for $T$. In particular, since
\begin{equation}
T\big|_{\cK_{r+2}(T)} \cK_{r+2}(T) = T_r \cK_{r+2}(T) = \cK_r(T),
\end{equation}
which follows from the fact that $T f \in L^2\big([1,\infty); \big(1+\lambda^r\big) d\mu(\lambda)\big)$ if and only if $f \in L^2\big([1,\infty); \big(1+\lambda^{r+2}\big) d\mu(\lambda)\big)$ (one recalls that $(Tf)(\lambda) = \lambda f(\lambda)$ for $f \in L^2\big([1,\infty); \big(1+\lambda^2\big) d\mu(\lambda)\big)$ by \eqref{2.6a}). Thus, $T_r$ is the maximally defined operator of multiplication by the independent variable $\lambda$ in $\cK_r(T)$, and hence self-adjoint.

Explicitly, 
\begin{align}
& \cK_r(T) = \big(\cW_r(T), (\dott,\dott)_{\cK_r(T)}\big)
= \big(L^2\big([1, \infty); \big(1+\lambda^r\big) d \mu(\lambda)\big),
\big(T^{r/2} \dott, T^{r/2} \dott \big)_{\cK}\big), \no \\
& \hspace*{9.4cm} r \in [0,\infty).
\end{align}

If $A$ in Theorem \ref{t2.3} has simple spectrum, then, as mentioned in Remark \ref{r2.1a}, $A$ is unitarily equivalent to some positive definite operator $T$ of the type \eqref{2.6a} and hence Theorem \ref{t2.3} follows in this particular case. For general operators $A$ it suffices to recall the spectral theorem in multiplication operator form in Theorem \ref{t2.4} below, and consider direct sums of operators of the type $T$.
\end{proof}

\begin{remark} \lb{r2.4}
$(i)$ The proof of Theorem \ref{t2.3} demonstrates that if $A$ has simple spectrum, then up to unitary equivalence, $A$ acts as the maximally defined operator of multiplication by the independent variable $\lambda$ in the polynomially weighted $L^2$-space $L^2\big([1,\infty); \big(1+\lambda^2\big) d\mu(\lambda)\big)$ with a finite Borel measure $\mu$, $\mu([1,\infty)) < \infty$. Consequently, in the particular case of simple spectrum of $A$, Theorem \ref{t2.3} then reduces to elementary properties of the polynomially weighted spaces $L^2\big([1,\infty); \big(1+\lambda^r\big) d\mu(\lambda)\big)$, $r \in [0,\infty)$. The general case becomes a direct sum of these simple cases upon alluding to Theorem \ref{t2.4} below. Either way, the spectral theorem in multiplication operator form does the work in connection with Theorem \ref{t2.3}. \\[1mm]
$(ii)$ In this context one might ask what are simple and nontrivial examples of self-adjoint operators with simple spectrum? An interesting class consists of three-coefficient Sturm--Liouville operators on an arbitrary interval $(a,b) \subseteq \bbR$, with associated differential expressions $\tau$ in the limit point case at $a$ and/or $b$ such that the corresponding minimal operator $T_{min}$ in $L^2((a,b);rdx)$ is bounded from below (this assumption can be relaxed) and one, and hence all, self-adjoint extensions of $T_{min}$ in $L^2((a,b);rdx)$ have purely discrete spectrum (equivalently, empty essential spectrum). Here $\tau$ is of the form 
\begin{equation}
\tau = \f{1}{r(x)} \bigg[- \f{d}{dx} p(x) \f{d}{dx} + q(x)\bigg] \, \text{ for a.e.~$x \in (a,b)$}, 
\end{equation}
with 
\begin{align}
\begin{split} 
& p, r > 0 \, \text{ a.e.~on $(a,b)$}, \quad q \, \text{ real-valued a.e.~on $(a,b)$},     \\
& (1/p), q, r \in L^1_{loc}((a,b); dx),
\end{split} 
\end{align}
and the minimal operator associated with $\tau$ is given by 
\begin{align}
& (T_{min} f)(x) = (\tau f)(x) \, \text{ for a.e.~$x \in (a,b)$},   \no \\
& f \in \dom(T_{min}) = \big\{g \in L^2((a,b); rdx) \, \big| \, g, g' \in AC_{loc}((a,b)); \\
& \hspace*{3.25cm} \supp(g) \text{ compact}; \, \tau g \in L^2((a,b); rdx)\big\}.    \no 
\end{align}
In this context one observes that in a situation where $\tau$ is in the limit circle case at $a$ and in the limit point case at $b$ (a frequently encountered half-line case), the measure $\mu_{0,T}$ in \eqref{2.7a}, with $T$ a self-adjoint extension of $T_{min}$ (all self-adjoint extensions of $T_{min}$ have simple spectrum), corresponds to the measure in the Nevanlinna--Herglotz representation associated with the Weyl--Titchmarsh $m$-function of $T$.

Naturally, this class of Sturm--Liouville operators is relevant in the context of certain orthogonal polynomials, for instance in the case of Hermite polynomials discussed in Section \ref{s4}. 
\\[1mm]
$(iii)$ Unitary equivalence of $A$ and $A_r$, $r \in (0,\infty)$, in Theorem \ref{t2.3}\,$(iii)$ answers (an enhanced version of) \cite[Conjecture~3, p.~78]{FL21} in the affirmative.
\hfill $\diamond$
\end{remark}

For the following version of the spectral theorem for self-adjoint operators, see, for instance,  \cite[Sects.~VII.2, VII.3, VIII.3]{RS80}, \cite[Sects.~5.1--5.4]{Si15}, and \cite[Sect.~7.3]{We80}):

\begin{theorem} \lb{t2.4}
Given a self-adjoint operator $A$ in the complex, separable Hilbert space $\cH$, there exist finite Borel measures $\mu_n$, $1 \leq n \leq N$, $N \in \bbN \cup \{\infty\}$ on $\sigma(A)$, and a unitary operator
\begin{equation}
U \colon \cH \to \bigoplus_{n=1}^N L^2(\bbR; d \mu_n(\lambda))
\end{equation}
such that
\begin{equation}
\big(UAU^{-1} \psi\big)_n(\lambda) = \lambda \psi_n(\lambda), \quad U^{-1} \psi \in \dom(A), \; 1 \leq n \leq N.
\end{equation}
Here $\psi \in \bigoplus_{n=1}^N L^2\big(\bbR; (1+\lambda^2)d \mu_n(\lambda)\big)$ is represented as
$\psi (\dott) = (\psi_1(\dott), \dots, \psi_N(\dott))$. \\
While $U$ and $\mu_n$, $1 \leq n \leq N$, are nonunique, one has
\begin{equation}
\supp(\mu_n) \subseteq \sigma(A), \quad 1 \leq n \leq N.
\end{equation}
\end{theorem}

\begin{remark} \lb{r2.5}
$(i)$ A comparison of \eqref{2.11} and the results in Theorem \ref{t2.3} yield
\begin{align}
& \wti A\big|_{\cH_{r+2}(A)} = A\big|_{\cH_{r+2}(A)} = A_r \colon \begin{cases}
\cH_r(A) \supseteq \dom(A_r) = \cH_{r+2}(A) \to \cH_r(A), \\
u \mapsto A u,
\end{cases}    \no  \\
& \hspace*{9.3cm}	 r \in [0,\infty),    \lb{2.32}
\end{align}
the inclusion (i.e., $\subseteq$) in \eqref{2.32} being strict (i.e., $\subset$) if $A$ is unbounded. Similarly, one verifies that
$\big(\wti A\big)^{s/2} \colon \cH_{r+s}(A) \to \cH_r(A)$ and $A^{s/2} \colon \cH_{r+s}(A) \to \cH_r(A)$,
$r,s \in [0,\infty)$, represent the same unitary map. \\[1mm]
$(ii)$ Perhaps the most remarkable property of the left-definite theory summarized in Theorem \ref{t2.3} concerns the unitary equivalence of $A$, $A_r$, and $A_s$ for all $r, s \in [0,\infty)$ which was also established in \cite{LW13}. A complete set of unitary invariants for a self-adjoint operator $B$ in some complex, separable Hilbert space includes, in addition to the spectral parts of uniform multiplicity
$\ell$, with $1 \leq \ell \leq N \in \bbN \cup \{\infty\}$, for instance,
\begin{equation}
\sigma(B), \quad \sigma_{ess}(B), \quad \sigma_{p}(B), \quad \sigma_{d}(B), \quad \sigma_{ac}(B), \quad \sigma_{sc}(B), \quad \rho(B).
\end{equation}
$(iii)$ Since $A \geq I_{\cH}$, $A^{-s} \in \cB(\cH)$ for $s \in [0,\infty)$, and thus by Theorem \ref{t2.3}\,$(ii)$, the situation $r \in (-\infty,0)$ is not of interest in our present context.
\hfill $\diamond$
\end{remark}

\begin{remark} \lb{r2.7} The concrete example of the closed, densely defined, symmetric operator $T_{min}$ in $L^2((0,1))$ given by  
\begin{align}
\begin{split}
& T_{min} f = - f'',      \\
& f \in \dom(T_{min}) = \big\{g \in H^2((0,1)) \, \big| \, g(0)=g'(0)=g(1)=g'(1)=0\},
\end{split}
\end{align}
with associated Friedrichs extension $T_F$ in $L^2((0,1))$ ,
\begin{align}
\begin{split}
& T_F f = - f'',      \\
& f \in \dom(T_F) = \big\{g \in H^2((0,1)) \, \big| \, g(0)=g(1)=0\},
\end{split}
\end{align}
and their respective squares,
\begin{align}
& (T_F)^2 f = f^{(4)},      \\
& f \in \dom\big(T_F^2\big) = \big\{g \in H^4((0,1)) \, \big| \, g(0)=g''(0)=g(1)=g''(1)=0\},  \no
\end{align}
and
\begin{align}
& (T_{min})^2 f =  f^{(4)},      \\
& f \in \dom\big((T_{min})^2\big) = \big\{g \in H^4((0,1)) \, \big| \, g(0)=g'(0)=g''(0)=g'''(0)=0;
\no \\
& \hspace*{6.05cm} g(1)=g'(1)=g''(1)=g'''(1)=0\big\},  \no
\end{align}
and finally, the Friedrichs extension $\big((T_{min})^2\big)_F$ of $(T_{min})^2$ in $L^2((0,1))$,
\begin{align}
& \big((T_{min})^2\big)_F f =  f^{(4)},      \\
& f \in \dom\big(\big(T_{min}^2\big)_F\big) = \big\{g \in H^4((0,1)) \, \big| \, g(0)=g'(0)=g(1)=g'(1)=0\big\},  \no
\end{align}
yields
\begin{equation}
\dom\big((T_F)^2\big) \neq \dom\big(\big((T_{min})^2\big)_F\big)
\end{equation}
providing a counterexample to \cite[Conjecture~2, p.~78]{FL21}.
\hfill $\diamond$
\end{remark}

\section{Some Applications} \lb{s3}

We illustrate the abstract approach to left-definite theory in Section \ref{s2} with the help of a few examples.

\begin{example} [The scale of Sobolev spaces on $\bbR$] \lb{e3.1}
Introducing
\begin{equation}
L^2_s(\bbR) = L^2 \Big(\bbR; \big(1 + |\xi|^2\big)^s d\xi\Big), \quad s \in \bbR,
\end{equation}
and identifying,
\begin{equation}
L^2_0(\bbR) = L^2 (\bbR) = \big(L^2 (\bbR)\big)^* = \big(L^2_0(\bbR)\big)^*,
\end{equation}
one gets the chain of Hilbert spaces with respect to $L^2_0(\bbR) = L^2 (\bbR)$,
\begin{equation}
L^2_s(\bbR) \subset L^2 (\bbR) \subset L^2_{-s} (\bbR) = \big(L^2_s (\bbR)\big)^*, \quad s \in (0,\infty).
\end{equation}
Next, one introduces the maximally defined operator $G_0$ of multiplication by the function
$\big(1 + |\dott|^2\big)^{1/2}$ in $L^2 (\bbR)$,
\begin{align}
\begin{split}
& (G_0 f)(\xi) = \big(1 + |\xi|^2\big)^{1/2} f(\xi), \\
& \, f \in \dom(G_0) = \Big\{ g \in L^2 (\bbR) \, \Big| \,
\big(1 + |\dott|^2\big)^{1/2} g \in L^2 (\bbR)\Big\}.
\end{split}
\end{align}
The operator $G_0$ extends to an operator defined on the entire scale $L^2_s(\bbR)$,
$s \in \bbR$, denoted by $\wti G_0$, such that
\begin{equation} \lb{3.5}
\wti G_0 : L^2_s(\bbR) \to L^2_{s-1}(\bbR), \quad
\big(\wti G_0\big)^{-1} : L^2_s(\bbR) \to L^2_{s+1}(\bbR), \, \text{ bijectively, } \, \; s \in \bbR.
\end{equation}
While
\begin{equation}
I: L^2 (\bbR) \to \big(L^2 (\bbR)\big)^* = L^2 (\bbR)
\end{equation}
represents the standard identification operator between $L^2_0(\bbR) = L^2 (\bbR)$ and its adjoint space,
$\big(L^2 (\bbR)\big)^* = \big(L^2_0(\bbR)\big)^*$, via the Riesz lemma $($thus, $I = I_{L^2 (\bbR)}$$)$, and one does not identify
$\big(L^2_s (\bbR)\big)^*$ with $L^2_s(\bbR)$ when $s>0$. In fact, it is the operator $\big(\wti G_0\big)^2$ that provides a unitary map
\begin{equation}
\big(\wti G_0\big)^2 : L^2_s(\bbR) \to L^2_{s-2}(\bbR), \quad s \in \bbR.
\end{equation}
In particular,
\begin{equation}
\big(\wti G_0\big)^2 : L^2_1(\bbR) \to L^2_{-1}(\bbR) = \big(L^2_1(\bbR)\big)^* \, \text{ is a unitary map}.
\end{equation}
More generally,
\begin{equation}
\big(\wti G_0\big)^r : L^2_s(\bbR) \to L^2_{s-r}(\bbR) \, \text{ is a unitary map, } \; r, s \in \bbR.
\end{equation}

Denoting the Fourier transform on $L^2 (\bbR)$ by $\cF$, and then extending it to the entire scale
$L^2_s(\bbR)$, $s \in \bbR$, more generally, to $\cS'(\bbR)$ by $\wti \cF$ $($with
$\wti \cF : \cS'(\bbR) \to \cS'(\bbR)$ a homeomorphism\,$)$, one obtains the scale of Sobolev
spaces via
\begin{equation} \lb{3.10}
H^s(\bbR) = \wti \cF L^2_s(\bbR), \quad s \in \bbR, \quad L^2(\bbR) = \cF L^2 (\bbR),
\end{equation}
and hence,
\begin{align}
\cF G_0 \cF^{-1} &= (T_0 + I)^{1/2} : H^1(\bbR) \to L^2(\bbR), \, \text{ bijectively, }     \no \\
\wti \cF \wti G_0 \wti \cF^{-1} &= \big(\wti T_0 + \wti I \big)^{1/2} : H^s(\bbR) \to H^{s-1}(\bbR),
\, \text{ bijectively, } \, \; s \in \bbR,    \lb{3.11} \\
\wti \cF \big(\wti G_0\big)^{-1} \wti \cF^{-1} &= \big(\wti T_0 + \wti I \big)^{-1/2} : H^s(\bbR)
\to H^{s+1}(\bbR), \, \text{ bijectively, } \, \;  s \in \bbR.   \no 
\end{align}
Here $T_0$ denotes the one-dimensional Laplace operator in $L^2(\bbR)$,
\begin{align}
& (T_0 f)(x) = - f''(x) \, \text{ for a.e.~$x\in \bbR$},   \no \\
& f \in \dom(T_0) = \big\{g \in L^2(\bbR) \, \big| \, g, g' \in AC_{loc}(\bbR); g'' \in  L^2(\bbR)\big\}
\lb{3.13} \\
& \hspace*{1.98cm} = H^2(\bbR).    \no
\end{align}
$($One observes that by a well-known Kolmogorov-type inequality, $f \in \dom (T_0)$ in \eqref{3.13} implies $f' \in L^2(\bbR)$ and hence leads to $H^2(\bbR)$.$)$

$T_0$ permits the extension $\wti T_0$ of $T_0$ being defined on the entire Sobolev scale according to \eqref{3.11},
\begin{equation}
\big(\wti T_0 + \wti I\big) : H^s(\bbR) \to H^{s-2}(\bbR) \, \text{ is a unitary map, } \; s \in \bbR.
\end{equation}
The special case $s=1$ yields the familiar fact,
\begin{equation}
\big(\wti T_0 + \wti I\big) : H^1(\bbR) \to H^{-1}(\bbR) = \big(H^1(\bbR)\big)^* \,
\text{ is a unitary map.}
\end{equation}
More generally, one obtains that
\begin{equation}
\big(\wti T_0 + \wti I\big)^{r/2} : H^s(\bbR) \to H^{s-r}(\bbR) \, \text{ is a unitary map, } \; r, s \in \bbR.
\end{equation}

One also observes that
\begin{align}
& H^0(\bbR) = L^2(\bbR), \quad \big(H^s(\bbR)\big)^* = H^{-s}(\bbR), \quad s \in \bbR,  \\
\begin{split}
& \, \cS(\bbR) \subsetneqq H^s(\bbR) \subsetneqq H^{s'}(\bbR) \subsetneqq L^2(\bbR) \subsetneqq H^{-s'}(\bbR) \subsetneqq H^{-s}(\bbR) \subsetneqq \cS'(\bbR),   \\
& \hspace*{9.07cm} s > s' > 0.
\end{split}
\end{align}

Thus, introducing
\begin{align}
& A_0 \colon = T_0 + I,    \no \\
& \wti A_0 \colon = \wti T_0 + \wti I,   \\
& \cH_r(A_0) = H^r(\bbR), \quad r \in [0,\infty),   \no 
\end{align}
identifies the scale of Sobolev spaces $H^r(\bbR)$ with the left-definite spaces $\cH_r(A_0)$, $r \in [0,\infty)$. The $($self-adjoint\,$)$ left-definite operators $A_{0,r}$ in $\cH_r(A_0) = H^r(\bbR)$ are then explicitly given by
\begin{align}
A_{0,0} &= A_0 = \cF (G_0)^2 \cF^{-1}\big|_{H^2(\bbR)},    \no \\
A_{0,r} &= \wti A_0\big|_{\cH_{r+2}(A_0)} = \wti A_0\big|_{H^{r+2}(\bbR)} = A_0\big|_{\cH_{r+2}(A_0)} = A_0\big|_{H^{r+2}(\bbR)}     \\
&= \wti \cF (\wti G_0)^2 \wti \cF^{-1}\big|_{H^{r+2}(\bbR)}, \quad r \in (0,\infty).  \no
\end{align}
\end{example}

Without going into more details, it is clear that the Fourier transform arguments in Example \ref{e3.1} extend to the multi-dimensional case in $L^2(\bbR^n)$, where
\begin{align}
\begin{split}
& A_{0,n} = - \Delta_n + I_{L^2(\bbR^n)}, \quad \dom(A_0) = H^2(\bbR^n), \quad n \in \bbN,  \\
& \cH_r(A_{0,n}) = H^r(\bbR^n), \quad r \in [0,\infty),
\end{split}
\end{align}
with $- \Delta_n$ denoting (minus) the Laplacian on $\bbR^n$, $n \in \bbN$. 

For various background material on fractional Sobolev spaces on $\bbR^n$, $n \in \bbN$, and their interpolation theory we refer, for instance, to \cite[Ch.~7]{AF03}, \cite{C-WHM15}, \cite[Ch.~3]{EE23}, \cite[Ch.~3]{HT08}, \cite[Chs.~2, 5--7, 10]{Le23}, \cite[Ch.~3]{Mc00}, \cite[Ch.~2]{Tr78}, 

Next, we turn to problems with purely discrete spectrum and start by recalling an elementary example from \cite{LW02}:

\begin{example} \lb{e3.2a}
Consider the self-adjoint operator $A$ in the Hilbert space $\ell^2(\bbN)$ given by
\begin{align}
\begin{split}
&(A f)_n = n f_n,  \quad n \in \bbN,   \\
& f \in \dom(A) = \big\{g = \{g_n\}_{n\in \bbN} \in \ell^2(\bbN) \, \big| \, \{n g_n\}_{n\in \bbN} \in \ell^2(\bbN)\big\},     \lb{3.24}
\end{split}
\end{align}
where, in obvious notation, elements $h \in \ell^2(\bbN)$ are represented as a complex-valued sequence $h =  \{h_n\}_{n\in \bbN}$. Thus, \eqref{3.24} represents a spectral representation of $A$ and one concludes that
\begin{equation}
A \geq I_{\ell^2(\bbN)}, \quad \sigma(A) = \bbN, \quad \text{$\sigma(A)$ is simple}.
\end{equation}
Moereover,
\begin{equation}
\dom\big(A^r\big) = \big\{g = \{g_n\}_{n\in \bbN} \in \ell^2(\bbN) \, \big| \, \big\{n^r g_n\big\}_{n\in \bbN} \in \ell^2(\bbN)\big\}, \quad r \in \bbR,
\end{equation}
and hence
\begin{align}
\begin{split}
\cV_r(A) = \dom\big(A^{r/2}\big) = \big\{g = \{g_n\}_{n\in \bbN} \in \ell^2(\bbN) \, \big| \, \big\{n^{r/2} g_n\big\}_{n\in \bbN} \in \ell^2(\bbN)\big\},&   \\
r \in [0,\infty),&
\end{split}
\end{align}
and
\begin{align}
\begin{split}
& \cH_0 = \ell^2(\bbN),   \\
& \cH_r(A) = \big(\cV_r(A), (\dott,\dott)_{\cH_r(A)}
= \big(A^{r/2} \dott, A^{r/2} \dott\big)_{\ell^2(\bbN)}\big), \quad r \in (0,\infty).
\end{split}
\end{align}
Finally, the self-adjoint operator $A_r = A\big|_{\cV_{r+2}(A)}$ in $\cH_r(A)$, $r \in [0,\infty)$, is of the type
\begin{align}
& (A_r f)_n = n f_n,  \quad n \in \bbN,   \no \\
& f \in \dom(A_r) = \big\{g = \{g_n\}_{n\in \bbN} \in \ell^2(\bbN) \, \big| \, \big\{n^{(r/2)+1} g_n\big\}_{n\in \bbN} \in \ell^2(\bbN)\big\},     \\
& \hspace*{8.55cm} r \in [0,\infty),    \no
\end{align}
and
\begin{equation}
\sigma(A_r) = \bbN, \quad r \in [0,\infty).
\end{equation}
\end{example}

\begin{example} \lb{e3.2}
Let $A$ satisfy the assumptions made in Hypothesis \ref{h2.1} but assume in addition that $A$ has only discrete and simple spectrum: $\sigma(A)=\sigma_d(A)$ and $\sigma_{ess}(A)=\emptyset$. For each $\lambda\in\sigma(A)$, let $\phi_\lambda$ denote the corresponding normalized eigenvector such that $\{\phi_\lambda\}_{\lambda\in\sigma(A)}$ is an orthonormal basis of eigenvectors. Let 
\begin{align}
U: \begin{cases}
\cH \to \ell^2(\sigma(A)),  \\
f \mapsto \left(U f\right)(\lambda) = (\phi_\lambda,f)_{\cH},
\end{cases}
\end{align}
be the unitary operator with inverse $U^{-1}=U^*$ given by
\begin{equation}
U^*: \begin{cases}
\ell^2(\sigma(A)) \to \cH,    \\
g \mapsto U^*g=\sum_{\lambda\in\sigma(A)} g(\lambda)\phi_\lambda.
\end{cases}
\end{equation}
Then, $U AU^*$ equals
\begin{align}
U A U^*: \begin{cases} \ell^2(\sigma(A)) \to \ell^2(\sigma(A)), \\
h \mapsto \left(U A U^*h\right)(\lambda)=\lambda h(\lambda).
\end{cases}
\end{align}
Thus, $U A U^*$ is a special case of the model operator $T$ described in \eqref{2.6a}, with $\mu$ being the counting measure on $\sigma(A)$. In particular, all results in Theorem \ref{t2.3} now extend to this special situation.
\end{example}

We continue with the particular case of the Laplacian on the interval $(0, 2 \pi)$ with boundary conditions of periodic-type:

\begin{example} \lb{e3.3}
Let $\phi \in (0,2\pi)\backslash\{\pi\}$, and consider in $L^2((0,2\pi))$,
\begin{align}
\begin{split}
& (A_{\phi} f)(x) = - f''(x), \quad x \in (0,2\pi),   \\
& f \in \dom(A_{\phi})= \big\{g\in H^2((0,2\pi)) \, \big| \, g(0)=e^{i\phi}g(2\pi), \, g'(0)=e^{i\phi}g'(2\pi)\big\}.
\end{split}
\end{align}
Then $A_{\phi}=P_{\phi}^2$, where $P_{\phi}$, the ``momentum operator'', is given by
\begin{align}
\begin{split}
& (P_{\phi} f)(x) = i f'(x), \quad x \in (0,2\pi),  \\
& f \in \dom(P_{\phi})=\big\{g\in H^1((0,2\pi)) \, \big| \, f(0)=e^{i\phi}f(2\pi)\big\},
\end{split}
\end{align}
with simple, discrete spectrum
\begin{equation}
\sigma(P_{\phi}) = \sigma_d(P_{\phi}) = \{n - [\phi/(2 \pi)]\}_{n \in \bbZ},
\end{equation}
and corresponding $($normalized\,$)$ eigenfunctions
\begin{equation}
\psi_{\phi,n}(x)=(2\pi)^{-1/2}\exp(i([\phi/(2\pi)]-n)x), \quad x \in (0,2 \pi), \; n \in \bbZ.
\end{equation}
As a result, $\{\psi_{\phi,n}\}_{n \in \bbZ}$ constitutes an orthonormal basis in $L^2((0,2\pi))$.
In addition,
\begin{equation}
A_{\phi}^{1/2} = |P_{\phi}|, \quad \dom\big(A_{\phi}^{1/2}\big) = \dom(|P_{\phi}|) = \dom(P_{\phi}).
\end{equation}
Due to the hypothesis $\phi \in (0,2\pi)\backslash\{\pi\}$, also $A_{\phi}$ has simple, discrete spectrum given by
\begin{equation}
\sigma(A_{\phi})=\sigma_{d}(A_{\phi})=\big\{|n-[\phi/(2\pi)]|^2\big\}_{n \in \bbZ},
\end{equation}
with corresponding $($normalized\,$)$ eigenfunctions $\psi_{\phi,n}$,
\begin{equation}
A_{\phi}\psi_{\phi,n} = |n-[\phi/(2\pi)]|^2 \psi_{\phi,n}, \quad n \in \bbZ.
\end{equation}

Next, we introduce the unitary map $\cF_\phi$ given by
\begin{align}
\cF_\phi: \begin{cases} L^2((0,2\pi))\to \ell^2(\Z),  \\
f \mapsto (\cF_\phi f)(n)= (\psi_{\phi,n},f)_{L^2((0,2\pi))},
\end{cases}
\end{align}
with inverse $\cF_\phi^{-1}=\cF_\phi^*$ given by
\begin{align}
\cF_\phi^{-1} : \begin{cases} \ell^2(\Z) \to L^2((0,2\pi)),    \\
g \mapsto \big(\cF_{\phi}^{-1}g\big)(\dott)=\sum_{n\in\Z} g(n)\psi_{\phi,n}(\dott).
\end{cases}
\end{align}
Denoting by $M_\phi$ the following maximally defined operator of multiplication in $\ell^2(\Z)$,
\begin{align}
\begin{split}
& (M_\phi f)(n)= \big(1+|n-[\phi/(2\pi)]|^2\big)^{1/2}f(n), \quad n \in \bbZ, \\
& f \in \dom(M_\phi) = \Big\{g\in\ell^2(\Z) \, \Big| \, \big(1+|\dott - [\phi/(2\pi)]|^2\big)^{1/2}g \in \ell^2(\Z)\Big\},
\end{split}
\end{align}
these operators are clearly strictly positive,
\begin{equation}
M_\phi \geq \Bigg(1+\bigg[\min\bigg\{\frac{\phi}{2\pi},1-\frac{\phi}{2\pi}\bigg\}\bigg]^2\Bigg)^{1/2}
I_{\ell^2(\Z)} \geq I_{\ell^2(\Z)},
\end{equation}
and hence boundedly invertible. Moreover, one has
\begin{equation}
\cF_\phi (I + A_{\phi})^{1/2}\cF^*_\phi=M_\phi.     \lb{3.38}
\end{equation}

Analogous to Example \ref{e3.1}, we now introduce the following scale of Hilbert spaces
\begin{equation}
\ell_{\phi,s}^2(\Z)=\bigg\{f:\Z\rightarrow\C \, \bigg| \, \sum_{n\in\Z} \big(1+|n-[\phi/(2\pi)]|^2\big)^s
|f(n)|^2<\infty\bigg\},\quad s\in\R,
\end{equation}
with corresponding inner product given by
\begin{equation} \lb{3.40}
(f,g)_{\ell^2_{\phi,s}(\Z)}:=\sum_{n\in\Z} \big(1+|n-[\phi/(2\pi)]|^2\big)^s \, \overline{f(n)}g(n).
\end{equation}

Identifying
\begin{equation}
\ell_{\phi,0}^2(\Z)=\ell^2(\Z)=(\ell^2(\Z))^*=(\ell_{\phi,0}^2(\Z))^*,
\end{equation}
which is independent of $\phi$, one obtains the following chain of Hilbert spaces
\begin{equation}
\ell_{\phi,s_2}^2(\Z)\subseteq \ell_{\phi,s_1}^2(\Z)\subseteq \ell^2(\Z)\subseteq \ell_{\phi,(-s_1)}^2(\Z)\subseteq \ell_{\phi,(-s_2)}^2(\Z), \quad s_2>s_1>0,
\end{equation}
where we identify
\begin{equation}
\big(\ell_{\phi,s}^2(\Z)\big)^*=\ell^2_{\phi,-s}(\Z).
\end{equation}

Mimicking the analysis in \eqref{3.5}, we extend $M_\phi$ to an operator $\wti{M}_\phi$ acting on the entire scale of Hilbert spaces, $\ell_{\phi,s}^2(\Z)$, $s\in\R$, such that
\begin{equation}
\wti{M}_\phi: \ell^2_{\phi,s}(\Z)\rightarrow \ell^2_{\phi,s-1}(\Z), \quad \big(\wti{M}_\phi\big)^{-1}: \ell^2_{\phi,s}(\Z)\rightarrow\ell^2_{\phi,s+1}(\Z),
\end{equation}
and by applying the inverse transformation $\cF_\phi^{-1}$ of $\ell_{\phi,s}^2(\Z)$, one obtains the corresponding $s$-th left-definite spaces $\cH_s(A_{\phi})$, $s \in [0,\infty)$, given by
\begin{equation}
\cH_s(A_{\phi})=(\cV_s(A_{\phi}),(\dott,\dott)_{\cH_s(A_{\phi})}).
\end{equation}
Here,
\begin{align}
\cV_s(A_{\phi})&=\dom\big(A_{\phi}^{s/2}\big)      \\
&= \bigg\{f\in L^2((0,2\pi)) \, \bigg| \,
\sum_{n\in\Z}|n-[\phi/(2\pi)]|^{2s}|( \psi_{\phi,n},f)_{L^2((0,2\pi))}|^2<\infty\bigg\},    \no \\
& \hspace*{8.85cm} s \in [0,\infty),   \no 
\end{align}
and
\begin{align}
(f,g)_{\cH_s(A_{\phi})} &= \big(A_{\phi}^{s/2}f,A_{\phi}^{s/2}g\big)_{L^2((0,2\pi))}     \no \\
&=\sum_{n\in\Z}|n-[\phi/(2\pi)]|^{2s} (f,\psi_{\phi,n})_{L^2((0,2\pi))}(\psi_{\phi,n},g)_{L^2((0,2\pi))},   \\
& \hspace*{7.2cm} s \in [0,\infty).    \no 
\end{align}
\end{example}

\begin{remark} \lb{r3.4}
In the exceptional cases $\phi \in \{0,\pi\}$, $A_{\phi}$ has twice degenerate eigenvalues (with exception of the lowest periodic eigenvalue $0$ for $\phi = 0$), but the analogue of Example \ref{e3.3} remains valid. \hfill $\diamond$
\end{remark}

\begin{remark} \lb{r3.5} ${}$
$(i)$ If $s \in \bbR \backslash \{0\}$, $\phi\in(0,2\pi) \backslash \{\pi\}$ one has
\begin{align}
\begin{split}
&\bigg\{f:\Z \to \C \, \bigg| \, \sum_{n\in\Z} \big(1+|n|^2\big)^s \, |f(n)|^2<\infty\bigg\}     \\
&\quad=\bigg\{f:\Z \to \C \, \bigg| \, \sum_{n\in\Z} \big(1+|n-[\phi/(2\pi)]|^2\big)^s \, |f(n)|^2<\infty\bigg\},
\end{split}
\end{align}
as an equality of vector spaces. However, the inner products defined in \eqref{3.40} are different and consequently $\ell^2_{s,\phi_1}(\Z)\neq \ell^2_{s,\phi_2}(\Z)$ for $\phi_1\neq\phi_2$. \\[1mm]
$(ii)$ In Appendix \ref{sA}, we provide a more explicit description of the linear spaces $\cV_s(A_{\phi})=\dom\big(A_{\phi}^{s/2}\big)$, $s\in(0,1)$. In particular, we will show that if $s \in (0,1/2)$, then $\dom\big(A_{\phi_1}^{s/2}\big) =\dom\big(A_{\phi_2}^{s/2}\big)=H^s((0,2\pi))$, $\phi_1,\phi_2\in[0,2\pi)$. But we emphasize that this does not imply equality of the corresponding Hilbert spaces $\cH_s(A_{\phi_1}), \cH_s(A_{\phi_2})$ as the corresponding inner products differ.
\hfill $\diamond$
\end{remark}

\begin{example}[Laplacian on the half-line] \lb{e3.7} Consider the minimal closed realization of $-d^2/dx^2$ in $L^2((0,\infty))$,
\begin{equation}
(B_{min}f)(x)=-f''(x),\quad x\in(0,\infty), \quad f\in\dom(B_{min}) = H^2_0((0,\infty)).
\end{equation}

Its nonnegative self-adjoint extensions are parametrized by
\begin{align}
\begin{split}
& (B_{\alpha} f)(x) = - f''(x), \quad \alpha \in [\pi/2,\pi], \; x\in(0,\infty),     \\
& f\in\dom(B_{\alpha}) = \big\{g\in H^2((0,\infty)) \, \big| \, \sin(\alpha) g'(0)+\cos(\alpha) g(0)=0\big\},
\end{split}
\end{align}
and for the rest of this example we always assume
\begin{equation}
\alpha \in [\pi/2,\pi], \, \text{ to guarantee that } \, B_{\alpha} \geq 0.    \lb{3.59}
\end{equation}
Here, the case $\alpha=\pi$ represents the Dirichlet boundary condition at the origin. One notes that the Dirichlet realization $B_\pi$ is the Friedrichs extension of $B_{min}$, while the Neumann realization $B_{\pi/2}$ is its Krein--von Neumann extension.
The operators $B_{\alpha}$ can be diagonalized via the unitary generalized Fourier or eigenfunction  transform $\cF_{\alpha}$,
\begin{equation}
\cF_{\alpha} : \begin{cases} L^2((0,\infty))\rightarrow L^2((0,\infty); d\rho_{\alpha}),    \\
h\mapsto \widehat{h}_{\alpha}(\dott)=\slim_{R\uparrow\infty}\int_0^R dx\,\phi_{\alpha}(\dott,x)h(x),
\end{cases}
\end{equation}
where $\slim$ refers to the strong limit in $L^2((0,\infty);d\rho_{\alpha})$, and its inverse
\begin{equation}
\cF_\al^{-1}: \begin{cases}  L^2((0,\infty);d\rho_{\alpha}) \to L^2((0,\infty)),   \\
\hatt h \mapsto \textstyle\slim_{\mu\uparrow\infty}\int_0^{\mu} d\rho_{\al}(\la) \, \phi_\al(\la,\dott)\hatt h(\la),
\end{cases}
\end{equation}
where $\slim$ denotes the one in $L^2((0,\infty))$. Here the functions $\phi_{\alpha}$ and
$\rho_{\alpha}$ can be chosen to be of the form
\begin{align}
\begin{split}
\phi_{\alpha}(z,x)=-\sin(\alpha)\cos(z^{1/2}x)+\cos(\alpha)z^{-1/2}\sin(z^{1/2}x),& \\
z \in \bbC, \; x \in (0,\infty),&
\end{split}
\end{align}
and
\begin{align}
& \rho_{\alpha}(\lambda)= \chi_{[0,\infty)}(\lambda) \begin{cases}
\frac{2}{\pi}\lambda^{1/2}, &\alpha=\pi/2,   \\
\frac{2}{\pi}\frac{1}{[\sin(\alpha)]^2}\left[\lambda^{1/2}-\cot(\alpha)\arctan\left(\lambda^{1/2}/\cot(\alpha)\right)\right], &\alpha\in(\pi/2,\pi),   \\
\frac{2}{3\pi}\lambda^{3/2},  &\alpha=\pi,
\end{cases}   \no \\
& \hspace*{10.2cm} \lambda \in \bbR.    \lb{3.63}
\end{align}
The diagonalization of $B_{\alpha}$ then takes on the following explicit form: Suppose $F \in C(\bbR)$, then
\begin{equation}
\cF_{\alpha} F(B_{\alpha}) \cF_{\alpha}^{-1} = M_F \, \text{ in } \, L^2((0,\infty); d\rho_{\alpha}).
\end{equation}
Here $M_G$ denotes the maximally defined operator of multiplication by the $d\rho_{\alpha}$-measurable function $G$ in $L^2((0,\infty); d\rho_{\alpha})$,
\begin{align}
\begin{split}
&(M_G \hatt h)(\la) = G(\la) \hatt h(\la) \;\text{ for $d\rho_{\alpha}$-a.e.~$\la\in (0,\infty)$},     \\
& \, \hatt h\in\dom(M_G)=\big\{\hatt k\in L^2((0,\infty);d\rho_{\alpha})  \, \big| \,  G\hatt k \in
L^2((0,\infty); d\rho_{\alpha})\big\}.     \lb{3.65}
\end{split}
\end{align}
Thus, in addition to \eqref{3.59} one concludes that\footnote{If $\alpha\in(0,\pi/2)$, then there is exactly one eigenvalue below the ac-spectrum $[0,\infty)$ of $B_{\alpha}$ at $-[\cot(\alpha)]^2$.}
\begin{equation}
\sigma(B_{\alpha})=\sigma_{ac}(B_{\alpha})=[0,\infty), \quad \sigma_{sc}(B_{\alpha})=\sigma_p(B_{\alpha})=\emptyset.
\end{equation}

Next, we denote by $M_{\alpha}$ the maximally defined operator of multiplication by the the function
$(1 + \dott)^{1/2}$ in $L^2((0,\infty);d\rho_{\alpha})$, given by
\begin{align}
\begin{split}
& (M_{\alpha} f)(\lambda)=(1+\lambda)^{1/2}f(\lambda),    \lb{3.67} \\
& f\in\dom(M_{\alpha}) = \big\{g\in L^2((0,\infty); d\rho_{\alpha}) \, \big| \, (1+\dott\,)^{1/2}g \in L^2((0,\infty);d\rho_{\alpha}) \big\}.
\end{split}
\end{align}
Then
\begin{equation}
\cF_{\alpha}(I+B_{\alpha})^{s/2}\cF_{\alpha}^{-1}=(M_{\alpha})^s, \quad \alpha \in [\pi/2,\pi], \; s\in(0,\infty).
\end{equation}

At this point, everything is analogous to Example \ref{e3.3}, following \eqref{3.38}, with the only modification that we consider the operator
\begin{equation}
B_{\alpha,1}:= I+B_{\alpha},  \quad \alpha\in[\pi/2,\pi],
\end{equation}
to ensure that $B_{\alpha,1} \geq I$. This naturally leads to the following scale of left-definite spaces for $s \in [0,\infty)$,
\begin{equation}
\cH_s(B_{\alpha,1})=(\cV_s(B_{\alpha,1}), (\dott,\dott)_{\cH_s(B_{\alpha,1})}),
\quad \alpha\in[\pi/2,\pi], \; s \in [0,\infty),
\end{equation}
where
\begin{align}
& \cV_s(B_{\alpha,1}) = \dom\big(B_{\alpha,1}^{s/2}\big) = \big\{f\in L^2((0,\infty))\, \big|\,\cF_{\alpha} f\in\dom\big(M_{\alpha}^s\big)\big\},    \no \\
& (f,g)_{\cH_s(B_{\alpha,1})} = \big(B_{\alpha,1}^{s/2}f,B_{\alpha,1}^{s/2}g\big)_{L^2((0,\infty))}
= \big(M_{\alpha}^s\cF_{\alpha} f, M_{\alpha}^s\cF_{\alpha} g\big)_{L^2((0,\infty);d\rho_{\alpha})}    \no \\
& \hspace*{1.95cm}
= \int_{(0,\infty)}d\rho_{\alpha}(\lambda) \, (1+\lambda)^s\overline{(\cF_{\alpha} f)(\lambda)} (\cF_{\alpha} g)(\lambda),    \\
& \hspace*{5cm} \alpha\in[\pi/2,\pi], \; s \in [0,\infty).    \no
\end{align}
\end{example}

For more details regarding the operator $B_{\alpha}$ in Example \ref{e3.7} see, \cite[Example~6.4.2]{GNZ24}.

We provide a brief description of the structure of the domains $\dom\big(B_{\alpha,1}^{s/2}\big)$, $\alpha\in[\pi/2,\pi]$, $s\in(0,\infty)$, in Appendix \ref{sC}.

\begin{example}[The Bessel operator on $(0,\infty)$] \lb{e3.8}
Let $\gamma \in (0,\infty)$ and introduce the Bessel differential expression $\tau_{\gamma} $ by
\begin{equation}
\tau_{\gamma} = - \f{d^2}{dx^2} +\frac{\gamma^2-1/4}{x^2}, \quad x \in (0,\infty).
\end{equation}
Then $T_{\gamma, max}$, the $($closed\,$)$ maximal operator associated with $\tau_{\gamma}$ in $L^2((0,\infty))$, is given by
\begin{align}
&T_{\gamma,max} f = \tau_{\gamma} f,    \no
\\
& f \in \dom(T_{\gamma,max})=\big\{g\in L^2((0,\infty)) \, \big| \,g,g^{[1]}\in AC_{loc}((0,\infty));   \lb{4.2.1A} \\
& \hspace*{6.65cm}  \tau_{\gamma} g\in L^2((0,\infty))\big\}.   \no
\end{align}
In addition, $\tau_{\gamma}$ is in the limit point case at $\infty$ and in the limit circle case at $0$ for $\gamma \in (0,1)$ and in the limit point case at $0$ for $\gamma \in [1,\infty)$. As a consequence, we find it convenient to distinguish the cases $\gamma \in (0,1)$ and $\gamma \in [1,\infty)$ in the following: \\[1mm]
$\mathbf{(I)}$ The case $\gamma \in (0,1)$$:$ We introduce the principal $u_{\gamma,0}(0,\dott)$ and nonprincipal solution $\hatt u_{\gamma,0}(0,\dott)$ of $\tau_{\gamma} u =0$ via
\begin{equation}
u_{\ga,0}(0, x) = x^{(1/2) + \ga},  \quad
\hatt u_{\ga,0}(0, x) = (2 \gamma)^{-1} x^{(1/2) - \ga}, \quad  x \in (0,1).    \lb{13.3.40A}
\end{equation}
The generalized boundary values for $g \in \dom(T_{\gamma,max})$ at $x=0$ are then of the form
\begin{align}
\begin{split}
\wti g(0) &= \lim_{x \downarrow 0} g(x)\big/\big[(2 \gamma)^{-1} x^{(1/2) - \gamma}\big],  \\
\wti g^{\, \prime} (0) &= \lim_{x \downarrow 0} \big[g(x) - \wti g(0) (2 \gamma)^{-1} x^{(1/2) - \gamma}\big]
 \big/x^{(1/2) + \gamma},
 \end{split}
\end{align}
The corresponding minimal operator $T_{\gamma,min}$ then takes on the form
\begin{align}
\begin{split}
& T_{\gamma,min} f = \tau_{\gamma} f, \\
& f \in \dom(T_{\gamma,min})= \big\{g\in\dom(T_{\gamma,max})  \, \big| \, \wti g(a) = {\wti g}^{\, \prime}(a) =0
= \wti g(b) = {\wti g}^{\, \prime}(b)\big\},      \lb{23.13.2.27b}
\end{split}
\end{align}
and one recalls that $T_{\gamma,max}^* = T_{\gamma,min}$ and $T_{\gamma,min}^* = T_{\gamma,max}$ and $T_{\gamma,min} \geq 0$. Moreover, the Friedrichs extension $T_{\gamma,F}$ of $T_{\gamma,min}$ in $L^2((0,\infty))$ is of the form
\begin{align}
\begin{split}
& T_{\gamma,F} f = \tau_{\gamma} f, \quad \gamma \in (0,1),    \\
& f \in \dom(T_{\gamma,F})=\{g\in\dom(T_{\gamma,max}) \, | \, \wti g(0) = 0\}.
\end{split}
\end{align}

\noindent
$\mathbf{(II)}$ The case $\gamma \in [1,\infty)$$:$ In this case
\begin{equation}
\text{$T_{\ga,max}$ is self-adjoint in $L^2((0,\infty))$.}
\end{equation}

In the following it is convenient to treat both cases simultaneously and hence we introduce
\begin{equation}
T_{\gamma} = \begin{cases} T_{\gamma,F}, & \gamma \in (0,1), \\
T_{\ga,max}, & \gamma \in [1,\infty),
\end{cases}
\end{equation}
noting
\begin{equation}
T_{\gamma} \geq 0, \quad \gamma \in (0,\infty).    \lb{3.81}
\end{equation}

Then the operator $T_{\gamma}$ can be diagonalized via the unitary generalized Fourier or eigenfunction  transform $\cF_{\gamma}$, given by
\begin{equation}
\cF_{\gamma}:\begin{cases} L^2((0,\infty))\rightarrow L^2((0,\infty); d\rho_{\gamma}),    \\
h\mapsto \widehat{h}_{\gamma}(\dott)=\slim_{R\rightarrow\infty}\int_0^R dx\,\phi_{\gamma}(\dott,x)h(x),
\end{cases} \quad \gamma \in (0,\infty),
\end{equation}
and its inverse,
\begin{equation}
\cF_{\gamma}^{-1} \colon \begin{cases} L^2(\bbR;d\rho_{\gamma}) \to L^2((0,\infty))  \\[1mm]
\hatt g \mapsto \slim_{\mu\uparrow\infty} \int_0^{\mu} d\rho_{\gamma}(\la)\,
\phi_{\gamma}(\la,\dott) \hatt g(\la),
\end{cases} \quad \gamma \in (0,\infty).
\end{equation}
Here the functions $\phi_{\gamma}$ and $\rho_{\gamma}$ can be chosen to be of the
form\footnote{We chose the normalization in \eqref{3.83} and \eqref{3.84} so that there is consistency with \eqref{3.63} for $\gamma = 1/2$ and $\alpha = \pi$.}
\begin{equation}
\phi_{\gamma}(z,x)= (\pi/2)^{1/2} z^{- \gamma/2} x^{1/2} J_{\gamma}\big(z^{1/2} x\big), \quad
\gamma \in (0,\infty), \; z \in \bbC, \; x \in (0,\infty),    \lb{3.83}
\end{equation}
and
\begin{equation}
\rho_{\gamma}(\lambda) =
\f{\lambda^{\gamma+1}}{\pi(\gamma+1)} \chi_{[0,\infty)}(\lambda), \quad
\ga \in (0,\infty), \; \lambda\in\bbR,    \lb{3.84}
\end{equation}
with
\begin{equation}
J_{\nu} (\zeta) = (\zeta/2)^{\nu} \sum_{k \in \bbN_0}  \f{(-1)^k (\zeta/2)^{2k}}{k! \Gamma(\nu + k + 1)},
\quad \nu \in \bbC \backslash (- \bbN),
\end{equation}
the standard Bessel function $($cf.\ \cite[Ch.~9]{AS72}$)$.

The rest of this example follows that of Example \ref{e3.7} nearly verbatim: For instance, the diagonalization of $T_{\gamma}$ takes on the following explicit form: Suppose $F \in C(\bbR)$, then
\begin{equation}
\cF_{\gamma} F(T_{\gamma}) \cF_{\gamma}^{-1} = M_F \, \text{ in } \, L^2((0,\infty); d\rho_{\gamma})
\end{equation}
$($cf.\ \eqref{3.65}$)$. Thus, in addition to \eqref{3.81} one concludes that\
\begin{equation}
\sigma(T_{\gamma})=\sigma_{ac}(T_{\gamma})=[0,\infty), \quad \sigma_{sc}(T_{\gamma})=\sigma_p(T_{\gamma})=\emptyset, \quad \gamma \in (0,\infty).
\end{equation}

Next $($see, \eqref{3.67}$)$, we denote by $M_{\gamma}$ the maximally defined operator of multiplication by the the function
$(1 + \dott)^{1/2}$ in $L^2((0,\infty);d\rho_{\gamma})$, given by
\begin{align}
\begin{split}
& (M_\gamma f)(\lambda)=(1+\lambda)^{1/2}f(\lambda),    \\
& f \in \dom(M_\gamma) = \big\{g\in L^2((0,\infty); d\rho_{\gamma}) \, \big| \, (1+\dott\,)^{1/2}g \in L^2((0,\infty);d\rho_{\gamma}) \big\}.
\end{split}
\end{align}
Then
\begin{equation}
\cF_{\gamma} (I+T_{\gamma})^{s/2}\cF_{\gamma}^{-1}=(M_\gamma)^s, \quad \gamma \in (0,\infty), \; s \in [0,\infty),
\end{equation}
and upon introducing
\begin{equation}
T_{\gamma,1}:= I+T_{\gamma},  \quad \gamma \in (0,\infty),
\end{equation}
to ensure that $T_{\gamma,1} \geq I$, this naturally leads to the following scale of left-definite spaces for $s \in [0,\infty)$,
\begin{equation}
\cH_s(T_{\gamma,1})=(\cV_s(T_{\gamma,1}), (\dott,\dott)_{\cH_s(T_{\gamma,1})}),
\quad \gamma \in (0,\infty), \; s \in [0,\infty),
\end{equation}
where
\begin{align}
& \cV_s(T_{\gamma,1}) = \dom\big(T_{\gamma,1}^{s/2}\big) = \big\{f\in L^2((0,\infty))\, \big|\,\cF_{\alpha} f\in\dom\big(M_{\gamma}^s\big)\big\},   \no  \\
& (f,g)_{\cH_s(T_{\gamma,1})} = \big(T_{\gamma,1}^{s/2}f,T_{\gamma,1}^{s/2}g\big)_{L^2((0,\infty))}
= \big(M_{\gamma}^s\cF_{\gamma} f, M_{\gamma}^s\cF_{\gamma} g\big)_{L^2((0,\infty);d\rho_{\gamma})}    \no \\
& \hspace*{1.95cm}
= \int_{(0,\infty)}d\rho_{\gamma}(\lambda) \, (1+\lambda)^s\overline{(\cF_{\gamma} f)(\lambda)} (\cF_{\gamma} g)(\lambda),     \\
& \hspace*{5.1cm} \gamma \in (0,\infty), \; s \in [0,\infty).    \no 
\end{align}

In \cite{GP79} $($see also \cite{FM23}, \cite{GPS21}, \cite{Ka78}, \cite{KMVZZ18}, \cite{Ro85}$)$, it is shown with the help of Hardy's inequality that
\begin{equation}
\dom\big(T_{\gamma}^{1/2}\big)=H^1_0((0,\infty)), \quad \gamma \in (0,\infty),    \lb{3.92}
\end{equation}
and thus, by the same interpolation argument as for the Dirichlet Laplacian on the half-line $(0,\infty)$ in Example \ref{e3.8} $($cf.\ Appendix \ref{sC}$)$, one obtains
\begin{equation}
\dom\big(T_{\gamma}^{s/2}\big)=H^s_{0,0}((0,\infty)), \quad \gamma \in (0,\infty), \; s\in(0,1],
\lb{3.93}
\end{equation}
where the spaces $H^{s/2}_{0,0}((0,\infty))$ are described in Equations \eqref{C.4}--\eqref{C.6}.
\end{example}

For more details in connection with $T_{\gamma,min}$, $T_{\gamma,max}$, and $T_{\gamma}$ in Example \ref{e3.8}, see \cite{EK07}, \cite[Sects.~13.2, 13.4, 13.6, Examples~13.3.5, 13.5.1, 13.7.1]{GNZ24}, \cite{Ka78}, \cite{Ro85}. As in the case of $\dom\big(B_{\alpha,1}^{s/2}\big)$, $\alpha\in[\pi/2,\pi]$, we provide a brief description of the structure of the domains $\dom\big(T_{\gamma,1}^{s/2}\big)$, $\gamma \in (0,\infty)$, $s\in(0,\infty)$, in Appendix \ref{sC}.

While the bulk of the results mentioned in Example \ref{e3.8} extends to the special case $\gamma = 0$, that is, it extends to the borderline case of semiboundedness for $T_{\gamma,min}$, the fact \eqref{3.92} is no longer valid for $\gamma = 0$ (see, \cite{GP79}, \cite{Ka78}). Hence, the direct connection to Sobolev spaces is lost for $\gamma = 0$.

\section{The Hermite Operator, The Harmonic Oscillator, Fractional Sobolev Spaces, and Interpolation Theory} \lb{s4}

In this section we take an in-depth look at Hermite polynomials, the underlying Hermite operator $A_H$ in $L^2\big(\bbR; e^{-x^2}dx\big)$, and the unitarily equivalent harmonic oscillator operator $T_{HO}$ in $L^2(\bbR)$. Our main technique in establishing the underlying left-definite theory relies on interpolation theory of fractional Sobolev spaces and domains of (strictly) positive fractional powers of $T_{HO}$, certain commutation techniques, and a precise description of $\dom\big(T_{HO}^s\big)$ \big(and hence that of $\dom\big(A_H^s\big)$\big), $s \in (0,\infty)$.

The Hermite differential expression is of the type
\begin{align}
\begin{split}
\tau_H &= - \f{d^2}{dx^2} + 2x \f{d}{dx} + c   \\
&= \f{1}{e^{-x^2}} \bigg[- \f{d}{dx} e^{-x^2} \f{d}{dx} + c e^{-x^2}\bigg], \quad c \in (0,\infty), \; x \in \bbR,
\end{split}
\end{align}
which generates the self-adjoint and maximally defined Hermite operator $A_H$ in $L^2\big(\bbR; e^{-x^2}dx\big)$ given by
\begin{align}
& (A_H f)(x) = (\tau_H f)(x), \quad x \in \bbR,     \\
& f \in \dom(A_H) = \big\{g \in L^2\big(\bbR; e^{-x^2}dx\big) \, \big| \, g, g' \in AC_{loc}(\bbR); \,
(\tau_H g) \in L^2\big(\bbR; e^{-x^2}dx\big)\big\},    \no
\end{align}
since $\tau_H$ is in the limit point case at $\pm \infty$ in $L^2\big(\bbR; e^{-x^2}dx\big)$. One verifies that $A_H$ has a compact resolvent,
\begin{equation}
(A_H - z I)^{-1} \in \cB_{\infty}\big(L^2\big(\bbR; e^{-x^2}dx\big)\big), \quad z \in \rho(A_H),  \lb{4.3}
\end{equation}
with
\begin{equation}
\sigma(A_H) = \sigma_d(A_H) = \{2m +c\}_{m \in \bbN_0}, \quad \text{$\sigma(A_H)$ is simple}.
\end{equation}
In fact, because of \eqref{4.3}, the resolvent of $A_H$ actually lies in the trace ideal
\begin{equation}
(A_H - z I)^{-1} \in \cB_{p}\big(L^2\big(\bbR; e^{-x^2}dx\big)\big), \quad z \in \rho(A_H),  \lb{4.5}
\end{equation}
for each $p \in (1,\infty)$. The normalized eigenfunctions associated with the eigenvalues $2m +c$, $m \in \bbN_0$, are the Hermite polynomials,
\begin{equation}
H_m(x), \quad m \in \bbN_0, \; x \in \bbR,
\end{equation}
given explicitly by 
\begin{align}
\begin{split} 
H_m(x) &= m! \, 2^m \sum_{j=0}^{\lfloor m/2 \rfloor} \f{(-1)^j}{2^{2j} j! (m-2j)!} x^{m-2j}    \\
&= (-1)^m e^{x^2} \f{d^m e^{-x^2}}{dx^m}; \quad m \in \bbN_0, \; x \in \bbR.  
\end{split} 
\end{align}
(We recall that $\lfloor x \rfloor$ equals the greatest integer less than or equal to $x \in \bbR$, see also \eqref{1.27}.) Explicitly,
\begin{align} 
\begin{split} 
& H_0(x) = 1, \quad H_1(x) = 2x, \quad H_2(x) = 4 x^2 -2, \quad H_3(x) = 8 x^3 - 12 x,    \\
& H_4(x) = 16 x^4 - 48x^2 +12, \quad H_5(x) = 32x^5-160x^3+120x, \; \text{ etc.}
\end{split}
\end{align} 
In particular,
\begin{equation}
(H_m,H_{m'})_{L^2(\bbR; e^{-x^2}dx)} = \pi^{1/2} 2^m m! \, \delta_{m,m'}, \quad m, m' \in \bbN_0.
\end{equation}
One can show the identity (see \cite{ELW00}),
\begin{align}
(\tau_H^n f)(x) = \sum_{j=0}^n (-1)^j c_j(n,c) e^{x^2} \big(e^{-x^2} f^{(j)}(x)\big)^{(j)},
\end{align}
where
\begin{align}
\begin{split}
c_0(n,c) &= c^n,     \\
c_j(n,c) &= 2^{n-j} \sum_{m=0}^{n-1} \begin{pmatrix} n \\ m \end{pmatrix} S^{(j)}_{n-m} 2^{-m} c^m, \quad j \in \bbN,
\end{split}
\end{align}
and
\begin{equation}
S^{(j)}_{\ell} = \sum_{k=0}^j \f{(-1)^{j+k}}{j!} \begin{pmatrix} j \\ k \end{pmatrix} c^{\ell}, \quad
j, \ell \in \bbN_0,
\end{equation}
are the Stirling numbers of the second kind. One has $S^{(j)}_0 =0$, $j \in \bbN$, and
\begin{equation}
c_j(n,c) > 0, \quad j \in \bbN_0, \, n \in \bbN,     \lb{4.14}
\end{equation}
since $c \in (0,\infty)$. In addition, one has
\begin{align}
\begin{split}
\big(A_H^{n/2} f, A_H^{n/2} g\big)_{L^2(\bbR; e^{-x^2}dx)} = \sum_{j=0}^n c_j(n,c) \int_{\bbR}
e^{-x^2} dx \, \ol{f^{(j)}(x)} g^{(j)}(x),& \\
n \in \bbN, \; f, g \in \dom\big(A_H^{n/2}\big).     \lb{4.15}
\end{split}
\end{align}

Turning to eigenfunction expansions induced by $A_H$ next, one notes the following facts upon re-normalizing $H_m$ as follows, 
\begin{equation}
K_m(x) = [\pi^{1/2} 2^m m!]^{-1/2} H_m(x), \quad m \in \bbN_0, \; x \in \bbR,
\end{equation}
such that 
\begin{equation}
(K_m,K_{m'})_{L^2(\bbR; e^{-x^2} dx)} = \delta_{m,m'}, \quad m, m' \in \bbN_0.
\end{equation}
Then one obtains
\begin{align}
& I = I_{L^2(\bbR; e^{-x^2}dx)} = \sum_{m \in \bbN_0} (K_m, \dott)_{L^2(\bbR; e^{-x^2}dx)} K_m,     \\
\begin{split}
& h = \sum_{m \in \bbN_0} c_m(h) K_m, \quad
\{c_{m}(h) = (K_{m},h)_{L^2(\bbR; e^{-x^2}dx)}\}_{m \in \bbN_0} \in \ell^2(\bbN_0),  \\
& \hspace*{7.42cm} h \in L^2\big(\bbR; e^{-x^2}dx\big),
\end{split}   \\
& A_H f = \sum_{m \in \bbN_0} (2m+c) (K_m, f)_{L^2(\bbR; e^{-x^2}dx)} K_m,  \no \\
& f \in \dom(A_H) = \big\{g \in L^2\big(\bbR; e^{-x^2}dx\big) \, \big| \, \\
& \hspace*{2.75cm} \{(2m + c) (K_m,g)_{L^2(\bbR; e^{-x^2}dx)}\}_{m \in \bbN_0} \in \ell^2(\bbN_0)\big\}.   \no
\end{align}
More generally, for $F$ a complex-valued Borel function on $\bbR$,
\begin{align}
& F(A_H) f = \sum_{m \in \bbN_0} F(2m+c) (K_m, f)_{L^2(\bbR; e^{-x^2}dx)} K_m,    \no \\
& f \in \dom(F(A_H)) = \big\{g \in L^2\big(\bbR; e^{-x^2}dx\big) \, \big| \, \\
& \hspace*{3.3cm} \{F(2m + c) (K_m,g)_{L^2(\bbR; e^{-x^2}dx)}\}_{m \in \bbN_0} \in \ell^2(\bbN_0)\big\}.    \no
\end{align}
In particular, for $s \in \bbR$,
\begin{align}
& A_H^{s/2} f = \sum_{m \in \bbN_0} (2m+c)^{s/2} (K_m, f)_{L^2(\bbR; e^{-x^2}dx)} K_m,    \no \\
& f \in \dom\big(A_H^{s/2}\big) = \big\{g \in L^2\big(\bbR; e^{-x^2}dx\big) \, \big| \, \\
& \hspace*{3cm} \big\{(2m + c)^{s/2} (K_m,g)_{L^2(\bbR; e^{-x^2}dx)}\big\}_{m \in \bbN_0} \in \ell^2(\bbN_0)\big\}, \quad s \in \bbR,     \no
\end{align}
and one can rewrite \eqref{4.15} in the form
\begin{align}
\begin{split}
& \big(A_H^{n/2} f, A_H^{n/2} g\big)_{L^2(\bbR; e^{-x^2}dx)} = \sum_{j=0}^n c_j(n,c) \big(f^{(j)}, g^{(j)}\big)_{L^2(\bbR; e^{-x^2}dx)}    \\
& \quad = \sum_{m \in \bbN_0} (2m+c)^n (f,K_m)_{L^2(\bbR; e^{-x^2}dx)}
(K_m, g)_{L^2(\bbR; e^{-x^2}dx)}, \quad n \in \bbN.
\end{split}
\end{align}
It follows from \eqref{4.14} and \eqref{4.15} that as sets, $\dom\big(A_H^{n/2}\big)$ and the Gaussian weighted Sobolev space $H^n\big(\bbR; e^{-x^2}dx\big)$, $n \in \bbN$, coincide, that is, one obtains
\begin{align}
\cV_n(A_H) &= \dom\big(A_H^{n/2}\big) = H^n\big(\bbR; e^{-x^2}dx\big)     \no \\
&= \big\{g \in L^2\big(\bbR; e^{-x^2}dx\big) \, \big| \, \\
& \hspace*{6mm} \big\{(2m + c)^{n/2} (K_m,g)_{L^2(\bbR; e^{-x^2}dx)}\big\}_{m \in \bbN_0} \in \ell^2(\bbN_0)\big\}, \quad n \in \bbN.    \no
\end{align}
Hence, one obtains for the $n$-th left-definite space associated with $A_H$,
\begin{align} 
\begin{split} 
\cH_n(A_H) &= (\cV_n(A_H); (\dott, \dott)_{ \cH_n(A_H)})
= \big(\dom\big(A_H^{n/2}\big), \big(A_H^{n/2}\dott, A_H^{n/2} \dott \big)_{\cH}\big)     \\
&= \big(H^n(\bbR; e^{-x^2}dx), \big(A_H^{n/2}\dott, A_H^{n/2} \dott \big)_{\cH}\big), \quad n \in \bbN.
\end{split} 
\end{align}
The associated self-adjoint left-definite operator $A_{H,n}$ in $\cH_n(A_H)$, $n \in \bbN$, is then of the form
\begin{align}
& A_{H,n} f = A_H f,    \no \\
& f \in \dom(A_{H,n}) = H^{n+2}\big(\bbR; e^{-x^2}dx\big)    \\
& \quad = \big\{g \in L^2\big(\bbR; e^{-x^2}dx\big) \, \big| \,    \no \\
& \quad \hspace*{6mm} \big\{(2m + c)^{(n/2)+1} (K_m,g)_{L^2(\bbR; e^{-x^2}dx)}\big\}_{m \in \bbN_0} \in \ell^2(\bbN_0)\big\}, \quad n \in \bbN.     \no
\end{align}

At this point we take this a decisive step further and define the fractional Gaussian weighted Sobolev space $H^s\big(\bbR; e^{-x^2}dx\big)$ via
\begin{align}
\cV_s(A_H) &= \dom\big(A_H^{s/2}\big) := H^s\big(\bbR; e^{-x^2}dx\big)      \no \\
& = \big\{g \in L^2\big(\bbR; e^{-x^2}dx\big) \, \big|      \lb{4.25} \\
& \hspace*{6mm} \big\{(2m + c)^{s/2} (K_m,g)_{L^2(\bbR; e^{-x^2}dx)}\big\}_{m \in \bbN_0} \in \ell^2(\bbN_0)\big\}, \quad s \in [0,\infty).     \no
\end{align} 
This yields for the $s$-th left-definite space associated with $A_H$,
\begin{align}
\begin{split} 
\cH_s(A_H) &= (\cV_s(A_H); (\dott, \dott)_{ \cH_s(A_H)})
= \big(\dom\big(A_H^{s/2}\big), \big(A_H^{s/2}\dott, A_H^{s/2} \dott \big)_{\cH}\big)    \lb{4.26} \\
&= \big(H^s\big(\bbR; e^{-x^2}dx\big), \big(A_H^{s/2}\dott, A_H^{s/2} \dott \big)_{\cH}\big),
\quad s \in [0,\infty).
\end{split} 
\end{align}

The associated self-adjoint left-definite operator $A_{H,s}$ in $\cH_s(A_H)$, $s \in [0,\infty)$, is then of the form
\begin{align}
& A_{H,s} f = A_H f,    \no \\
& f \in \dom(A_{H,s}) = H^{s+2}\big(\bbR; e^{-x^2}dx\big)    \lb{4.28} \\
& \quad = \big\{g \in L^2\big(\bbR; e^{-x^2}dx\big) \, \big| \,    \no \\
& \quad \hspace*{6mm} \big\{(2m + c)^{(s/2)+1} (K_m,g)_{L^2(\bbR; e^{-x^2}dx)}\big\}_{m \in \bbN_0} \in \ell^2(\bbN_0)\big\}, \quad s \in [0,\infty).     \no
\end{align}

One notes that the (generalized) Fourier coefficients $c_m(g)=(K_m,g)_{L^2(\bbR; e^{-x^2}dx)}$, $m \in \bbN_0$, play precisely the role of the coefficients $g_n$, $n \in \bbN$, in Example \ref{e3.2a}.

After a thorough study of the harmonic oscillator operator and its fractional powers in the remainder of this section, we will be able to provide an additional characterization of $\cV_s(A_H)$ and $\cH_s(A_H)$, $s \in [0,\infty)$, involving fractional Sobolev spaces, at the end of Section \ref{s4}. 

Next, we turn to the harmonic oscillator operator $T_{HO}$ in $L^2(\bbR)$: The harmonic oscillator differential expression $\tau_{HO}$ is of the form
\begin{equation}
\tau_{HO} = - \f{d^2}{dx^2} + x^2, \quad x \in \bbR,
\end{equation}
which is in the limit point case at $\pm \infty$ with respect to $L^2(\bbR)$ and hence gives rise to the self-adjoint and maximally defined associated operator
\begin{align}
\begin{split} 
& (T_{HO} f)(x) = (\tau_{HO} f)(x), \quad x \in \bbR,    \lb{4.30} \\
& f \in \dom(T_{HO}) = \big\{g \in L^2(\bbR) \, \big| \, g, g' \in AC_{loc}(\bbR); \, (\tau_{HO} g) \in L^2(\bbR)\big\}
\end{split} 
\end{align}
in $L^2(\bbR)$. The unitary maps
\begin{align}
U \colon \begin{cases} L^2\big(\bbR; e^{-x^2} dx\big) \to L^2(\bbR), \\
v(x) \mapsto e^{- x^2/2} v(x), \end{cases} \quad
U^{-1} \colon \begin{cases} L^2(\bbR) \to L^2\big(\bbR; e^{-x^2} dx\big), \\
u(x) \mapsto e^{x^2/2} u(x), \end{cases}     \lb{4.32}
\end{align}
yield the following unitary equivalence of the shifted harmonic oscillator $T_{HO}$ and the Hermite operator $A_H$ in the form
\begin{equation}
U^{-1} [T_{HO} + (c-1) I] U = A_H.    \lb{4.33} 
\end{equation}
Thus,
\begin{equation}
\sigma(T_{HO}) = \sigma_d(T_{HO}) = \{2m + 1\}_{m \in \bbN_0}, \quad \text{$\sigma(T_{HO})$ is simple},
\end{equation}
and for each $p \in (1,\infty)$,
\begin{equation}
(T_{HO} - z I)^{-1} \in \cB_{p}\big(L^2(\bbR)\big), \quad z \in \rho(T_{HO}).    \lb{4.35}
\end{equation}
In particular, the normalized eigenfunctions $u_m$ associated with the eigenvalues $2m + 1$, $m \in \bbN_0$, of $T_{HO}$ are of the form
\begin{align} 
\begin{split} 
u_m(x) &= \big[\pi^{1/2} 2^m m!\big]^{-1/2} H_m(x) e^{-x^2/2}     \\
&= K_m(x) e^{-x^2/2} , \quad m \in \bbN_0, \; x \in \bbR,
\end{split} 
\end{align} 
such that
\begin{equation}
(u_m,u_{m'})_{L^2(\bbR)} = \delta_{m,m'}, \quad m, m' \in \bbN_0.
\end{equation}

In the following we will give an alternative description of the domains of fractional powers of the harmonic oscillator, $\dom(T_{HO}^\theta)$. We start with the following definition that will provide us with the appropriate scale of Hilbert spaces.

\begin{definition}
Let $V \in C^1(\bbR^n)$ such that $V(x)>0$ for all $x\in\bbR^n$. For $\theta \in (0,\infty)$, we define $H^\theta_V(\bbR^n)$ to be the closure of $C_0^\infty(\bbR^n)$ with respect to the norm
\begin{equation}
\|u\|_{H^\theta_V(\bbR^n)}:=\left(\|u\|^2_{H^\theta({\bbR ^n})}+ \big\|V^{\theta/2}u\big\|_{L^2(\bbR ^n)}^2\right)^{1/2}.
\end{equation}
Here, $H^\theta(\bbR^n)$ is the usual fractional Sobolev space given by
\begin{equation}
H^\theta(\bbR^n)=\bigg\{f\in L^2(\bbR^n)\bigg| \int_{\bbR^n} d^n\xi \, 
\big|\hatt{f}(\xi)\big|^2 \big(1+|\xi|^2\big)^{\theta}<\infty\bigg\},
\end{equation}
with inner product
\begin{equation*}
\langle f,g\rangle_{H^\theta(\bbR^n)}=\int_{\bbR^n} d^n\xi \, \overline{\hatt{f}(\xi)}\hatt{g}(\xi) 
\big(1+|\xi|^2\big)^{\theta}.
\end{equation*}
\end{definition}

\begin{lemma} \lb{lemma:4.2}
The space $H^\theta_V(\bbR^n)$ is given by\footnote{With some abuse of notation we denote the function $V$ and the operator of multiplication in $L^2(\bbR)$ by $V$ with the same symbol.}
\begin{equation}
H^\theta_V(\bbR^n)=H^\theta(\bbR^n) \cap \dom\big(V^{\theta/2}\big).
\end{equation}
\end{lemma}
\begin{proof}
First we show that if $f\in H_V^\theta(\bbR^n)$, then $f\in H^\theta(\bbR^n)\cap\dom\big(V^{\theta/2}\big)$. But $f\in H_V^{\theta}(\bbR^n)$ means that there is a sequence $\{f_k\}_{k\in\bbN}$ in $C_0^\infty(\bbR^n)$  such that
\begin{equation}
\lim_{k\rightarrow\infty}\|f_k-f\|_{H_V^\theta(\bbR^n)}=0.
\end{equation}
Then $f\in H^\theta(\bbR^n)\cap\dom\big(V^{\theta/2}\big)$ follows immediately from the facts that
\begin{align} 
\begin{split} 
& \|f_k-f\|_{H^\theta(\bbR^n)} \leq \|f_k-f\|_{H^\theta_V(\bbR^n)}, \\
&\|f_k-f\|_{L^2(\bbR^n)}+ \big\|V^{\theta/2}(f_k-f)\big\|_{L^2(\bbR^n)} 
\leq \|f_k-f\|_{H^\theta_V(\bbR^n)}. 
\end{split} 
\end{align}

In order to show that $H^\theta(\bbR^n)\cap\dom\big(V^{\theta/2}\big)\subseteq H^\theta_V(\bbR^n)$ we need to show that for every $f\in H^\theta(\bbR^n)\cap\dom\big(V^{\theta/2}\big)$ and every $\varepsilon>0$, there is a $g\in C_0^\infty(\bbR^n)$ such that $\|f-g\|_{H_V^\theta(\bbR^n)}<\varepsilon$.

From the proof of Lemma 6.67 in \cite{Le23}, there exists a sequence of compactly supported, non-negative and smooth functions $\phi_k\in C_0^\infty(\bbR^n)$ such that $\phi_k(x)=1$ for $|x|\leq k$, $\phi_k(x)=0$ for $|x|\geq(k+1)$, and $|\nabla \phi_k(x)|\leq C$, where $C$ is independent of $k \in \bbN$. Moreover, the $\phi_k$ can be chosen such that $\|\phi_k\|_\infty=1$. Defining $f_k:=\phi_k f$, it is shown in the proof of Lemma 6.67 in \cite{Le23} that $f_k\in H^\theta(\bbR^n)$ and that $\lim_{k\rightarrow\infty}\|f_k-f\|_{H^\theta(\bbR^n)}=0$. Next, let $0 \leq \eta\in C_0^\infty(\bbR^n)$ be the smooth bump function given by  
\begin{equation}
\eta(x) = \begin{cases} c e^{(|x|^2 - 1)^{-1}}, & |x| < 1, \\
0, & |x| \geq 1, \end{cases} \quad c =\bigg(\int_{\bbR^n} d^n x \, e^{(|x|^2 - 1)^{-1}}\bigg)^{-1},
\end{equation}
and for $\delta\in(0,1)$ define $\eta_\delta(x):=\delta^{-n}\eta(x/\delta)$ as well as $f_{k,\delta}:=f_k*\eta_\delta$. One observes that $f_{k,\delta}\in C_0^\infty(\bbR^n)$ and $\supp(f_{k,\delta})\subseteq B_{k+1+\delta}(0)$. By Theorems 6.62 and 6.66 in \cite{Le23}, it follows that $\lim_{\delta\downarrow 0}\|f_{k,\delta}-f_k\|_{H^\theta(\bbR^n)}=0$ for any fixed $k$. Consequently, one can fix $k' \in \bbN$ sufficiently large and then $\delta'$ sufficiently small such that
\begin{equation}
\|f-f_{k',\delta'}\|_{H^\theta(\bbR^n)}\leq \|f-f_{k'}\|_{H^\theta(\bbR^n)}+\|f_{k'}-f_{k',\delta'}\|_{H^\theta(\bbR^n)}<\varepsilon.
\end{equation}

Next, consider
\begin{equation} \lb{4.50}
\big\|V^{\theta/2}(f_{k,\delta}-f)\big\|_{L^2(\bbR^n)}\leq \big\|V^{\theta/2}(f_{k}-f)\big\|_{L^2(\bbR^n)}
+ \big\|V^{\theta/2}(f_{k,\delta}-f_k)\big\|_{L^2(\bbR^n)}.
\end{equation}
Focusing on the first of the two terms on the right-hand side of this inequality yields, 
\begin{equation}
\big\|V^{\theta/2}(f_{k}-f)\big\|_{L^2(\bbR^n)}^2=\int_{\bbR^n} d^nx \, V(x)^\theta   |f(x)|^2(1-\phi_k)^2 \leq \int_{|x|\geq k}  d^nx \, V(x)^\theta   |f(x)|^2.
\end{equation}
Since $f\in \dom\big(V^{\theta/2}\big)$, this implies that there is a $k''\geq k'$ sufficiently large, $k'' \in \bbN$, such that
\begin{equation}
\big\|V^{\theta/2}(f_{k''}-f)\big\|_{L^2(\bbR^n)}<\varepsilon/3.
\end{equation}
Next, consider the second term in \eqref{4.50}:
\begin{align} 
\begin{split} 
\big\|V^{\theta/2}(f_{k'',\delta}-f_{k''})\big\|_{L^2(\bbR^n)}^2&=\int_{B_{k''+1+\delta}(0)} d^nx \, V(x)^\theta  |f_{k'',\delta}(x)-f_{k''}(x)|^2    \\
&\leq C_{V,\theta,k''} \|f_{k'',\delta}-f_{k''}\|_{L^2(\bbR^n)}, 
\end{split} 
\end{align}
where $C_{V,\theta,k''}:=\sup_{x\in B_{k''+2}(0)} \big\{V(x)^\theta  \big\} \in (0,\infty)$, since $V \in C^1(\bbR^n)$ by assumption. Choosing $0 < {\delta''}\leq \delta'$ sufficiently small such that 
$C_{V,\theta,k''} \|f_{k'',{\delta''}}-f_{k''}\|_{L^2(\bbR^n)}<\varepsilon/3$, one then gets
\begin{align}
&\|f_{k'',{\delta''}}-f\|_{H_V^\theta(\bbR^n)}  \no \\ 
&\leq \|f-f_{k'',{\delta''}}\|_{H^\theta(\bbR^n)} 
+ \big\|V^{\theta/2}(f_{k''}-f)\big\|_{L^2(\bbR^n)} 
+ \big\|V^{\theta/2}(f_{k'',{\delta''}}-f_{k''})\big\|_{L^2(\bbR^n)}   \no \\
&<\varepsilon/3+\varepsilon/3+\varepsilon/3=\varepsilon,  
\end{align}
finishing the proof.
\end{proof}

We will make use of the following weighted Gagliardo--Nirenberg inequality, which was proven in \cite{Li86}. However, we will follow the presentation in \cite[Theorem 1.2]{DS23}:

\begin{theorem} \lb{t4.3} 
Let $1\leq p,q,r<\infty$, $m\in\bbN_0$, $\kappa\in\bbN$, $(m/\kappa) \leq \theta \leq 1$, $[\kappa-m- (d/p)] \notin\bbN_0$, and $\alpha,\beta,\gamma\in\bbR$ be such that $\alpha>-d/p$, $\beta>-d/q$, $\gamma>-d/r$ satisfy the conditions
\begin{equation}
\frac{1}{r}-\frac{m-\gamma}{d}=\theta\left(\frac{1}{p}-\frac{\kappa-\alpha}{d}\right)+(1-\theta)\left(\frac{1}{q}+\frac{\beta}{d}\right)
\end{equation}
and
\begin{equation}
0\leq \theta\alpha+(1-\theta)\beta-\gamma\leq\theta \kappa-m.
\end{equation}
Then there exists a constant $C=C_{p,q,r,\alpha,\beta,\gamma,\theta}$ such that
\begin{equation}
\big\| |x|^\gamma\nabla^m f \big\|_{L^r(\bbR^d)}\leq C \big\||x|^\alpha\nabla^\kappa f \big\|_{L^p(\bbR^d)}^\theta \big\||x|^\beta f \big\|_{L^q(\bbR^d)}^{1-\theta}, \quad f\in C_0^\infty(\bbR^d). 
\end{equation}
\end{theorem}

\begin{corollary} \lb{coro:4.4}
For every $k\in\bbN$, there exist constants $C_k, C_{k}' \in (0,\infty)$ such that 
\begin{align} \lb{4.54}
\big\|x^{2k}f'' \big\|_{L^2(\bbR)} & \leq C_k' \big\|f^{(2k+2)}\big\|_{L^2(\bbR)}^{1/(1+k)}\big\|x^{2k+2}f\big\|_{L^2(\bbR)}^{k/(1+k)}     \\ 
& \leq C_k\big[\big\|f^{(2k+2)}\big\|_{L^2(\bbR)}+\big\|x^{2k+2}f\big\|_{L^2(\bbR)}\big], \quad 
f\in C_0^\infty(\bbR). \lb{4.55}
\end{align}
Similarly, one gets
\begin{align} \lb{4.56}
\begin{split} 
\left\|\frac{d^{2k}}{dx^{2k}}(x^{2}f)\right\|_{L^2(\bbR)} \leq& C_k' \big\|f^{(2k+2)}\big\|_{L^2(\bbR)}^{1/(1+k)} \big\|x^{2k+2}f\big\|_{L^2(\bbR)}^{k/(1+k)}\\
\leq& C_k \big[\big\|f^{(2k+2)}\big\|_{L^2(\bbR)}
+ \big\|x^{2k+2}f\big\|_{L^2(\bbR)}\big], \quad f\in C_0^\infty(\bbR).
\end{split} 
\end{align} 
\end{corollary}
\begin{proof}
Let $k\in\bbN$. From a direct calculation, it can be verified that the choice of parameters $d=1$, $p=q=r=2$, $m=2$, $\gamma=2k$, $\alpha=0$, $\kappa=2k+2$, $\beta=2k+2$, and $\theta=\tfrac{1}{k+1}$ satisfies all the assumptions in Theorem \ref{t4.3}, from which the inequality in \eqref{4.54} follows. The second inequality \eqref{4.55} is an application of Young's inequality for products. Next, we focus on showing \eqref{4.56}. By taking Fourier transforms we get
\begin{align}
\begin{split} 
\left\|\frac{d^{2k}}{dx^{2k}}(x^{2}f)\right\|_{L^2(\bbR)}=&\big\| \xi^{2k}\hatt{f}''\big\|_{L^2(\bbR)} 
\leq C_k \big[\big\|\hatt{f}^{(2k+2)}\big\|_{L^2(\bbR)}+\big\|\xi^{2k+2}\hatt{f}\big\|_{L^2(\bbR)}\big]\\ 
= & C_k \big[\big\|x^{2k+2}{f}\big\|_{L^2(\bbR)}+\big\|{f}^{(2k+2)}\big\|_{L^2(\bbR)}\big], 
\end{split} 
\end{align}
where we used inequality \eqref{4.55} to estimate $\big\| \xi^{2k}\hatt{f}''\big\|_{L^2(\bbR)}$.
\end{proof}

\begin{theorem} \lb{t4.5} Let $P=i\ d/dx$ and $X$ denote the momentum operator and operator of multiplication by $x$ defined on $C_0^\infty(\bbR)$, respectively. Then, for any $k\in\bbN$, there exist $a_k, b_k, s_k, t_k \in (0,\infty)$ such that
\begin{equation} \lb{4.59a}
P^{4k}+X^{4k}\leq a_k \big(P^2+X^2\big)^{2k}+b_k I \leq s_k \big(P^{4k}+X^{4k}\big)+t_k I
\end{equation}
in the quadratic form sense.
\end{theorem}
\begin{proof} We begin by showing the first inequality in \eqref{4.59a}, which we will prove by induction. For the base case $k=1$, consider
\begin{align}
\big(P^2+X^2\big)^2 &=P^4+X^2P^2+P^2X^2+X^4=X^4+P^4+2PX^2P 
+ \big[P,\big[P,X^2\big]\big]   \no \\
& \geq X^4+P^4-2 I, 
\end{align}
where we used that $PX^2P\geq 0$ in the sense of quadratic forms and $\big[P,\big[P,X^m\big]\big]=-m(m-1)X^{m-2}$, which is a consequence of the commutation relation $[P,X]=i I$. Later, we will also use $\big[X,\big[X,P^m]]=-m(m-1)P^{m-2}$. Hence, if $k=1$, then \eqref{4.59a} is satisfied with $a_1=1$ and $b_1=2$.

For the inductive step, assume that \eqref{4.59a} holds up to $k$. Multiplying both sides of \eqref{4.59a} by $P^2+X^2$ from the left and right yields
\begin{equation} \lb{4.61}
\big(P^2+X^2\big)\big(P^{4k}+X^{4k}\big)\big(P^2+X^2\big)\leq a_k\big(P^2+X^2\big)^{2(k+1)}+b_k\big(P^2+X^2\big)^2.
\end{equation}
First, we focus on the left-hand side of this inequality:
\begin{align}
&(P^2+X^2)(P^{4k}+X^{4k})(P^2+X^2)   \no \\
& \quad =P^{4(k+1)}+X^{4(k+1)}+X^2P^{4k}X^2+P^2X^{4k}P^2   \no \\ 
&\qquad+X^2P^{4k+2}+P^{4k+2}X^2+P^2X^{4k+2}+X^{4k+2}P^2 \no \\
& \quad \geq P^{4(k+1)}+X^{4(k+1)}+X^2P^{4k+2}+P^{4k+2}X^2+P^2X^{4k+2}+X^{4k+2}P^2  \no \\
& \quad = P^{4(k+1)}+X^{4(k+1)}+2 P X^{4k+2} P+\big[P,\big[P,X^{4k+2}\big]\big]+2XP^{4k+2}X   \no \\
& \qquad +\big[X,\big[X,P^{4k+2}\big]\big]  \no \\
& \quad \geq P^{4(k+1)}+X^{4(k+1)}+\big[P,\big[P,X^{4k+2}\big]\big]+\big[X,\big[X,P^{4k+2}\big]\big]   \no \\ 
& \quad = P^{4(k+1)}+X^{4(k+1)}-(4k+2)(4k+1)\big(P^{4k}+X^{4k}\big)   \no \\   
& \quad = \frac{1}{2}\big(P^{4(k+1)}+X^{4(k+1)}\big)    \lb{4.70}  \\
& \qquad +\left(\frac{1}{2}P^{4(k+1)}-(4k+2)(4k+1)P^{4k}\right)+\left(\frac{1}{2}X^{4(k+1)}-(4k+2)(4k+1)X^{4k}\right)    \no \\ 
& \quad \geq \frac{1}{2}\big(P^{4(k+1)}+X^{4(k+1)}\big)-c_k I,      \lb{4.71}
\end{align}
where $c_k>0$ is a constant such that
\begin{equation}
\frac{1}{2}t^{4(k+1)}-(4k+2)(4k+1)t^{4k}\geq -\frac{c_k}{2}, \quad t\in\bbR ,
\end{equation}
which ensures that the second and third term in \eqref{4.70} are each bounded from below by $-(c_k/2) I$,  thus resulting in the final estimate \eqref{4.71}. Applying this estimate to \eqref{4.61} then yields
\begin{equation}
\frac{1}{2}\big(P^{4(k+1)}+X^{4(k+1)}\big)-c_k I \leq a_k\big(P^2+X^2\big)^{2k}+b_k I.\lb{4.73}
\end{equation}
Next, we focus on the right-hand side of \eqref{4.73}. Here, we will use the fact that for any $f\in C_0^\infty(\bbR)$:
\begin{align} 
\begin{split} 
\Big( f, \big(P^2+X^2\big)^2f\Big)_{L^2(\bbR)} &=\|T_{HO}f\|_{L^2(\bbR)}^2\leq \big\|T_{HO}^{k+1}f\big\|_{L^2(\bbR)}^2     \\ 
& = \Big(f, (P^2+X^2)^{2(k+1)}f\Big)_{L^2(\bbR)},
\end{split}
\end{align}
where the inequality $\|T_{HO}f\|_{L^2(\bbR)}^2 \leq \big\|T_{HO}^{k+1}f\big\|_{L^2(\bbR)}^2$ follows from the fact that the spectrum of $T_{HO}$ is given by $\{(2n+1)\}_{n=0}^\infty$. Consequently, we may estimate $\big(P^2+X^2\big)^2 \leq \big(P^2+X^2\big)^{2(k+1)}$ for $k\in\bbN$. Thus, \eqref{4.73} can be estimated by
\begin{align} 
\begin{split} 
\frac{1}{2}\big(P^{4(k+1)}+X^{4(k+1)}\big)-c_k I & \leq a_k\big(P^2+X^2\big)^{2(k+1)}+b_k\big(P^2+X^2\big)^{2}   \\
&\leq (a_k+b_k)\big(P^2+X^2\big)^{2(k+1)},
\end{split}
\end{align}
which is equivalent to
\begin{equation}
\big(P^{4(k+1)}+X^{4(k+1)}\big) \leq  2(a_k+b_k)\big(P^2+X^2\big)^{2(k+1)}+2c_k I,
\end{equation}
and thus \eqref{4.59a} holds for $k+1$ with $a_{k+1}=2(a_k+b_k)$ and $b_{k+1}=2c_k$.

Next, we will show the second inequality in \eqref{4.59a}:
\begin{equation} \lb{4.77}
a_k\big(P^2+X^2\big)^{2k}+b_k I \leq s_k\big(P^{4k}+X^{4k}\big)+t_k I,
\end{equation}
once again by induction. 

For what follows, it will be useful to note that for any $k\in\bbN_0$ and any $f\in C_0^\infty(\bbR)$:
\begin{align}
& \big( f,\big(P^2X^{4k+2}+X^{4k+2}P^2\big)f\big)_{L^2(\bbR)}    \no \\
& \quad = 2\Re \big(f,P^2X^{4k+2}f\big)_{L^2(\bbR)} \no \\
& \quad \leq 2\big| \big(f,P^2X^{4k+2}f\big)_{L^2(\bbR)}\big|   \no \\
& \quad = 2\big|\big(X^{2k}P^2f,X^{2k+2}f\big)_{L^2(\bbR)}\big|    \no \\
& \quad \leq \big\|X^{2k}P^2f\big\|_{L^2(\bbR)}^2+\big\|X^{2k+2}f\big\|_{L^2(\bbR)}^2    \no \\
& \quad = \big(f, \big(P^2X^{4k}P^2+X^{4(k+1)}\big)f\big)_{L^2(\bbR)},
\end{align}
and consequently
\begin{equation} \lb{4.81}
P^2X^{4k+2}+X^{4k+2}P^2\leq P^2X^{4k}P^2+X^{4(k+1)}
\end{equation}
in the quadratic form sense. Similarly, one obtains
\begin{equation} \lb{4.82}
X^2P^{4k+2}+P^{4k+2}X^2\leq X^2P^{4k}X^2+P^{4(k+1)}.
\end{equation}
For the base case $k=1$, one has $a_1=1$ and $b_1=2$. One then gets
\begin{equation} \lb{4.83}
\big(P^2+X^2\big)^2+2 I=P^4+P^2X^2+X^2P^2+X^4+2 \leq 2 \big(P^4+X^4\big)+2 I,
\end{equation}
where we used \eqref{4.81} for $k=0$. This shows the base case with $s_1=2$ and $t_1=2$. For the inductive step, suppose \eqref{4.77} holds up to $k$. Multiplying both sides of \eqref{4.77} by $\big(P^2+X^2\big)$ then yields
\begin{align} \lb{4.84}
\begin{split} 
& a_k\big(P^2+X^2\big)^{2(k+1)}+b_k\big(P^2+X^2\big)^2    \\
& \quad \leq s_k\big(P^2+X^2\big)\big(P^{4k}+X^{4k}\big)\big(P^2+X^2\big)+t_k\big(P^2+X^2\big)^2.
\end{split} 
\end{align}
First, we focus on the term $\big(P^2+X^2\big)\big(P^{4k}+X^{4k}\big)\big(P^2+X^2\big)$:
\begin{align}
&\big(P^2+X^2\big)\big(P^{4k}+X^{4k}\big)\big(P^2+X^2\big)    \no \\
& \quad = P^{4(k+1)}+X^{4(k+1)}+X^2P^{4k}X^2+P^2X^{4k}P^2    \no \\ 
&\qquad +X^2 P^{4k+2}+P^{4k+2}X^2+P^2X^{4k+2}+X^{4k+2}P^2     \no \\
& \quad \leq 2 P^{4(k+1)}+2X^{4(k+1)}+2X^2P^{4k}X^2+2P^2X^{4k}P^2, \lb{4.88}
\end{align}
where the last inequality follows from \eqref{4.81} and \eqref{4.82}. It is now a consequence of Corollary \ref{coro:4.4}, that
\begin{align}
X^2P^{4k}X^2\leq 2C_k\big(P^{4(k+1)}+X^{4(k+1)}\big)\, \text{ and } \, P^2X^{4k}P^2 
\leq 2C_k\big(P^{4(k+1)}+X^{4(k+1)}\big),
\end{align}
which we use to further estimate \eqref{4.88} by 
\begin{equation}
2 P^{4(k+1)}+2X^{4(k+1)}+2X^2P^{4k}X^2+2P^2X^{4k}P^2\leq (2+4C_k) \big(P^{4(k+1)}+X^{4(k+1)}\big).
\end{equation}
Summarizing the previous steps, we have shown that 
\begin{equation} \lb{4.91}
\big(P^2+X^2\big)\big(P^{4k}+X^{4k}\big)\big(P^2+X^2\big)\leq (2+4C_k) \big(P^{4(k+1)}+X^{4(k+1)}\big).
\end{equation}
Using again that $\big(P^2+X^2\big)^2\geq I$, we use \eqref{4.91} to estimate \eqref{4.84}, 
\begin{align}
&a_k \big(P^2+X^2\big)^{2(k+1)}+b_k I \leq a_k \big(P^2+X^2\big)^{2(k+1)}+b_k \big(P^2+X^2\big)^2 \no \\ 
& \quad \leq s_k(2+4C_k) \big(P^{4(k+1)}+X^{4(k+1)}\big)+t_k \big(P^2+X^2\big)^2  \no \\   
& \quad \leq s_k(2+4C_k) \big(P^{4(k+1)}+X^{4(k+1)}\big)+2t_k \big(P^4+X^4\big)  \no \\    
& \quad \leq [s_k(2+4C_k+2t_k] \big(P^{4(k+1)}+X^{4(k+1)}\big)+2t_k I,
\end{align}
where we have used $\big(P^2+X^2\big)^2\leq 2\big(P^4+X^4\big)$ by \eqref{4.83} and $X^4\leq X^{4(k+1)}+I$ as well as $P^4\leq P^{4(k+1)}+I$ for the last step.

Consequently, there exist $s_{k+1}, t_{k+1} \in (0,\infty)$ such that 
\begin{equation}
a_{k+1} \big(P^2+X^2\big)^{2(k+1)}+b_{k+1} I \leq s_{k+1} \big(P^{4(k+1)}+X^{4(k+1)}\big)+t_{k+1} I,
\end{equation}
finishing the proof.
\end{proof}

The commutation technique exploited in the proof of Theorem \ref{t4.5} goes back to Glimm and Jaffe \cite{GJ69} (see also their use in \cite{Si70}). 

\begin{theorem} \lb{t4.6}
For any $k\in\mathbb{N}$, the space of compactly supported smooth functions is a core of $T_{HO}^k$.
\end{theorem} 

This follows, for instance, from Chernoff \cite[Lemma~2.1 and Sect.~4]{Ch73}. We very briefly sketch the argument. First, Chernoff proves the following result:

\begin{theorem} $($\cite[Lemma~2.1]{Ch73}$)$ \lb{t4.7} 
Let $A$ be a symmetric operator on $\cD$ in $\cK$, leaving $\cD$ invariant, $A \cD \subseteq \cD$. Suppose that there exists a one-parameter unitary family $V(t)$, $t \in [0,\infty)$ on $\cK$ leaving $\cD$ invariant and commuting with $A$ on $\cD$,
\begin{equation}
V(t) \cD \subseteq \cD, \quad V(t) A v = A V(t) v, \quad v \in \cD, \; t \in [0,\infty).  
\end{equation}
If 
\begin{equation}
\f{d}{dt} V(t) v = i A V(t) v, \quad v \in \cD,
\end{equation}
then for each $n \in\bbN$, $A^n$ is essentially self-adjoint on $\cD$ in $\cK$. 
\end{theorem}

Following Chernoff in \cite[Sect.~4]{Ch73}, one  then argues as follows: Consider 
\begin{equation}
T = \big[\big(-d^2\big/dx^2\big) + x^2\big]\big|_{C_0^{\infty}(\bbR)} \geq I 
\end{equation} 
in $L^2(\bbR)$. Then the hyperbolic partial differential equation
\begin{align}
\begin{split} 
& u_{tt} = - T u, \; t \in [0,\infty),    \lb{4.90} \\
& u(0,\dott) = u_1, \; u_t(0,\dott) = u_2, \quad u_k \in C_0^{\infty}(\bbR), \; k=1,2, 
\end{split} 
\end{align}
is known to have a unique global solution satisfying $u(t,\dott) \in C_0^{\infty}(\bbR)$ for each $t \in [0,\infty)$.

Next, one considers the equivalent first-order system 
\begin{align}
\begin{split}
& \f{d}{dt} \begin{pmatrix} u_t(t,\dott) \\ u(t,\dott) \end{pmatrix} = \begin{pmatrix} 0 & - T \\ I & 0 \end{pmatrix} \begin{pmatrix} u_t(t,\dott) \\ u(t,\dott) \end{pmatrix}, \quad t \in [0,\infty),      \lb{4.91a} \\
& \begin{pmatrix} u_t(0,\dott) \\ u(0,\dott) \end{pmatrix} =  \begin{pmatrix} u_2 \\ u_1 \end{pmatrix} 
\in \cD_0 := C_0^{\infty}(\bbR) \oplus C_0^{\infty}(\bbR).
\end{split}
\end{align}
Introducing the Hilbert space $\cK_0 = L^2(\bbR) \oplus \cH_T$, with 
\begin{equation}
\cH_T = \ol{C_0^{\infty}(\bbR)}^{\, \| \dott\|_{\cH_T}}, \quad (f,g)_{\cH_T} = (f, T g)_{L^2(\bbR)}, 
\end{equation} 
the problem \eqref{4.91a} can be rewritten as  
\begin{align}
\begin{split} 
& v_t(t) = i A_0 v(t), \; t \in [0,\infty), \quad v(0) = \begin{pmatrix} u_2 \\ u_1 \end{pmatrix} \in \cD_0,     \lb{4.93} \\
& A := - i \begin{pmatrix} 0 & - T \\ I & 0 \end{pmatrix}, \quad \dom(A) = \cD_0.   
\end{split} 
\end{align}
An elementary computation shows that $A$ on $\cD_0$ is symmetric in $\cK_0$. Moreover, given the equivalence of \eqref{4.90} and \eqref{4.91a} and hence that of \eqref{4.90} and \eqref{4.93}, one concludes that \eqref{4.93} gives rise to a unique and globally defined unitary solution operator $V_0(t)$, $t \in [0,\infty)$, in $\cK$ that leaves $\cD_0$ invariant and commutes with $A_0$ on $\cD_0$; in particular, $v(t) = V_0(t) v(0)$, $t \in [0,\infty)$. Thus, Theorem \ref{t4.6} applies to $A_0$ in $\cK_0$ and hence all integer powers of $A_0$ are essentially self-adjoint on $\cD_0$ in $\cK_0$. Since 
\begin{equation}
(i A_0)^{2n} = \begin{pmatrix} T^n & 0 \\ 0 &T^n \end{pmatrix} \, \text{ on } \, \cD_0, \quad n \in \bbN, 
\end{equation}
also 
\begin{equation} 
T^n, \; n \in \bbN, \, \text{ is essentially self-adjoint on } \,  C_0^{\infty}(\bbR) \, \text{ in } \, L^2(\bbR).
\end{equation} 

As pointed out by Chernoff, this strategy extends to appropriate Dirac-type and Laplace--Beltrami operators  $T$ on complete Riemannian manifolds and certain $n$-dimensional second-order elliptic differential operators bounded from below. Earlier results in this direction appeared in \cite{Co72} and \cite{Tr68}.

\begin{corollary}
For any $k\in\bbN$, we have
\begin{equation}
\dom\big(T_{HO}^k\big)=H^{2k}(\bbR)\cap\dom\big(X^{2k}\big). 
\end{equation}
\end{corollary}
\begin{proof} From Theorem \ref{t4.5} it follows that the graph norm of $T_{HO}^k$ is equivalent to the norm 
\begin{equation}
\|f\|_{H_{x^{2k}}^{2k}(\bbR)}^2=\|f\|^2_{H^{2k}(\bbR)}+ \big\|X^{2k}f\big\|^2.
\end{equation}
Hence the closure of $C_0^\infty(\bbR)$ with respect to the graph norm of $T_{HO}^k$ is equal to the closure of $C_0^\infty(\bbR)$ with respect to $\|\cdot\|_{H_{x^{2k}}^{2k}(\bbR)}$. But by Theorem \ref{t4.6}, the former is equal to $\dom\big(T_{HO}^k\big)$, while Lemma \ref{lemma:4.2} implies that the latter is equal to $H^{2k}(\bbR)\cap\dom\big(X^{2k}\big)$, completing the proof.
\end{proof}

We will use the following general result by Triebel:
\begin{theorem} $($\cite[Satz 1]{Tr69}$)$
Let $V \in C^1(\bbR ^n)$ with $V(x) \geq \varepsilon >0$ for all $x\in\bbR ^n$.
Suppose that there is $C \in (0,\infty)$ such that
\begin{equation}
|\nabla V(x)| \leq C V(x).
\end{equation}
Moreover, assume that for every $K \in (0,\infty)$, there exists $N_K \in (0,\infty)$ such that $V(x) \geq K$ whenever $|x|\geq N_K$ for $x\in\bbR^n$. Then for $m\in\bbN$, $\theta \in (0,1)$, the interpolation space $\big(L^2(\bbR ^n), H_V^m(\bbR ^n)\big)_{\theta,2}$ is given by $H^{m\theta}_V(\bbR ^n)$, that is,
\begin{equation}
\big(L^2(\bbR ^n), H_V^m(\bbR ^n)\big)_{\theta,2} = H^{m\theta}_V(\bbR ^n), \quad 
m \in \bbN, \; \theta \in (0,1). 
\end{equation}
\lb{prop:Satz1}
\end{theorem}

\begin{remark}
Satz~1 in \cite{Tr69} actually contains stronger results than presented here, for instance, it permits  interpolation between Sobolev spaces of higher order and also weighted $L^2$-spaces. We presented Satz~1 in a simplified form sufficient for our purpose. \hfill $\diamond$
\end{remark}

\begin{theorem} \lb{t4.10}
The domains of the fractional powers $T_{HO}^s$, $s \in (0,\infty)$, of the harmonic oscillator defined in \eqref{4.30} are given by
\begin{equation}
\dom\big(T_{HO}^s\big)=H^{2s}_{x^2}(\bbR)=H^{2s}(\bbR)\cap\dom\big(|X|^{2s}\big), 
\quad s \in (0,\infty).
\end{equation}
\end{theorem}
\begin{proof}
If $s \in \bbN$, the result follows from Corollary \ref{coro:4.4}. Hence we assume $s \notin \bbN$ from now on. Pick $m\in \bbN$ such that $m > s$. First, one notes that with the choice $V(x)=x^2+1$, the function $V$ satisfies the assumptions of Proposition \ref{prop:Satz1}, that is, $V \in C^1(\bbR)$, $V(x)\geq 1>0$, $x\in\bbR$, and $|\nabla V(x)|=2|x|\leq x^2+1= V(x)$, $x\in\bbR$. Moreover, for every $K>0$, we may choose $N_K=\sqrt{K}$ to ensure $V(x) \geq K$ whenever $|x| \geq N_K$. Hence, by Proposition \ref{prop:Satz1}, one concludes that
\begin{equation}
\big(L^2(\bbR ), H_{x^2+1}^{2m}(\bbR )\big)_{s/m,2}=H^{2s}_{x^2+1}(\bbR )=H^{2s}_{x^2}(\bbR).
\end{equation}
Next, one observes that 
\begin{equation} 
H^{2m}_{x^2+1}(\bbR)=H^{2m}_{x^2}(\bbR)=H^{2m}(\bbR)\cap\dom\big(X^{2m}\big)=\dom\big(T_{HO}^m\big) 
\end{equation} 
by Corollary \ref{coro:4.4}. Consequently, by \eqref{A.6a}, one gets
\begin{align}
\begin{split} 
\dom\big(T_{HO}^s\big) &= \big(L^2(\bbR ), \dom\big(T^m_{HO}\big)\big)_{s/m,2} 
= \big(L^2(\bbR ), H_{x^2+1}^{2m}(\bbR )\big)_{s/m,2}   \\
& =H^{2s}_{x^2}(\bbR)=H^{2s}(\bbR)\cap\dom\big(|X|^{2s}\big), \quad s \in (0,\infty), 
\end{split} 
\end{align}
completing the proof.
\end{proof}

\begin{remark} \lb{r4.11} 
Given Theorem \ref{t4.10}, we can describe the elements of $\dom\big(T_{HO}^s\big)$ for $s \in (0,\infty)$ in various ways. For the first description, we will use that $f\in \dom\big(T_{HO}^s\big)=H^{2s}(\bbR)\cap\dom\big(|X|^{2s}\big)$ if and only if $f$ and its Fourier transform $\hatt{f}$ both are elements of $\dom\big(|X|^{2s}\big)$, that is, 
\begin{equation}
\int_\bbR dx \, \big(1+|x|^{2s}\big)^2 \Big[|f(x)|^2+\big|\hatt{f}(x)\big|^2\Big] < \infty. 
\end{equation}  
For the next characterization, we use that $f\in H^{2s}(\bbR)$ if and only if
\begin{align} 
\begin{split}
& \big\|f^{(\lfloor 2s\rfloor)}\big\|^2_{L^2(\bbR)} 
 + \begin{cases} \int_\bbR\int_\bbR dx dy \, \frac{\big|f^{(\lfloor 2s\rfloor)}(x)-f^{(\lfloor 2s\rfloor)}(y)\big|^2}{|x-y|^{1+2\{2s\}}}, & 2s \in (0,\infty) \backslash \bbN, \\
0, & 2s \in \bbN \end{cases}   \\
& \quad + \|f\|^2_{L^2(\bbR)} < \infty.
\end{split}
\end{align}
(we recall that $\{x\} \in (0,1)$ denotes the fractional part of $x \in (0,\infty)$, see \eqref{1.27}).  
Consequently, $f\in\dom\big(T_{HO}^{2s}\big)$ if and only if
\begin{align} \lb{4.95a}
\begin{split} 
& \big\|f^{(\lfloor 2s\rfloor)}\big\|^2_{L^2(\bbR)}   
+ \begin{cases} \int_\bbR\int_\bbR dx dy \, \frac{\big|f^{(\lfloor 2s\rfloor)}(x)-f^{(\lfloor 2s\rfloor)}(y)\big|^2}{|x-y|^{1+2\{2s\}}}, & 2s \in (0,\infty) \backslash \bbN, \\
0, & 2s \in \bbN \end{cases}    \\
& \quad + \int_\bbR dx \, \big[1 + |x|^{4s}\big] |f(x)|^2  < \infty. 
\end{split}
\end{align}
In particular, if $s<1/2$, then \eqref{4.95a} simplifies to the condition 
\begin{equation}
\int_\bbR\int_\bbR dx dy \, \frac{|f(x)-f(y)|^2}{|x-y|^{1+4s}}+\int_\bbR dx \, \big[1 + 
|x|^{4s}\big] |f(x)|^2 < \infty, \quad s \in (0,1/2).    \lb{4.96}
\end{equation} 
Thus, the expression 
\begin{equation}
\int_\bbR\int_\bbR dx dy \, \frac{|f(x)-f(y)|^2}{|x-y|^{1+2s}}+\int_\bbR dx \, \big[1 + 
|x|^{2s}\big] |f(x)|^2 < \infty, \quad s \in (0,1),    \lb{4.99}
\end{equation} 
can be thought of as a modification of the standard Gagliardo--Slobodeckij norm 
\begin{equation}
\int_\bbR\int_\bbR dx dy \, \frac{|f(x)-f(y)|^2}{|x-y|^{1+2s}} + \|f\|^2_{L^2(\bbR)} < \infty, \quad s \in (0,1),    \lb{4.100}
\end{equation} 
adapted to the domains of fractional powers $T_{HO}^s$, $s \in (0,1/2)$, of $T_{HO}$ on $H^2(\bbR) \cap \dom\big(X^2\big)$ instead of $\big(- d^2\big/dx^2\big)^s$, $s \in (0,1/2)$, for the Laplacian $- d^2/dx^2$ on $H^2(\bbR)$ in $L^2(\bbR)$. 

For the last characterization, again since $f\in\dom\big(T_{HO}^s\big)$ if and only if $\hatt{f}\in\dom\big(T_{HO}^s\big)$, equation \eqref{4.95a} can be reformulated with $\hatt{f}$ taking the role of $f$:
\begin{align} 
\begin{split} 
& \big\|\hatt{f}^{(\lfloor 2s\rfloor)}\big\|^2_{L^2(\bbR)}
+ \begin{cases} \int_\bbR\int_\bbR d\xi d\eta \, \frac{\big|\hatt{f}(\xi)-\hatt{f}(\eta)\big|^2}{|\xi-\eta|^{1+2\{2s\}}}, & 2s \in (0,\infty) \backslash \bbN, \\
0, & 2s \in \bbN \end{cases}   \\ 
& \quad +\int_\bbR d\xi \, \big[1 + |\xi|^{4s}\big] \big|\hatt{f}(\xi)\big|^2<\infty, \quad s \in (0,\infty). 
\end{split}
\end{align}
${}$ \hfill $\diamond$
\end{remark}

For a different approach to a characterization of Sobolev spaces associated to the (multi-dimensional) harmonic oscillator, see \cite{BT03}. 

At this point we can revisit $\cV_s(A_H)$ and $\cH_s(A_H)$, $s \in [0,\infty)$, and provide a new characterization based on Theorem \ref{t4.10} as follows: Employing the unitary equivalence \eqref{4.33} between $A_H$ and $T_{HO}$, one infers 
\begin{equation}
A_H^{s/2} = U^{-1} [T_{HO} + (c-1) I]^{s/2} U, \quad s \in [0,\infty),
\end{equation}
and hence \eqref{4.25}, \eqref{4.32} imply
\begin{align}
& \cV_s(A_H) = \dom\big(A_H^{s/2}\big) = H^s\big(\bbR; e^{-x^2}dx\big)      \no \\
& \quad = \dom\big([T_{HO} + (c-1) I]^{s/2} U\big)    \no \\
& \quad = \big\{g \in L^2\big(\bbR; e^{-x^2} dx\big) \, \big| \, \omega \cdot g \in H^s(\bbR) 
\cap \dom \big(|X|^s\big)\big\}, \quad s \in [0,\infty),
\end{align}
where we abbreviated
\begin{equation}
\omega(x) = e^{- x^2/2}, \quad x \in \bbR. 
\end{equation}
Thus, one obtains yet another characterization of the left-definite space $\cH_s(A_H)$ in \eqref{4.26} associated with $A_H$.

\appendix
\section{Domains of Periodic Fractional Laplacians}\lb{sA}

In this appendix we provide a description of
\begin{equation}
\dom\big(A_{\phi}^{s/2}\big)=\dom\big(|P_{\phi}|^s\big), \quad s\in(0,1),   \lb{A.1}
\end{equation}
in terms of fractional Sobolev spaces and explicit boundary conditions (when applicable). The principal motivation underlying this appendix stems from the apparent lack of ready availability of this topic in the standard literature on fractional periodic Sobolev spaces (see, however, \cite{Am20}, \cite[Ch.~3]{Be73}, \cite[Ch.~3]{II01}).

To accomplish a precise description of the domains \eqref{A.1}, we will make use of the relation between interpolation spaces and domains of generators of the underlying strongly continuous semigroups in a complex, separable Hilbert space $\cH$. In this context we recall that $T(t) \in \cB(\cH)$, $t \in [0,\infty)$, is called a {\it strongly continuous semigroup in $\cH$} if
\begin{align}
\begin{split}
& T(0) = I_{\cH}, \quad T(s)T(t) = T(s+t), \; s,t \in [0,\infty),     \\
& T(\dott)f \colon [0,\infty) \to \cH \, \text{ is continuous in $\cH$ for each $f \in \cH$.}
\end{split}
\end{align}
We also recall the notion of real interpolation spaces using the $K$-method: Suppose $X, Y$ are Banach spaces with $Y \subseteq X$ and such that for some $c \in (0,\infty)$,
\begin{equation}
\|Y\|_{X} \leq c \|y\|_{Y}, \quad y \in Y.
\end{equation}
Introducing (cf.\ \cite[eq.~(3)]{Tr69})
\begin{equation}
\wti K(t,x,X,Y) = \inf_{x = a+b, \, a \in X, b \in Y} [\|a\|^2_X + t \|b\|^2_Y], \quad x \in X, \, t \in (0,\infty),
\end{equation}
the real interpolation spaces $(X,Y)_{\theta,p}$ are given by (see also \cite[Definition~1.1, Sect.~3.1]{Lu18} or \cite[Sect.~1.3]{Tr78} for details),
\begin{align}
\begin{split}
(X,Y)_{\theta,p} = \bigg\{x \in X \, \bigg|\, \int_0^{\infty} \f{dt}{t^{1 + p \theta}} \big[\wti K(t,x,X,Y)\big]^{p/2} < \infty\bigg\},& \\
\theta \in (0,1), \; p \in [1,\infty) \cup \{\infty\},&
\end{split}
\end{align}
and with
\begin{align}
\begin{split}
\|x\|_{\theta,p} = \Big\|t^{- (1/p) - \theta} \big[\wti K(t,x,X,Y)\big]^{1/2}\Big\|_{L^p((0,\infty);dt)},&      \\
\theta \in (0,1), \; p \in [1,\infty) \cup \{\infty\}, \; x \in X,&
\end{split}
\end{align}
the normed spaces $\big((X,Y)_{\theta,p}, \|\dott\|_{\theta,p}\big)$, $\theta \in (0,1)$, $p \in [1,\infty) \cup \{\infty\}$,  are Banach spaces.

First, we need the following useful results:
\begin{theorem} $($\cite[Proposition~5.10]{Lu18}, \cite[Subsect.~1.13.2]{Tr78}$)$ \lb{tA.1} ${}$ \\
Let $A$ generate a strongly continuous semigroup $T_A(t)$, $t \in [0,\infty)$, on the complex, separable  Hilbert space $\cH$. Then,
\begin{align}
\begin{split}
\big(\cH,\dom\big(A^k\big)\big)_{\theta,2} = \bigg\{f\in\cH \, \bigg| \, \int_0^\infty \f{dt}{t^{1+2k \theta}} \, \big\|[T_A(t)-I_{\cH}]^k f \big\|_{\cH}^2 < \infty \bigg\},&  \\ 
\theta \in (0,1), \; k \in \bbN,&
\end{split}
\end{align}
and the norms 
\begin{equation}
\|f\|_{\theta,2} \, \text{ and } \, \big[\|f\|^2_{\cH} + \|\psi_A(\dott;f)\|^2_{L^2((0,\infty)}\big]^{1/2}, 
\end{equation}
are equivalent. Here
\begin{equation}
\psi_A(t;f) = t^{-(1/2) -k \theta} \big\|[T_A(t) - I_{\cH}]^k f\big\|_{\cH}, \quad t \in (0,\infty), \; 
f \in \big(\cH,\dom\big(A^k\big)\big)_{\theta,2}.
\end{equation}
\end{theorem}

For completeness, we also recall a resolvent analogue of Theorem \ref{tA.1}:
\begin{theorem} $($\cite[Subsect.~1.14.2]{Tr78}$)$ \lb{tA.2} ${}$ \\
Suppose that $A: \dom(A) \to \cH$, $\dom(A) \subseteq \cH$ is a linear operator satisfying 
\begin{align}
& (-\infty,0) \subset \rho(A),   \\
& \text{there exists a constant $C \in (0,\infty)$ such that }   \no \\
& \big\|\lambda (A-\lambda I_{\cH})^{-1}\big\|_{\cB(\cH)} \leq C, \quad \lambda \in (-\infty,0).
\end{align}
Then, for $k \in \bbN$, $\theta \in (0,1)$, 
\begin{align}
& \big(\cH,\dom\big(A^k\big)\big)_{\theta,2} = \bigg\{f\in\cH \, \bigg| \, \int_0^\infty d\lambda \, \lambda^{-1+2k \theta} \big\|[A(A -\lambda I_{\cH})^{-1}]^k f \big\|_{\cH}^2 < \infty \bigg\},    \no \\ 
& \hspace*{8cm} \theta \in (0,1), \; k \in \bbN,&
\end{align} 
and the norms 
\begin{equation}
\|f\|_{\theta,2} \, \text{ and } \, \big[\|f\|^2_{\cH} + \|\varphi_A(\dott;f)\|^2_{L^2((0,\infty)}\big]^{1/2} 
\end{equation}
are equivalent. Here
\begin{equation}
\varphi_A(\lambda;f) = \lambda^{-1/2) + k \theta} \big\|[A(A -\lambda I_{\cH})^{-1}]^k f \big\|_{\cH}, \quad \lambda \in (0,\infty), \; f \in \big(\cH,\dom\big(A^k\big)\big)_{\theta,2}.
\end{equation}
\end{theorem}

If $A$ is a strictly positive self-adjoint operator, then the following result provides the link between
$\big(\dom\big(A^{\ell}\big),\dom\big(A^{k}\big)\big)_{\theta,2}$ and
$\dom\big(A^{(1-\theta) \ell + \theta k}\big)$, $\theta \in (0,1)$, $k, \ell \in\bbN_0$. More generally, one has the following fact:

\begin{theorem} $($\cite[Theorem~4.36]{Lu18}, \cite[Subsect.~1.18.10]{Tr78}$)$ \lb{tA.3}
Let $A$ be a strictly positive self-adjoint operator in $\cH$ and $\alpha, \beta \in \bbC$, $\Re(\alpha) \in [0, \infty)$, $\Re(\beta) \in [0,\infty)$. Then\footnote{By \cite[Theorem\ 4.36]{Lu18}, \cite[Subsect.~1.18.10]{Tr78}, the result \eqref{A.6} is consistent with the complex interpolation method (see, e.g., \cite{CE19}, \cite[Ch.~2]{Lu18}, \cite[Sect.~1.9]{Tr78}).}
\begin{equation}
\big(\dom\big(A^{\alpha}\big),\dom\big(A^{\beta}\big)\big)_{\theta,2}
= \dom\big(A^{(1-\theta)\alpha + \theta \beta}\big), \quad \theta\in(0,1).    \lb{A.6}
\end{equation}
In particular, if $\alpha = 0$, one obtains
\begin{equation}
\big(\cH,\dom\big(A^{\beta}\big)\big)_{\theta,2}
= \dom\big(A^{\theta \beta}\big), \quad \theta\in(0,1).    \lb{A.6a}
\end{equation}
\end{theorem}

For extensions to larger classes of operators $A$ in Theorem \ref{tA.3}, see, for instance, \cite{Ka61} and \cite{Li62}.

As a concrete example, relevant to Section \ref{s4}, we briefly take a closer look at the semigroup $e^{- t A_H}$ generated by $A_H$, making the connection with Mehler's formula in connection with the harmonic oscillator operator $T_{HO}$ in $L^2(\bbR)$.

\begin{example} \lb{eA.4}
The harmonic oscillator semigroup $e^{-t T_{HO}}$ has the well-known explicit integral kernel given by Mehler's formula
\begin{align}
& e^{-t T_{HO}} (x,x') = [2 \pi \sinh(2t)]^{-1/2}     \\
& \quad \times \exp{\big\{- 2^{-1} \coth(2t) (x^2 + x'^2) + [\sinh(2t)]^{-1} x x'\big\}},
\quad t \in (0,\infty), \; x, x' \in \bbR    \no
\end{align}
$($see, e.g., \cite[Sect.~12.9]{CFKS87} which offers four different proofs\,$)$. In addition, one has the spectral expansion
\begin{align}
\begin{split} 
e^{-t T_{HO}} (x,x') &= \sum_{m \in \bbN_0} \big[\pi^{1/2} 2^m m!\big]^{-1} e^{- (2m+1)t} H_m(x) H_m(x') e^{-(x^2+x'^2)/2}     \\
&= \sum_{m \in \bbN_0}e^{- (2m+1)t} u_m(x) u_m(x'), \quad t \in (0,\infty), \; x, x' \in \bbR.
\end{split} 
\end{align}
Since
\begin{equation}
U^{-1} e^{- t [T_{HO} + (c-1) I]} U = e^{-t A_H}, \quad t \in (0,\infty),
\end{equation}
one obtains for the integral kernel of the Hermite semigroup operator,
\begin{align}
& e^{-t A_H} (x,x') = [2 \pi \sinh(2t)]^{-1/2} e^{-t(c-1)}      \lb{A.13} \\
& \quad \times \exp{\big\{[2 \sinh(2t)]^{-1} \big[- e^{-2t} (x^2 + x'^2) + 2x x'\big] \big\}},
\quad t \in (0,\infty), \; x, x' \in \bbR.     \no
\end{align}

Thus, by Theorems \ref{tA.1} and \ref{tA.3}, one obtains 
\begin{align}
& \big(L^2\big(\bbR; e^{-x^2}dx\big),\dom\big(A_H^k\big)\big)_{\theta,2}
= \dom\big(A_H^{\theta k}\big)   \\
& \quad = \bigg\{f\in L^2\big(\bbR; e^{-x^2}dx\big) \, \bigg| \, \int_0^\infty \f{dt}{t^{1+2k \theta}} \,
\big\|[e^{-t A_H} -I]^k f \big\|_{L^2(\bbR; e^{-x^2}dx)}^2 < \infty \bigg\},  \no \\
& \hspace*{8.9cm} \theta \in (0,1), \; k \in \bbN.    \no
\end{align}
\end{example}

Next, employing the definition of $H^s(\R)$ in \eqref{3.10}, we recall that
\begin{align}
H^s((0,2\pi))&= \big\{f\in L^2((0,2\pi)) \, \big| \, \text{there exists} \, g\in H^s(\R) \, \text{s.t.} \, g\upharpoonright_{(0,2\pi)}=f\big\},    \\
H^s_0((0,2\pi))&=\overline{C_0^\infty((0,2\pi))}^{\|\dott\|_{H^s(\R)}},\\
H^s_{0,0}((0,2\pi))&= \big\{f\in H^s((0,2\pi)) \, \big| \, \widetilde{f}\in H^s(\R)\big\},     \lb{A.3}
\end{align}
where the extension by zero operation $\widetilde{f}$ in \eqref{A.3} is given by
\begin{equation}
\widetilde{f}(x)=\begin{cases} f(x), & x\in(0,2\pi),  \\
0, & x \in \bbR \backslash (0,2\pi).
\end{cases}
\end{equation}

Before we proceed with the main theorem, we need the following result:

\begin{lemma}\lb{lA.3}
Let $s\in (1/2,\infty)$ and $f\in\dom\big(|P_{\phi}|^s\big)$. Then $f$ is uniformly continuous on $[0,2\pi]$ and satisfies the boundary condition $f(0)=e^{i\phi}f(2\pi)$.
\end{lemma}
\begin{proof}
First, one notes that every eigenfunction $\psi_{\phi,n}$ is (uniformly) continuous on $[0,2\pi]$ and satisfies the boundary condition $\psi_{\phi,n}(0)=e^{i\phi}\psi_{\phi,n}(2\pi)$. Assuming $f\in\dom\big(|P_{\phi}|^s\big)$ we will show that
\begin{equation}
\sum_{n\in\mathbb{Z}} \psi_{\phi,n}(\dott) (\psi_{\phi_n},f)_{L^2((0,2\pi))}      \lb{A.7}
\end{equation}
converges uniformly on $[0,2\pi]$. This follows from the Weierstra{\ss}  $M$-test: One has
\begin{equation}
|\psi_{\phi,n}(x)|\leq 1, \quad n \in \bbZ, \; x \in [0,2\pi].
\end{equation}
Choosing $M_n:=|(\psi_{\phi,n},f)_{L^2((0,2\pi))}|$, one gets
\begin{align}
\sum_{n\in\mathbb{Z} \backslash \{0\}}M_n &=\sum_{n\in\mathbb{Z} \backslash \{0\}} \frac{1}{\big|n-\tfrac{\phi}{2\pi}\big|^{s}} \big|n-\tfrac{\phi}{2\pi}\big|^{s}|(\psi_{\phi,n},f)_{L^2((0,2\pi))}|     \no \\
& \leq \Bigg(\sum_{n\neq 0} \frac{1}{\big|n-\tfrac{\phi}{2\pi}\big|^{2s}}\Bigg)^{1/2}\bigg(\sum_{n\neq 0} \big|n-\tfrac{\phi}{2\pi}\big|^{2s} |(\psi_{\phi,n},f)_{L^2((0,2\pi))}|^2 \bigg)^{1/2}    \no \\
&<\infty.
\end{align}
We denote by $F$ the uniform limit in \eqref{A.7} and define the partial sums $f_N:=\sum_{n=-N}^N\psi_{\phi,n} (\psi_{\phi,n},f)_{L^2((0,2\pi))}$. By Parseval's relation one has
$\lim_{N \to \infty}\|f_N-f\| = 0$. Moreover, since $F$ is the uniform limit of $f_N$, one also gets
$\lim_{N \to \infty}\|f_N-F\| = 0$. Hence,
\begin{equation}
0\leq \|f-F\|\leq \|f-f_N\|+\|F-f_N\|\underset{N\rightarrow\infty}{\longrightarrow} 0,
\end{equation}
implying $f=F$. Thus $f$ is continuous, hence, uniformly continuous on $[0,2\pi]$ and
\begin{equation} f(0)-e^{i\phi}f(2\pi)=\sum_{n\in\mathbb{Z}} \big[\psi_{\phi,n}(0)-e^{i\phi}\psi_{\phi,n}(2\pi)\big] (\psi_{\phi,n},f)_{L^2((0,2\pi))} = 0.
\end{equation}
\end{proof}

At this point we can provide a full characterization of the domains $\dom\big(A_{\phi}^{s/2}\big) = \dom\big(|P_{\phi}|^s\big)$, $s \in (0,1)$:

\begin{theorem} \lb{tA.4} The following assertions $(i)$--$(iii)$ hold: \\[1mm]
$(i)$ If $s \in (0,1/2)$, then
\begin{equation} \lb{A.12}
\dom\big(A_{\phi}^{s/2}\big) = \dom\big(|P_{\phi}|^s\big) = H^s((0,2\pi)).
\end{equation}
$(ii)$ If $s=1/2$, then
\begin{align}
& \dom\big(A_{\phi}^{1/4}\big) = \dom\big(|P_{\phi}|^{1/2}\big)     \\
& \quad =\bigg\{f\in H^{1/2}((0,2\pi)) \, \bigg| \,\int_0^{2\pi} \f{dt}{t^2} \int_0^t dx \, |e^{i\phi}f(x+2\pi-t)-f(x)|^2
< \infty\bigg\}.     \no
\end{align}
$(iii)$ If $s \in (1/2,1)$, then
\begin{equation} \lb{A.14}
\dom\big(A_{\phi}^{s/2}\big) = \dom\big(|P_{\phi}|^s\big) = \big\{f\in H^s((0,2\pi)) \, \big| \, f(0)=e^{i\phi}f(2\pi)\big\}.
\end{equation} 
\end{theorem}
\begin{proof} First one notes that since $\dom\big(A_{\phi}^{1/2}\big)=\dom(|P_{\phi}|)=\dom(P_{\phi})$, one gets for every $s\in(0,1]$:
\begin{align}
\begin{split}
\dom\big(A_{\phi}^{s/2}\big)& =\dom\big(|P_{\phi}|^s\big) = \big(L^2((0,2\pi)),\dom(|P_{\phi}|\big)_{s,2}      \\
&= \big(L^2((0,2\pi)),\dom(iP_{\phi})\big)_{s,2},
\end{split}
\end{align}
where the identity $\dom\big(|P_{\phi}|^s\big)= \big(L^2((0,2\pi)),\dom(|P_{\phi}|)\big)_{s,2}$ follows from Theorem \ref{tA.3}. We now apply Theorem \ref{tA.1} to the operator $iP_{\phi}$, which generates the strongly continuous unitary group $T_{\phi}(t)$ given by
\begin{equation}
\left(T_{\phi}(t)f\right)(x)=e^{in\phi}f((x-t) \; {\rm mod} \;2\pi), \quad t \in \bbR, \; x \in (0,2\pi),
\end{equation}
where $n\in\mathbb{N}$ is the unique positive integer such that $x-t+2\pi n\in[0,2\pi)$. Using that $T_{\phi}(t)$ is unitary, one has for $f\in L^2((0,2\pi))$,
\begin{align}
\begin{split}
& \int_{0}^{2\pi} \f{dt}{t^{1+2s}}  \, \|T_{\phi}(t)f-f\|_{L^2((0,2\pi))}^2
\leq \int_{0}^{\infty} \f{dt}{t^{1+2s}}  \, \|T_{\phi}(t)f-f\|_{L^2((0,2\pi))}^2     \\
& \quad \leq
\int_{0}^{2\pi} \f{dt}{t^{1+2s}} \, \|T_{\phi}(t)f-f\|_{L^2((0,2\pi))}^2 + 4\|f\|_{L^2((0,2\pi))}^2\int_{2\pi}^\infty\frac{dt}{t^{1+2s}},
\end{split}
\end{align}
implying that $f\in \dom\big(|P_{\phi}|^s\big)$ if and only if $f\in L^2((0,2\pi))$ and
\begin{equation}
\int_{0}^{2\pi} \f{dt}{t^{1+2s}} \, \|T_{\phi}(t)f-f\|_{L^2((0,2\pi))}^2 < \infty .
\end{equation}
Next, consider
\begin{align}
&\int_{0}^{2\pi} \f{dt}{t^{1+2s}} \left(\int_0^{2\pi} dx \, |(T_{\phi}(t)f)(x)-f(x)|^2\right)    \no \\
& \quad = \int_{0}^{2\pi} \f{dt}{t^{1+2s}}  \left(\int_t^{2\pi} dx \, |(T_{\phi}(t)f)(x)-f(x)|^2\right)     \no \\
&\qquad+	\int_{0}^{2\pi} \frac{dt}{t^{1+2s}} \left(\int_0^{t} dx \, |(T_{\phi}(t)f)(x)-f(x)|^2\right)     \no \\
& \quad =\int_{0}^{2\pi} \f{dt}{t^{1+2s}} \left(\int_t^{2\pi} dx \, |f(x-t)-f(x)|^2\right)    \no \\
&\qquad  +	\int_{0}^{2\pi} \f{dt}{t^{1+2s}} \left(\int_0^{t} dx \, |e^{i\phi}f(x+2\pi-t)-f(x)|^2\right)  \no \\
& \quad = (A)+(B).     \lb{A.20}
\end{align}
From \cite[Eq.\ (1.15)]{Le23} one infers that
\begin{equation}
(A)= |||f|||^2_{H^s((0,2\pi))}\big/2,
\end{equation}
where
\begin{equation}
|||f|||^2_{H^s((0,2\pi))}=\int_0^{2\pi} dy \int_0^{2\pi} dx \, \frac{|f(x)-f(y)|^2}{|x-y|^{1+2s}},
\end{equation}
and the $H^s((0,2\pi))$ norm of $f$ is equivalent to
\begin{equation}
\|\dott\|^2_{H^s((0,2\pi))} = |||\dott |||^2_{H^s((0,2\pi))} + \|\dott\|_{L^2((0,2\pi))}^2.
\end{equation}
In other words, $f\in\dom\big(|P_{\phi}|^s\big)$ if and only if
$f\in H^s((0,2\pi))$ and the second integral $(B)$ in \eqref{A.20} is finite.

In this context one observes that
\begin{align}
(B) &= \int_{0}^{2\pi} \frac{dt}{t^{1+2s}} \left(\int_0^{t} dx \, |e^{i\phi}f(x+2\pi-t)-f(x)|^2 \right)   \no \\
&\leq 2 \int_{0}^{2\pi} \f{dt}{t^{1+2s}} \int_0^{t} dx \, |f(x+2\pi-t)|^2
+ 2\int_{0}^{2\pi} \f{dt}{t^{1+2s}} \int_0^{t} dx \, |f(x)|^2   \no \\
&= 2\int_{0}^{2\pi} dx \left[\int_{2\pi-x}^{2\pi} \frac{dt}{t^{1+2s}}\right]|f(x)|^2
+ 2\int_{0}^{2\pi} dx \left[\int_x^{2\pi}\frac{dt}{t^{1+2s}}\right]|f(x)|^2    \no \\
& \leq
\frac{1}{s}\int_0^{2\pi} dx \, \frac{|f(x)|^2}{(2\pi-x)^{2s}}
+ \frac{1}{s}\int_0^{2\pi} dx \, \frac{|f(x)|^2}{x^{2s}}.   \lb{A.23}
\end{align}
If $s \in (0,1/2)$, one gets from Hardy's Inequality \cite[Theorem~1.76]{Le23} that the expression in \eqref{A.23} is bounded by
\begin{equation}
\eqref{A.23}\leq C_s |||f|||_{H^s((0,2\pi))}^2,
\end{equation}
where the constant $C_s \in (0,\infty)$ only depends on $s$. Hence, if $s \in (0,1/2)$, one concludes that 
$f\in\dom\big(|P_{\phi}|^s\big)$ if and only if $f\in H^s((0,2\pi))$.

If $s=1/2$, then the condition
\begin{equation}
(B)=\int_{0}^{2\pi} \frac{dt}{t^{2}} \left(\int_0^{t}|e^{i\phi}f(x+2\pi-t)-f(x)|^2dx\right) < \infty
\end{equation}
corresponds precisely to what is stated in item $(ii)$.

Finally, if $s \in (1/2,1)$, it follows from Lemma \ref{lA.3} that if $f\in\dom\big(|P_{\phi}|^s\big)$, then $f$ satisfies the boundary condition $f(0)=e^{i\phi}f(2\pi)$. In other words,
\begin{equation}
\dom\big(|P_{\phi}|^s\big)\subseteq \big\{f\in H^s((0,2\pi)) \, \big| \, f(0)=e^{i\phi}f(2\pi)\big\}.
\end{equation}
To prove the reverse inclusion, we note that estimate \eqref{A.23} implies that
\begin{equation}
H^s_{0,0}((0,2\pi))\subseteq \dom\big(|P_{\phi}|^s\big)
\end{equation}
(see \cite[Exercise 1.88]{Le23}). But for $s\in (1/2,1)$, it is known \cite[Theorem~1.87]{Le23} that
\begin{equation}
H_{0,0}^s((0,2\pi)) = H^s_0((0,2\pi)) = \big\{f\in H^s((0,2\pi)) \big| \, f(0)=f(2\pi)=0\big\}.
\end{equation}
Since $\psi_{\phi,0}\in\dom\big(|P_{\phi}|^s\big)$, one obtains
\begin{equation}
H_0^s(0,2\pi)\dot{+}\mbox{lin.span}\{\psi_{\phi,0}\}\subseteq \dom\big(|P_{\phi}|^s\big).
\end{equation}
One verifies that
\begin{equation}
H_0^s(0,2\pi)\dot{+}\mbox{lin.span}\{\psi_{\phi,0}\} = \big\{f\in H^s((0,2\pi)) \, \big| \, f(0)=e^{i\phi}f(2\pi)\big\},
\end{equation}
completing the proof.
\end{proof}

Next, we turn to a characterization of $\dom\big(A_{\phi}^{s/2}\big)$ for arbitrary $s\in\R$. Since we already provided a complete description of $\dom\big(A_{\phi}^{s/2}\big)$ for $s\in(0,1)$, we restrict our considerations to the case $s\in [1,\infty)$.

\begin{theorem} \lb{tA.5}
Let $s \in [1,\infty)$ and write $s=m+\vt\geq 1$, where $m = \lfloor s\rfloor \in\N$ is the integer part of $s$ and $\vt= \{s\}$ the fractional part of $s$. Then the following assertions $(i)$--$(iii)$ hold: \\[1mm]
$(i)$ If $\vt \in [0,1/2)$, then
\begin{align}
\begin{split}
& \dom\big(A_{\phi}^{(m+\vt)/2}\big) = \dom\big(|P_{\phi}|^{m+\vt}\big)   \\
& \quad = \big\{f\in H^{m+\vt}((0,2\pi)) \, \big| \, f^{(k)}(0)=e^{i\phi}f^{(k)}(2\pi), \, k\in\{0,1,\cdots,m-1\}\big\}.
\end{split}
\end{align}
$(ii)$ If $\vt=1/2$, then
\begin{align}
&\dom\big(A_{\phi}^{(m+(1/2))/2}\big) = \dom\big(|P_{\phi}|^{m+(1/2)}\big)   \no \\
&\quad = \bigg\{f\in H^{m+1/2}((0,2\pi)) \, \bigg| \, f^{(k)}(0)=e^{i\phi}f^{(k)}(2\pi), \, k\in\{0,1,\cdots,m-1\};    \\   &\hspace*{3.35cm} \int_0^{2\pi} \f{dt}{t^2} \int_0^t dx \, |e^{i\phi}f^{(m)}(x+2\pi-t)-f^{(m)}(x)|^2 < \infty\Bigg\}.    \no
\end{align}
$(iii)$ If $\vt \in (1/2,1)$, then
\begin{align}
\begin{split}
& \dom\big(A_{\phi}^{(m+\vt)/2}\big) = \dom\big(|P_{\phi}|^{m+\vt}\big)    \\
& \quad =\big\{f\in H^{m+\vt}((0,2\pi)) \, \big| \, f^{(k)}(0)=e^{i\phi}f^{(k)}(2\pi)\:\:\forall k\in\{0,1,\cdots,m\}\big\}.
\end{split}
\end{align}
\end{theorem}
\begin{proof}
Since
\begin{equation}
\dom\big(A_{\phi}^{s/2}\big)=\dom\big(A_{\phi}^{(m+\vt)/2}\big)
= \Big\{f\in\dom(A_{\phi}^{m/2}) \, \Big| \, A_{\phi}^{m/2}f\in\dom\big(A_{\phi}^{\vt/2}\big)\Big\},
\end{equation}
and
\begin{align}
\dom\big(A_{\phi}^{m/2}\big) &= \dom\big(|P_{\phi}|^m\big)=\dom\big(P_{\phi}^m\big)    \\
&= \big\{f\in H^m((0,2\pi)) \, \big| \, f^{(k)}(0)=e^{i\phi}f^{(k)}(2\pi), \, k\in\{0,1,\cdots,m-1\}\big\},   \no
\end{align}
the result follows from Theorem \ref{tA.4}, which provides the description of the domains
$\dom\big(A_{\phi}^{\vt/2}\big)$.
\end{proof}

\section{More on Fractional Sobolev Spaces on $(0,2\pi)$} \lb{sB}

Although not needed in the bulk of this paper, we include the following results as they seem interesting in their own right.

The motivation for writing this appendix, in part, stems from the classical result that the form domain, $\dom\big(T_F^{1/2}\big)$, of the Friedrichs extension, $T_F$, of a lower semibounded symmetric operator $\dot T$ in the separable, complex Hilbert space $\cH$, is a proper subset of the form domain of any other nonnegative self-adjoint extension of $\dot T$ (cf.\ e.g., \cite{AS80}). In the following we assume that $\dot T$ has a strictly positive lower bound, $\dot T \geq \varepsilon I_{\cH}$ for some $\varepsilon > 0$ and non-zero, finite deficiency indices. Viewing the inclusion relation ``$\subseteq$" as a partial order on the form domains of all nonnegative self-adjoint extensions of $\dot T$, this means that $\dom\big(T_F^{1/2}\big)$ is its unique minimal element.

Since for any two nonnegative self-adjoint extensions $T_1, T_2$ of $\dot T$, one has (this is a corollary of \cite[Proposition~2.4, Theorem~3.1]{AS80})
\begin{align}
\begin{split}
& \dom\big(T_1^{1/2}\big) \subseteq \dom\big(T_2^{1/2}\big) \, \text{ if and only if } \\
& \quad \dom(T_2)\cap\dom(T_F) \subseteq \dom(T_1)\cap\dom(T_F),
\end{split}
\end{align}
the maximal elements with respect to this partial order correspond to the form domains of those extensions $T_2$ of $\dot T$ that are disjoint (equivalently, relatively prime) to $T_F$, that is\footnote{Since the deficiency indices of $\dot T$ are assumed to be finite, disjointness also implies transversality, see, e.g., \cite[Lemma~1.7.7]{{BHS20}}}, they satisfy $\dom(T_2)\cap\dom(T_F)=\dom(\dot T)$.
The Krein--von Neumann extension of $\dot T$ is one such example, but there are of course infinitely many more. Hence, assuming that $T_2$ satisfies $\dom(T_2)\cap\dom(T_F)=\dom(\dot T)$, one concludes that for any nonnegative self-adjoint extension $T_1$ of $\dot T$, with  $T_1\neq T_F$, one has
\begin{equation}
\dom\big(T_F^{1/2}\big)\subsetneq\dom\big(T_1^{1/2}\big) \subseteq \dom\big(T_2^{1/2}\big).
\end{equation}

Next, we choose $\dot A$ to be the usual minimal Laplacian on the interval $(0,2\pi)$:
\begin{equation}
(A_{min}f)(x)= - f''(x), \quad x \in (0,2\pi), \quad f \in \dom(A_{min})=H^2_0((0,2\pi)),
\end{equation}
whose Friedrichs extension $A_F$ is the Dirichlet Laplacian
\begin{align}
& (A_F f)(x) = - f''(x), \quad x \in (0,2\pi),   \no \\
& f \in \dom(A_F) = \big\{g\in H^2((0,2\pi)) \, \big| \, g(0)=g(2\pi)=0\big\}      \\
& \hspace*{2cm} = H^2((0,2\pi)) \cap H^1_0((0,2\pi)),     \no
\end{align}
whose form domain is given by
\begin{equation}
\dom\big(A_F^{1/2}\big) = H^1_0((0,2\pi)).
\end{equation}
For the disjoint extension, we choose the Neumann Laplacian $A_N$, given by
\begin{align}
\begin{split}
& (A_N f)(x) = - f''(x), \quad x \in (0,2\pi),   \\
& f \in \dom(A_N) = \big\{g\in H^2((0,2\pi)) \, \big| \, g'(0)=g'(2\pi)=0\big\},
\end{split}
\end{align}
with form domain
\begin{equation}
\dom\big(A_N^{1/2}\big) = H^1((0,2\pi)).
\end{equation}
For the form domains, one gets the inclusions
\begin{equation}
\dom\big(A_F^{1/2}\big)\subsetneqq \dom\big(A_{\phi}^{1/2}\big)\subsetneqq \dom\big(A_N^{1/2}\big).
\end{equation}
We now show that for this particular example, the inclusions
\begin{equation}
\dom\big(A_F^{s/2}\big)\subsetneqq \dom\big(A_{\phi}^{s/2}\big)\subsetneqq \dom\big(A_N^{s/2}\big),
\quad s\in[1/2,1),
\end{equation}
hold. However, if $s\in(0,1/2)$, these inclusions collapse and one obtains
\begin{equation}
\dom\big(A_F^{s/2}\big)=\dom\big(A_{\phi}^{s/2}\big)=\dom\big(A_N^{s/2}\big)=H^s((0,2\pi)), \quad s\in(0,1/2),
\end{equation}
instead. To this end, we need the following result.

\begin{theorem}\lb{tB.1}
Let $\phi, \phi_1, \phi_2 \in [0,2\pi)$. Then
\begin{equation} \lb{B.11}
H^{1/2}_{0,0}((0,2\pi))\subsetneqq \dom\big(A_{\phi}^{1/4}\big)\subsetneqq H^{1/2}((0,2\pi)).
\end{equation}
Moreover, if $\phi_1\neq \phi_2$, then
\begin{equation} \lb{B.12}
\dom\big(A_{\phi_1}^{1/4}\big)\cap \dom\big(A_{\phi_2}^{1/4}\big)=H_{0,0}^{1/2}((0,2\pi)).
\end{equation}
\end{theorem}
\begin{proof} By \cite[Exercise 1.88]{Le23}, one has
\begin{equation}
H_{0,0}^{1/2}((0,2\pi))=\bigg\{f\in H^{1/2}((0,2\pi)) \, \bigg| \, \int_{0}^{2\pi} dx \, \left(\frac{|f(x)|^2}{x}+\frac{|f(x)|^2}{2\pi-x}\right) < \infty\bigg\},
\end{equation}
so it immediately follows from \eqref{A.23}, with $s=1/2$, that $f\in H^{1/2}_{0,0}((0,2\pi))$ implies $f\in \dom\big(A_{\phi}^{1/4}\big)$. To show that the inclusion is strict, we note that the function
$\psi_{\phi,0}(x)$ is an element of $\dom\big(A_{\phi}^{1/4}\big)$, but not of $H^{1/2}_{0,0}((0,2\pi))$.

Since the inclusion $\dom\big(A_{\phi}^{1/4}\big) \subseteq H^{1/2}((0,2\pi))$ is clear, it remains to show that it is strict. This can be seen by choosing a $C^{\infty}$-function $u$ which equals one on the interval $(0,\varepsilon)$, for $0 < \varepsilon$ sufficiently small, and equals zero on $(2\varepsilon,2\pi)$.

We finish by showing \eqref{B.12}. The inclusion $``\supseteq"$ follows from \eqref{B.11}. Now suppose that
$f\in \dom\big(A_{\phi_1}^{1/4}\big)\cap \dom\big(A_{\phi_2}^{1/4}\big)$, which means that $f\in H^{1/2}((0,2\pi))$. Then
\begin{equation} \lb{B.14}
\int_{0}^{2\pi} \frac{dt}{t^{2}} \, \left(\int_0^{t} dx \, |e^{i\phi_1}f(x+2\pi-t)-f(x)|^2\right) < \infty,
\end{equation}
and
\begin{equation}\lb{B.15}
\int_{0}^{2\pi} \frac{dt}{t^{2}} \, \left(\int_0^{t} dx \, |e^{i\phi_2}f(x+2\pi-t)-f(x)|^2\right) < \infty.
\end{equation}
Now, if
\begin{equation} \lb{B.16}
\int_{0}^{2\pi} \frac{dt}{t^{2}} \, \left(\int_0^{t} dx \,\overline{f(x+2\pi-t)}f(x)\right) = 0,
\end{equation}
then, due to the absence of cross-terms, conditions \eqref{B.14} and \eqref{B.15} are equivalent and simplify to
\begin{equation}
\int_{0}^{2\pi} \frac{dt}{t^{2}} \, \left(\int_0^{t} dx \, \big[|f(x+2\pi-t)|^2+|f(x)|^2\big]\right) < \infty.
\end{equation}
Similarly, one gets
\begin{align} 
&\int_{0}^{2\pi} \frac{dt}{t^{2}} \left(\int_0^{t} dx \, \big[|f(x+2\pi-t)|^2+|f(x)|^2\big]\right)     \no \\
&\quad =\int_{0}^{2\pi} dx \left[\int_{2\pi-x}^{2\pi} \frac{dt}{t^{1+2s}}\right]|f(x)|^2
+ 2\int_{0}^{2\pi} dx \left[\int_x^{2\pi}\frac{dt}{t^{1+2s}}\right]|f(x)|^2    \no \\
& \quad = \int_0^{2\pi} dx \left(\frac{|f(x)|^2}{x}+\frac{|f(x)|^2}{2\pi-x}\right)+\widetilde{C}\|f\|_{L^2((0,2\pi))}^2, \lb{B.19}
\end{align}
which, together with \eqref{B.16}, implies $f\in H^{1/2}_{0,0}((0,2\pi))$. Hence, from now on, assume that the integral in \eqref{B.16} is not equal to zero. Define the constant 
\begin{equation}
M:=\frac{\Re\Big(\int_0^{2\pi} \frac{dt}{t^2}\int_0^t dx \, e^{- i\phi_1}\ol{f(x+2\pi-t)} f(x)\Big)}{\Re\Big(\int_0^{2\pi}\frac{dt}{t^2} \int_0^t dx \, e^{-i\phi_2} \ol{f(x+2\pi-t)}f(x)\Big)}
\end{equation}
and note that since $\phi_1\neq\phi_2$, one concludes that $M \in \bbR \backslash \{1\}$.
Now, consider  $|\eqref{B.14}-M \eqref{B.15}|$, which simplifies to
\begin{equation}
|1-M| \int_{0}^{2\pi} \frac{dt}{t^{2}} \left(\int_0^{t} dx \, \big[|f(x+2\pi-t)|^2+|f(x)|^2\big]\right) < \infty.
\end{equation}

From this, one concludes that $f\in H^{1/2}_{0,0}((0,2\pi))$ by repeating the calculations leading to 
\eqref{B.19}.
\end{proof}

\begin{corollary} \lb{cB.2} Let $\phi, \phi_1, \phi_2 \in [0,2\pi)$. Then the following items $(i)$ and $(ii)$ hold: \\[1mm]
$(i)$ If $s\in(0,1/2)$, then
\begin{equation}\lb{B.22}
\dom\big(A_F^{s/2}\big)=\dom\big(A_{\phi}^{s/2}\big)= \dom\big(A_N^{s/2}\big)=H^s((0,2\pi)).
\end{equation}
$(ii)$ If $s\in[1/2,1)$, then
\begin{equation}
\dom\big(A_F^{s/2}\big)\subsetneq\dom\big(A_{\phi}^{s/2}\big)\subsetneqq \dom\big(A_N^{s/2}\big).
\end{equation}
Moreover, if $\phi_1\neq \phi_2$, then
\begin{equation}
\dom\big(A_{\phi_1}^{s/2}\big) \cap \dom\big(A_{\phi_2}^{s/2}\big) = \dom\big(A_F^{s/2}\big).
\end{equation}
\end{corollary}
\begin{proof} By \cite{Fu67}, one has
\begin{equation}
\dom\big(A_F^{s/2}\big)=H^{s}_{0,0}((0,2\pi)), \quad \dom\big(A_N^{s/2}\big)=H^s((0,2\pi)), \quad s\in(0,1).
\end{equation}

By \cite[Theorem~6.105\,$(i)$]{Le23} (see also \cite[Theorem~1.87\,$(i)$]{Le23}) one has
\begin{equation}
H_{0,0}^{s}((0,2\pi)) = H^s_0((0,2\pi)) = H^s((0,2\pi)), \quad s\in(0,1/2),
\end{equation}
which together with \eqref{A.12} implies \eqref{B.22}.
If $s=1/2$, the result follows from Theorem \ref{tB.1}.
Finally, if $s\in(1/2,1)$, the result follows from the explicit form of $\dom\big(A_{\phi}^{s/2}\big)$ given in \eqref{A.14}.
\end{proof}

\section{Domains of Fractional Laplacians on $(0,\infty)$} \lb{sC}

Finally, we extend some of the results from Appendix \ref{sB} in connection with the bounded interval $(0,2\pi)$ to the half-axis $(0,\infty)$.

The form domains of $B_{\alpha}$, $\alpha\in[\pi/2,\pi]$, are given by
\begin{equation}
\dom\big(B_{\alpha}^{1/2}\big)=\begin{cases} H^1_0((0,\infty)), & \alpha=\pi,   \\
H^1((0,\infty)), & \alpha\in[\pi/2,\pi). \end{cases}
\end{equation}
Using the spectral theorem for self-adjoint operators, one concludes that
\begin{equation}
\dom(B_{\alpha}^{s/2})=\dom(B_{\alpha,1}^{s/2}), \quad \alpha\in[\pi/2,\pi], \; s \in (0,\infty).    \lb{C.2}
\end{equation}
For $s \in (0,1]$, the result \eqref{C.2} also follows from interpolation theory by \cite[Lemma\ 4.11]{Lu18}. In particular,  fractional powers of strictly positive operators defined via the spectral theorem and via (real or complex) interpolation are consistent, see, \cite[Theorem~4.36]{Lu18}.

Employing \cite[Sec.\ 1.10.3 and Theorems\ 1.11.6 and 1.11.7]{LM72}, one obtains
\begin{equation}
\dom\big(B_{\alpha}^{s/2}\big)=\begin{cases} H^s_{0,0}((0,\infty)), & \alpha=\pi,      \\
H^s((0,\infty)), & \alpha\in[\pi/2,\pi), \end{cases}
\quad s \in (0,1].   \lb{C.3}
\end{equation}
One then notes the following facts, see, \cite[Theorem\ 1.87]{Le23} (see also \cite[Sect.~1.11]{LM72}):
\begin{equation}
H_{0,0}^s((0,\infty))=H_{0}^s((0,\infty))=H^s((0,\infty)), \quad s\in(0,1/2);    \lb{C.4}
\end{equation}
moreover,
\begin{align}
\begin{split}
H^{1/2}_{0,0}((0,\infty)) &= \left\{f\in H^{1/2}((0,\infty))\bigg|\int_0^\infty dx\, \frac{|f(x)|^2}{x}<\infty\right\}    \\
& \subsetneqq H_0^{1/2}((0,\infty))=H^{1/2}((0,\infty)), \quad s=1/2,      \lb{C.5}
\end{split}
\end{align}
and
\begin{equation}
H_{0,0}^{s}((0,\infty))=H_0^{1/2}((0,\infty))\subsetneqq H^{s}((0,\infty)), \quad s\in(1/2,1).  \lb{C.6}
\end{equation}

In this context we note that the case $\alpha = \pi$ in \eqref{C.3} together with \eqref{3.93} yields
\begin{equation}
\dom\big(B_{\pi}^{s/2}\big)=H^s_{0,0}((0,\infty)) = \dom\big(T_{\gamma}^{s/2}\big), \quad \gamma \in (0,\infty), \; s\in(0,1].
\end{equation}

\medskip

\noindent
{\bf Acknowledgments.} We are indebted to John Lee, Marius Mitrea, and Richard Wellman for helpful discussions.



\begin{thebibliography}{99}
%
\bi{AS72} M.\ Abramowitz and I.\ A.\ Stegun, {\it Handbook of Mathematical Functions}, 9th printing, Dover, New York, 1972.
%
\bi{AF03} R.\ A.\ Adams and J.\ J.\ F.\ Fournier, {\it Sobolev Spaces}, 2nd ed., Pure Appl. Math., Vol.~140, Academic Press, Elsevier, Oxford, 2003.   
%
\bi{AS80} A.\ Alonso, B.\ Simon, {\it The Birman-Kre\u\i n-Vishik theory of self-adjoint extensions of semibounded operators}, J.\ Operator Theory {\bf 4}, 251--270 (1980).
%
\bi{Am20} V.\ Ambrosio, {\it On some convergence results for fractional periodic Sobolev spaces}, Opuscula Math. {\bf 40}, 5--20 (2020).
%
\bi{Be73} R.\ Beals, {\it Advanced Mathematical Analysis}, Graduate Texts in Math., Vol.~12, Springer, New York, 1973.
%
\bi{BHS20} J.\ Behrndt, S.\ Hassi, and H.\ De Snoo, {\it Boundary Value Problems, Weyl Functions, and Differential Operators}, Monographs in Math., Vol.~108, Birkh\"{a}user, Springer, 2020.
%
\bi{BBW20} C.\ Bennewitz, M.\ Brown, and R.\ Weikard, {\it Spectral and Scattering Theory for Ordinary Differential Equations. Vol.\ I: Sturm--Liouville Equations}, Springer, 2020.
%
\bi{Be68} Ju.\ Berezanskii, {\it Expansions in Eigenfunctions of Selfadjoint Operators}, Transl. Math. Monographs, Vol.\ 17, Amer. Math. Soc., Providence, RI, 1968.
%
\bi{BSU96} Y.\ M.\ Berezansky, Z.\ G.\ Sheftel, and G.\ F.\ Us, {\it Functional Analysis, Vol.\ II}, Operator Th.: Adv. Appls., Vol.\ 86, Birkh\"auser, Basel, 1996.
%
\bi{BT03} B.\ Bongioanni and J.\ L.\ Torrea, {\it Sobolev spaces associated to the harmonic oscillator}, Proc. Indian Acad. Sci. (Math. Sci.) {\bf 116}, 337--360 (2003).  
%
\bi{C-WHM15} S.\ N.\ Chandler-Wilde, D.\ P.\ Hewett, and A.\ Moiola, {\it Interpolation of Hilbert and Sobolev spaces: Quantitative estimates and counterexamples}, Mathematika {\bf 61}, 414--443 (2015).  
%
\bi{Ch73} P.\ R.\ Chernoff, {\it Essential self-adjointness of powers of generators of hyperbolic equations}, J. Funct. Anal. {\bf 12}, 401--414 (1973). 
%
\bi{Co72} H.\ O.\ Cordes, {\it Self-adjointness of powers of elliptic operators on non-compact manifolds}, Math. Ann. {\bf 195}, 257--272 (1972). 
%
\bi{CS02} F.\ Cucker and S.\ Smale, {\it On the mathematical foundations of learning}, Bull. Amer Math. Soc. (N.S.) {\bf 39}, 1--49 (2002). 
%
\bi{CE19} M.\ Cwikel and A.\ Einav, {\it Interpolation of weighted Sobolev spaces}, J. Funct. Anal. {\bf 277}, 2381--2441 (2019). 
%
\bi{CFKS87} H.\ L.\ Cycon, R.\ G.\ Froese, W.\ Kirsch, and B.\ Simon, {\it Schr\"odinger Operators with Applications to Quantum Mechanics and Global Geometry}, Texts Monogr. Phys., Springer, Berlin, 1987.
%
\bi{De98} V.\ Derkach, {\it Extensions of Laguerre operators in indefinite inner product spaces}, Math. Notes {\bf 63}, 449--459 (1998).
%
\bi{DHS11} V.\ Dom{\' i}nguez, N.\ Heuer, and F.-J.\ Sayas, {\it Hilbert scales and Sobolev spaces defined by associated Legendre functions}, J. Comput. Appl. Math. {\bf 235}, 3481--3501 (2011).
%
\bibitem{DS23} R.\ Duarte and J.\ Drumond Silva, {\it Weighted Gagliardo-Nirenberg interpolation
inequalities}, Journal of Functional Analysis, {\bf 285}, 110009 (2023).
%
\bi{EE23} D.\ E.\ Edmunds and W.\ D.\ Evans, {\it Fractional Sobolev Spaces and Inequalities}, Cambridge Tracts in Mat., Vol.~230, Cambridge Univ. Press, Cambridge, 2023. 
%
\bi {Ev80} W.\ N.\ Everitt, {\it Legendre polynomials and singular differential operators}, in {\it Ordinary and Partial Differential Equations. Proceedings, Dundee, Scotland 1978}, Lecture Notes  Math., Vol.\ 827, Springer, Berlin, 1980, pp.~83--106.
%
\bi{EK07} W.\ N.\ Everitt and H.\ Kalf, {\it The Bessel differential equation and the Hankel transform}, J. Comput. Appl. Math. {\bf 208}, 3--19 (2007).
%
\bi{EKLWY07} W.\ N.\ Everitt, K.\ H.\ Kwon, L.\ L.\ Littlejohn, R.\ Wellman, and G.\ J.\ Yoon, {\it
Jacobi--Stirling numbers, Jacobi polynomials, and the left-definite analysis of the classical Jacobi differential expression}, J. Comput. Appl. Math. {\bf 208}, 29--56 (2007).
%
\bi{ELM02} W.\ N.\ Everitt, L.\ L.\ Littlejohn, and V. Mari{\' c}, {\it On properties of the Legendre differential expression}, Results Math. {\bf 42}, 42--68 (2002).
%
\bi {ELW00} W.\ N Everitt, L.\ L Littlejohn, and R.\ Wellman, {\it The left-definite spectral theory for the classical Hermite differential equation}, J. Comput. Appl. Math. {\bf 121}, 313--330 (2000).
%
\bi{ELW02} W.\ N.\ Everitt, L.\ L.\ Littlejohn, and R.\ Wellman, {\it Legendre polynomials,
Legendre--Stirling numbers, and the left-definite spectral analysis of the Legendre differential
expression}, J. Comput. Appl. Math. {\bf 148}, 213--238 (2002).
%
\bi {ELW04} W.\ N Everitt, L.\ L.\ Littlejohn, and R.\ Wellman, {\it The Sobolev orthogonality and spectral analysis of the Laguerre polynomials $\{L_{n}^{-k}\}$ for positive integers $k$}, J. Comput. Appl. Math. {\bf 171}, 199--234 (2004).
%
\bi{FM23} R.\ Frank, K.\ Merz, {\it On Sobolev norms involving Hardy operators in a half-space}, J.  Funct. Anal. {\bf 285}, 110104 (2023).
%
\bi{FL21} D.\ Frymark and C.\ Liaw, {\it Perspectives on general left-definite theory}, in {\it From Operator Theory to Orthogonal Polynomials, Combinatorics, and Number Theory. A Volume in Honor of Lance Littlejohn's 70th Birthday}, F.\ Gesztesy and A.\ Martinez-Finkelshtein (eds.), Operator Th.: Adv. Appls., Vol.~285, Birkh\"auser, Cham, 2021, pp.~69--89.
%
\bi{Fu67} D.\ Fujiwara, {\it Concrete characterization of the domains of fractional powers of some elliptic differential operators of the second order}, Proc. Japan Acad. {\bf 43}, 82--86 (1967).
%
\bi{GNZ24} F.\ Gesztesy, R.\ Nichols, and M.\ Zinchenko, {\it Sturm--Liouville Operators, Their Spectral Theory, and Some Applications}, Colloquium Publications, Amer. Math. Soc., Providence, RI, 939 pages, to appear.
%
\bi{GPS21}  F.\ Gesztesy, M.\ M.\ H.\ Pang, and J.\ Stanfill, {\it On domain properties of Bessel-type operators}, Discrete Cont. Dyn. Syst. Ser.~S {\bf 17}, 1911--1946 (2024). 
%
\bi{GP79} F.\ Gesztesy and L.\ Pittner, {\it On the Friedrichs extension of	ordinary differential operators with strongly singular potentials}, Acta Phys. Austriaca {\bf 51}, 259--268 (1979).
%
\bi{GJ69} J.\ Glimm and A.\ Jaffee, {\it Singular perturbations of selfadjoint operators}, Commun. Pure Appl. Math. {\bf 22}, 401--414 (1969).  
%
\bi{HT08} D.\ D.\ Haroske and H.\ Triebel, {\it Distributions, Sobolev Spaces, Elliptic Equations}, EMS Textbook in Math., European Math. Soc., Z\"urich, 2008. 
%
\bi{II01} R.\ J.\ Iorio and V.\ de Margalh\~{a}es Iorio, {\it Fourier Analysis and Partial Differential Equations}, Cambridge Studies in Adv. Math., Vol.\ 70, Cambridge Univ. Press, Cambridge, 2001.
%
\bi{Ka78} H.\ Kalf, {\it A characterization of the Friedrichs extension of Sturm--Liouville operators}, J. London Math. Soc. (2) {\bf 17}, 511--521 (1978).
%
\bi{Ka61} T.\ Kato, {\it Fractional powers of dissipative operators}, J. Math. Soc. Japan {\bf 13}, 246--274 (1961).  
%
\bi{KMVZZ18} R.\ Killip, C.\ Miao, M.\ Visan, J.\ Zhang, J.\ Zheng, {\it Sobolev spaces adapted to the Schr\"odinger operator with inverse-square potential}, Math.\ Z.\ {\bf 288}, 1273--1298 (2018).
%
\bi {KL95} K.\ H.\ Kwon and L.\ L.\ Littlejohn, {\it The orthogonality of the Laguerre polynomials $\{L_{n}^{-k}(x)\}$ for positive integers $k$}, Ann. Numer. Math. {\bf 2}, 289--303 (1995).
%
\bi{Le23} G.\ Leoni, {\it A First Course in Fractional Sobolev Spaces}, Graduate Studies in Mathematics, Vol.\ 229, Amer. Math. Soc., Providence, RI, 2023.
%
\bi{Li86} C.\ S.\ Lin, {\it Interpolation inequalities with weights}, Commun.\ Partial Diff.\ Equ. {\bf 11}, 1515--1538 (1986).
%
\bi{Li62} J.\ L.\ Lions, {\it Espaces d'interpolation et domaines de puissances fractionnaires 
d'op{\' e}rateurs}, J. Math. Soc. Japan {\bf 14}, 233--241 (1962).
%
\bi{LM72} J.\ L.\ Lions and E.\ Magenes, {\it Non-Homogeneous Boundary Value Problems and Applications, I}, Grundlehren, Vol.\ 181, Springer, New York, 1972.
%
\bi{LW02} L.\ L.\ Littlejohn and R.\ Wellman, {\it A general left-definite theory for certain self-adjoint operators with applications to differential equations}, J. Diff. Eq. {\bf 181}, 280--339 (2002).
%
\bi{LW13} L.\ L.\ Littlejohn and R.\ Wellman, {\it On the spectra of left-definite operators}, Complex Anal. Oper. Theory {\bf 7}, 437--455 (2013).
%
\bi{LW16}  L.\ L.\ Littlejohn and Q.\ Wicks, {\it Glazman--Krein--Naimark theory, left-definite theory and the square of the Legendre polynomials differential operator}, J. Math. Anal. Appl. {\bf 444}, 1--24 (2016).
%
\bi{LZ07} L.\ L.\ Littlejohn and A.\ Zettl, {\it Left-definite variations of the classical Fourier expansion theorem}, Electronic Trans. Num. Anal. {\bf 27}, 124--139 (2007).
%
\bi{Lu18} A.\ Lunardi, {\it Interpolation Theory}, 3rd edition, Lecture Notes, Vol.\ 16, Edizioni della Normale, Scuola Normale Superiore, Pisa, 2018.
%
\bi{Mc00} W.\ McLean, {\it Strongly Elliptic Systems and Boundary Integral Equations}, Cambridge University Press, Cambridge, 2000.
%
\bi{RS80} M.\ Reed and B.\ Simon, {\it Methods of Modern Mathematical Physics. I: Functional Analysis}, rev.\ and enl.\ ed., Academic Press, New York, 1980.
%
\bi{Ro85} R.\ Rosenberger, {\it A new characterization of the Friedrichs extension of semibounded Sturm--Liouville operators}, J. London Math. Soc. (2) {\bf 31}, 501--510 (1985).
%
\bi{Sc12} K.\ Schm\"udgen, {\it Unbounded Self-Adjoint Operators on Hilbert Space}, Graduate Texts Math., Vol.\ 265, Springer, New York, 2012.
%
\bi{Si70} B.\ Simon, {\it Coupling constant analyticity for the anharmonic oscillator}, Ann. Phys. {\bf 58}, 76--136 (1970). 
%
\bi{Si15} B.\ Simon, {\it Operator Theory, A Comprehensive Course in Analysis}, Part 4, Amer. Math. Soc., Providence, RI, 2015.
%
\bi {SZ03} S.\ Smale and D.-Z.\ Zhou, {\it Estimating the approximation error in learning theory}, Anal. Appl. (Singap.) {\bf 1}  no.~1, 17--41 (2003).
%
\bi {SZ07} S.\ Smale and D.-Z.\ Zhou, {\it Learning theory estimates via integral operators and their approximations}, Constr. Approx. {\bf 26}, 153--172 (2007). 
%
\bi{Sz75} G.\ Szeg\H{o}, {\it Orthogonal Polynomials}, 4th Edition, Colloquium Publs., Vol.\ 23, Amer. Math. Soc., Providence, RI, 1975.
%
\bi{Tr68} H.\ Triebel, {\it Erzeugung des nuklearen lokalkonvexen Raumes 
$C^{\infty}\big({\ol \Omega}\big)$ durch einen elliptischen Differentialoperator zweiter Ordnung}, Math. Ann. {\bf 177}, 247--264 (1968).
%
\bi{Tr69} H.\ Triebel, {\it Singul\"are elliptische Differentialgleichungen und Interpolationss\"atze for Sobolev-Slobodeckij-R\"aume mit Gewichtsfunktionen}, Arch. Rat. Mech. Anal. {\bf 32}, 113--134, 1969.
%
\bi{Tr70} H.\ Triebel, {\it Allgemeine Legendresche Differentialoperatoren I: Selbstadjungiertheit, Defetindex, Definitionsgebiete ganzer Potentenzen, Erzeugung der lokalkonvexen R\"aume 
$C^{\infty}_{s,t}[a,b]$}, J. Funct. Anal. {\bf 6}, 1--25 (1970).
%
\bi{Tr70a} H.\ Triebel, {\it Allgemeine Legendresche Differentialoperatoren II}, Ann. Scuola Normale Superiore Pisa, 3rd Ser., {\bf 24}, 1--35 (1970). 
%
\bi{Tr78} H.\ Triebel, {\it Interpolation Theory, Function Spaces, Differential Operators}, North-Holland Math. Library, Vol.~18, North-Holland Publ., Amsterdam, 1978. 
%
\bi{Vo00} R.\ Vonhoff, {\it A left-definite study of Legendre's differential equations and of the fourth-order Legendre type differential equation}, Results. Math. {\bf 37}, 155--196 (2000).
%
\bi{We80} J.\ Weidmann, {\it Linear Operators in Hilbert Spaces}, Graduate Texts in Mathematics, Vol.\ 68, Springer, New York, 1980.
%
\bi{Ze05} A.\ Zettl, {\it Sturm--Liouville Theory}, Mathematical Surveys and Monographs, Vol.\ 121, Amer. Math. Soc., Providence, RI, 2005.
%
\end{thebibliography}
\end{document}